\documentclass[10pt]{elsarticle}
\usepackage{amssymb}
\usepackage{amsmath}
\usepackage{amsfonts}
\usepackage{amsbsy}
\usepackage{amscd}
\usepackage{amsthm}
\usepackage{times}
\usepackage{psfrag}
\usepackage{graphics}
\usepackage{color}
\usepackage{colortbl}
\usepackage{graphicx}
\usepackage{caption}
\usepackage{subcaption}
\usepackage{array,supertabular}
\usepackage{tabularx}
\usepackage{booktabs}
\usepackage{multirow}
\usepackage{ulem}
\usepackage{units}
\usepackage{mathabx}
\usepackage{accents}
\usepackage{algorithmicx}
\usepackage{algorithm}
\usepackage{xfrac}
\usepackage{algpseudocode}
\usepackage{rotating}
\usepackage[capitalize]{cleveref}

\crefname{secinapp}{Section}{Sections}
\Crefname{secinapp}{Section}{Sections}

\newlength{\dhatheight}
\newcommand{\doublehat}[1]{%
    \settoheight{\dhatheight}{\ensuremath{\hat{#1}}}%
    \addtolength{\dhatheight}{-0.3ex}%
    \hat{\vphantom{\rule{1pt}{\dhatheight}}%
    \smash{\hat{#1}}}}

\newtheorem{theorem}{Theorem}[section]

\newtheorem{lemma}[theorem]{Lemma}

\def\XXint#1#2#3{{\setbox0=\hbox{$#1{#2#3}{\int}$ }
\vcenter{\hbox{$#2#3$ }}\kern-.57\wd0}}

\newcounter{myalgorithmctr}

\begin{document}

\begin{frontmatter}
\title{Hierarchical analysis-suitable T-splines: Formulation,
  B\'{e}zier extraction, and application as an adaptive basis for
  isogeometric analysis}

\author[byu1]{E. J. Evans\corref{cor1}}
\ead{ejevans@math.byu.edu}

\author[byu2]{M. A. Scott}

\author[china]{X. Li}

\author[byu3]{D. C. Thomas}

\cortext[cor1]{Corresponding author}

\address[byu1]{Department of Mathematics,
  Brigham Young University,
  Provo, Utah 84602, USA}

\address[byu2]{Department of Civil and Environmental Engineering,
  Brigham Young University,
  Provo, Utah 84602, USA}

\address[china]{University of Science and Technology of China,
Hefei, Anhui Province 230026, P. R. China}

\address[byu3]{Department of Physics and Astronomy,
  Brigham Young University,
  Provo, Utah 84602, USA}

\begin{abstract}
In this paper hierarchical analysis-suitable T-splines (HASTS) are
developed. The resulting spaces are a superset of both
analysis-suitable T-splines and hierarchical B-splines. The additional
flexibility provided by the hierarchy of
T-spline spaces results in simple, highly localized refinement
algorithms which can be utilized in a 
design or analysis context. A detailed theoretical formulation is
presented including a proof of local linear independence for analysis-suitable
 T-splines, a requisite theoretical ingredient for HASTS. 
 B\'{e}zier extraction is extended to HASTS simplifying the 
 implementation of HASTS in existing finite element codes. The
behavior of a simple HASTS refinement algorithm is compared to the
local refinement algorithm for analysis-suitable T-splines
demonstrating the superior efficiency and locality of the HASTS
algorithm. Finally, HASTS are utilized as a basis for adaptive isogeometric
analysis.
\end{abstract}


\end{frontmatter}

\section{Introduction}
In this work, a hierarchical extension of analysis-suitable T-splines
is developed and utilized in the context of isogeometric design and
analysis. We call this new spline description hierarchical
analysis-suitable T-splines (HASTS). The class of HASTS is a strict
superset of both analysis-suitable
T-splines~\cite{LiZhSeHuSc10,ScLiSeHu10,BeBuChSa12,BeBuSaVa12,LiScSe12}
and hierarchical
B-splines~\cite{FoBa88,SchDeScEvBoRaHu12,VuGiJuSi11,ScThEv13,GiJu13,GiJuSp12}. 

T-splines, introduced in the CAD community~\cite{SeZhBaNa03}, are a generalization of non-uniform
rational B-splines (NURBS) which address fundamental
limitations in NURBS-based design. For example, a T-spline can model a complicated design as a single,
watertight geometry and are also locally
refinable~\cite{SeCaFiNoZhLy04, ScLiSeHu10}. Since their advent they
have emerged as an important technology across multiple disciplines
and can be found in several major commercial CAD products~\cite{TSManual12,Autodesk360}.

Isogeometric 
analysis was introduced in~\cite{HuCoBa04} and described in detail
in~\cite{Cottrell:2009rp}. The isogeometric paradigm is simple: use the smooth spline basis
that defines the geometry as the basis for analysis. As a result, exact
geometry is introduced into the analysis, the 
smooth basis can be leveraged by the
analysis~\cite{EvBaBaHu09,HuEvRe13,CoHuRe07}, and new innovative approaches to 
model design~\cite{WaZhScHu11,LiZhHuScSe14},
analysis~\cite{SchDeScEvBoRaHu12,ScSiEvLiBoHuSe12,ScWuBl12,BeBaDeHsScHuBe09},
optimization~\cite{Wall08}, and
adaptivity~\cite{Bazilevs2009,DoJuSi09,ScThEv13, ScThEv13} are
made possible. The use of T-splines as a basis for isogeometric analysis (IGA) has gained
widespread attention across a number of
application areas~\cite{Bazilevs2009,ScBoHu10,ScLiSeHu10,Verhoosel:2010vn,
 Verhoosel:2010ly,BoScLaHuVe11,BeBaDeHsScHuBe09,SchDeScEvBoRaHu12,
 ScSiEvLiBoHuSe12,SiScTaThLi14,DiLoScWrTaZa13,HoReVeBo14,BaHsSc12,BuSaVa12,GiKoPoKaBeGeScHu14}. Particular focus has been placed on the use of T-spline local refinement in an
analysis
context~\cite{ScLiSeHu10,BoScLaHuVe11,Verhoosel:2010ly,Verhoosel:2010vn}.

In the context of CAD, where a designer interacts directly with the
geometry, T-spline local refinement is most useful if 
confined to a \textit{single level}. In other words, all local refinement is
done on one control mesh and all
control points have similar influence on the shape of the surface. In
this way, the geometric behavior of the surface is easily controlled
through the manipulation of control points before and after
refinement. In the context of analysis, however, where not all control
points need to have a geometric interpretation, the single level
restriction can be relaxed. This \textit{hierarchical} point of view has
important advantages:
\begin{enumerate} 
\item Hierarchical local refinement remains completely
  local. Single level T-spline local refinement always entails a degree of
  nonlocal control point propagation~\cite{ScLiSeHu10}.
\item Hierarchical local coarsening is achieved by
simply removing higher levels of refinement where
needed~\cite{ScThEv13}. Local coarsening operations for single level
T-spline descriptions are possible but their algorithmic complexity
remains uncertain~\cite{SeCaFiNoZhLy04}.
\item Hierarchical refinement and coarsening operations use a fixed control
mesh which simplifies algorithmic developments,
especially for parallel computations. Single level
local refinement requires expensive mesh manipulation and
modification operations.
\item Hierarchies of finite-dimensional subspaces are the natural
  setting for many optimized iterative solvers
and preconditioning techniques for large-scale linear systems.
\end{enumerate}
Initial investigations employing hierarchical B-spline refinement in
the context of IGA have demonstrated the promise of the hierarchical
approach~\cite{KuVeZeBr14,SchDeScEvBoRaHu12,SchEvReScHu13,GrKrSch02,ScRa11}.

HASTS inherit the design strengths of T-splines without the single
level restriction. In this way, a complex T-spline design can be
encapsulated in the first level of the hierarchy while higher
levels can be leveraged to develop
adaptive multiresolution schemes which are smooth, highly localized,
geometrically exact, and appropriate for the analysis task at hand. We
feel that this provides the appropriate
mathematical foundation for the development of integrated isogeometric
design and analysis methodologies for demanding applications in
science and engineering. 
Note that, in this paper, we restrict our theoretical
developments to HASTS defined over four-sided domains. However,
extending HASTS to domains of arbitrary topological genus should be
straightforward in the context of the recently introduced spline
forest~\cite{ScThEv13}.

We note that in addition to T-splines, hierarchical B-splines, and
NURBS a number of alternative spline
technologies have been proposed as a basis for IGA with varying
strengths and weaknesses. Truncated hierarchical B-splines 
(THB-splines)~\cite{GiJuSp12, KiGiJu14,GiJuSp13,BeKiBrChMoOhKi14} are a modification of
hierarchical B-splines~\cite{FoBa88,VuGiJuSi11,SchEvReScHu13} which
possess a partition of unity and enhanced numerical
conditioning. B-spline forests~\cite{ScThEv13} are a generalization of
hierarchical B-splines to surfaces and volumes of arbitrary
topological genus. Polynomial splines over hierarchical T-meshes
(PHT-splines)~\cite{htspline2, htspline3, 
  htspline4, htspline5}, modified T-splines~\cite{KaChDe13}, and locally refined splines
(LR-splines)~\cite{DoLyPe13, Br13} are closely related to T-splines
with varying levels of smoothness and approaches to local refinement. Generalized
B-splines~\cite{MaPeSa11,CoMaPeSa10} and T-splines~\cite{BrBeChOhKi14}
enhance a piecewise polynomial spline basis by including
non-polynomial functions, typically
trigonometric or hyperbolic functions. Generalized splines permit the exact representation of
conic sections without resorting to rational functions.
Generalized splines can also be used to represent solution features
with known non-polynomial characteristics exactly in
certain circumstances.

\subsection{Structure and content of the paper}
\label{sec:content}
In~\cref{sec:tmesh} the T-mesh is described and appropriate notational
conventions are introduced. Analysis-suitable T-splines are then described
in~\cref{sec:tspsp}. The local linear independence of 
analysis-suitable T-splines is established
in~\cref{sec:lli}. Hierarchical analysis-suitable T-splines are then defined
in~\cref{sec:hts}.  In preparation for their use in design and
analysis a B\'{e}zier extraction framework is introduced
in~\cref{sec:extract}. HASTS are then utilized as a basis for
isogeometric analysis in~\cref{sec:iga}.
In ~\cref{sec:conclusion} we draw conclusions. We note that the paper
has been written so the proof of local linear independence in
Section~\ref{sec:lli} is self-contained and can be skipped if the
reader is not interested in the detailed theory of analysis-suitable
T-splines.

\section{The T-mesh}
\label{sec:tmesh}
The T-mesh is used to define the topological structure of the
associated T-spline space. In other words, the T-mesh defines the
basis functions and their relationship to one another. We closely
follow the notational conventions introduced
in~\cite{LiScSe12,BeBuSaVa12,BeBuChSa12}.

A T-mesh $\mathsf{T}$ in two dimensions is a rectangular partition of
$\doublehat{\Omega} = [1,m] \times [1,n]$ such that all vertices $V = \{i,j\}
\in \mathsf{V}$ have integer coordinates. All cells $C \in \mathsf{C}$
are rectangular, non-overlapping, and open. An edge is a horizontal or vertical line
segment between vertices which does not intersect any cell. The
valence of a vertex $V \in \mathsf{V}$ is the number of edges
coincident to that vertex. Since all cells are assumed rectangular,
only valence three (i.e., T-junction) or four is allowed for all vertices $V \subset
(1,m) \times (1,n)$. The sets of horizontal and vertical coordinates in the
T-mesh are denoted by $h\mathsf{I}=\{1,2,\ldots,m\}$ and
$v\mathsf{I}=\{1,2,\ldots,n\}$. The
horizontal and vertical skeletons, $h\mathsf{S}$ and 
$v\mathsf{S}$, of a T-mesh are the union of all horizontal and
vertical edges, respectively, and associated vertices.
The entire skeleton is denoted by $\mathsf{S} = h\mathsf{S} \cup
v\mathsf{S}$. 

We split $\doublehat{\Omega}$ into an
active region $\mathsf{AR}$ and a 
frame region $\mathsf{FR}$ such that $\doublehat{\Omega} = \mathsf{FR} \cup
\mathsf{AR}$ and $\mathsf{AR} =  [1 + \lfloor (p + 1)/2
\rfloor, m -\lfloor (p + 1)/2
\rfloor] \times [1+\lfloor (q + 1)/2
\rfloor, n-\lfloor (q + 1)/2
\rfloor]$, and 
$\mathsf{FR} = \overline{\doublehat{\Omega}\setminus
  \mathsf{AR}}$ where $p$ and $q$ are polynomial degrees. Note that
\textit{both} $\mathsf{FR}$ and $\mathsf{AR}$ are closed. Further, all
T-meshes considered in this work are admissible as described
in~\cite{LiScSe12}, a mild restriction always adopted in practice.
The notation $\mathsf{T}^1 \subseteq \mathsf{T}^2$ will indicate
that $\mathsf{T}^2$ can be created by adding vertices and edges to
$\mathsf{T}^1$. \cref{fig:tmesh} shows an example T-mesh.
\begin{figure}[htb]
  \centering
  {\includegraphics[scale=0.6]{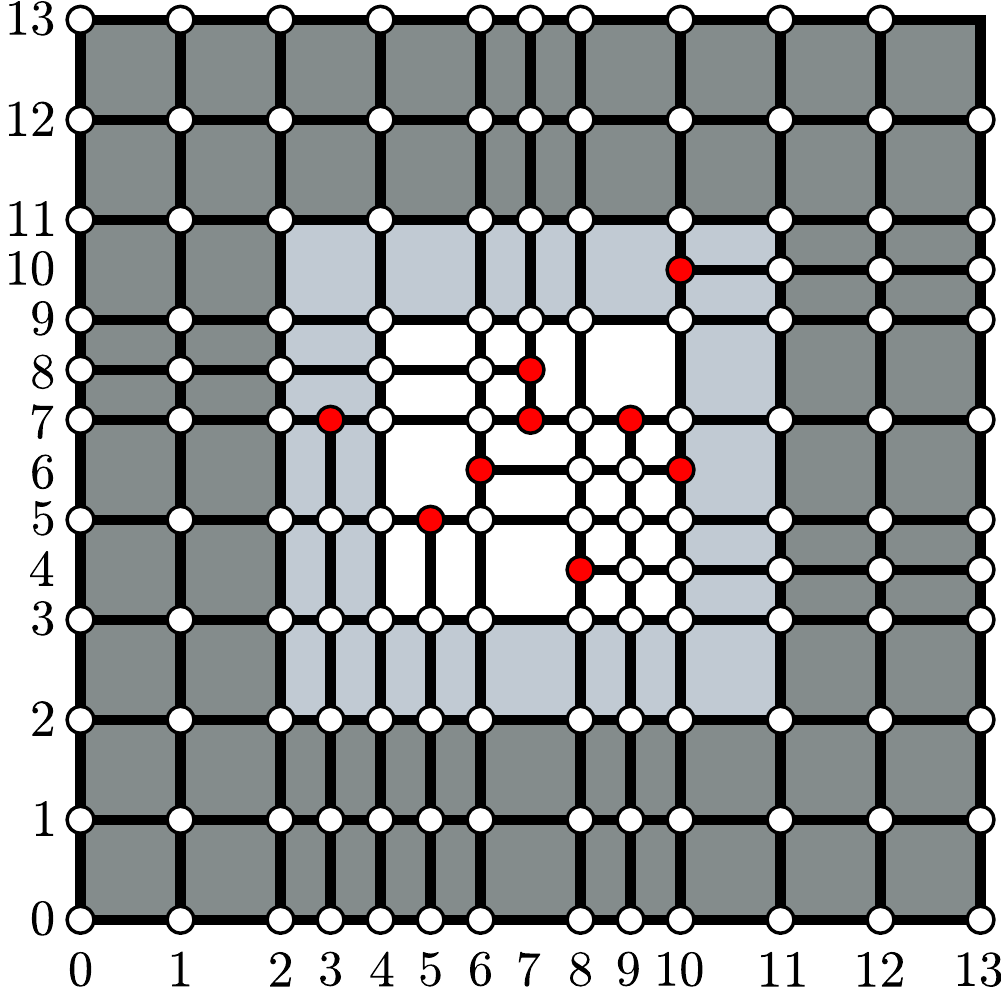}}
  \caption{A bicubic T-mesh. The frame region ($\mathsf{F}$R) is dark grey and
    the active region ($\mathsf{AR})$ is the union of the light grey
    and white regions. Note that the region with zero parametric area
    (see Section~\ref{sec:ts_basis} for a description of the
    parametric space of a T-spline) is the union of the dark and light grey
    regions.} 
\label{fig:tmesh} 
\end{figure}

\subsection{Analysis-suitable T-meshes}
\label{sec:asts}
Analysis-suitable T-splines (ASTS) were introduced
in~\cite{LiZhSeHuSc10}.  
The analysis-suitability of a T-spline is dictated by the structure
of the underlying T-mesh. We define face and edge extensions 
to be closed line segments that originate at T-junctions. For example,
to define a horizontal face extension we trace out a horizontal line
by moving in the direction of the missing  
edge until $\lfloor (p + 1)/2 \rfloor$ vertical edges or 
vertices are intersected.  To define an edge extension we trace out a
horizontal line by moving in 
the direction opposite the face extension until 
$\lceil (p - 1)/2 \rceil$ vertical edges or vertices are
intersected. A T-junction extension includes both the face and edge extensions.
Since extensions are defined as closed line segments they may 
intersect at their end points. An \textit{extended T-mesh}, $\mathsf{T}_{ext}$,
is the T-mesh formed by adding the T-junction extensions to
$\mathsf{T}$.  The collection of rectangular cells in
$\mathsf{T}_{ext}$ is denoted by $\mathsf{C}_{ext}$.
We say a T-mesh is \textit{analysis-suitable} if no horizontal T-junction 
extension intersects a vertical T-junction extension.
Face and edge extensions (along with analysis-suitability) are 
illustrated in~\Cref{fig:astmesh}.

\begin{figure}
\begin{center}
\begin{tabularx}{1\textwidth}{XX}
\begin{center}
\includegraphics[scale=0.5]{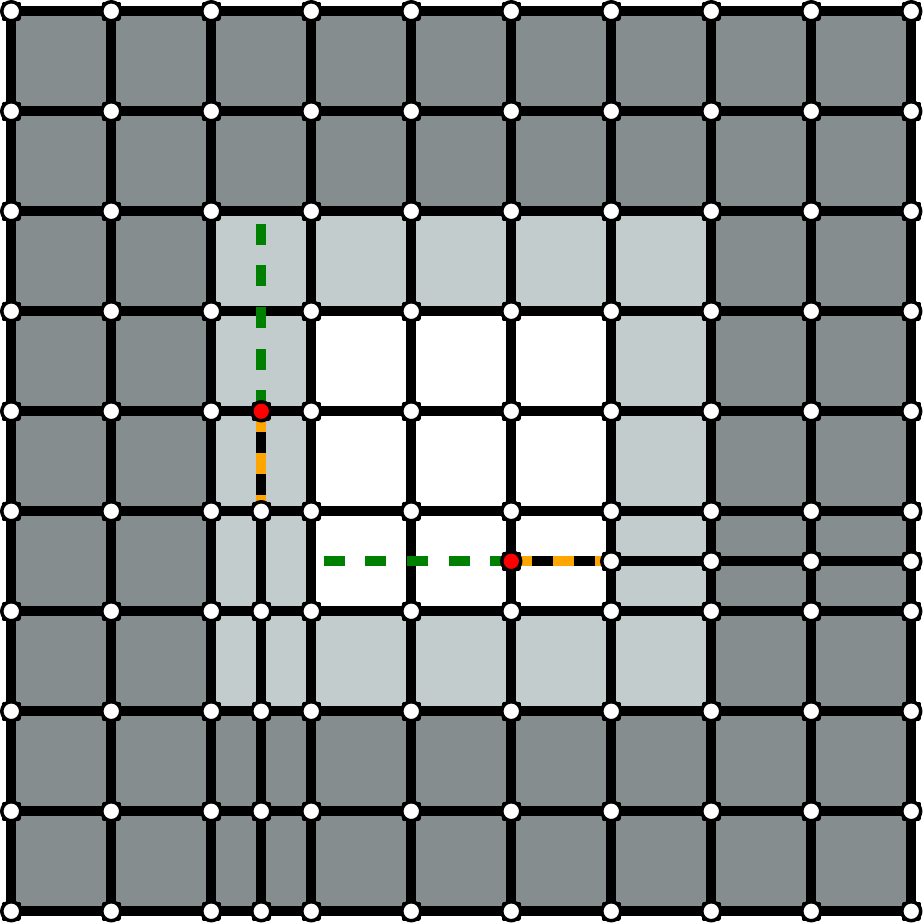}
\end{center} 
&
\begin{center}
\includegraphics[scale=0.5]{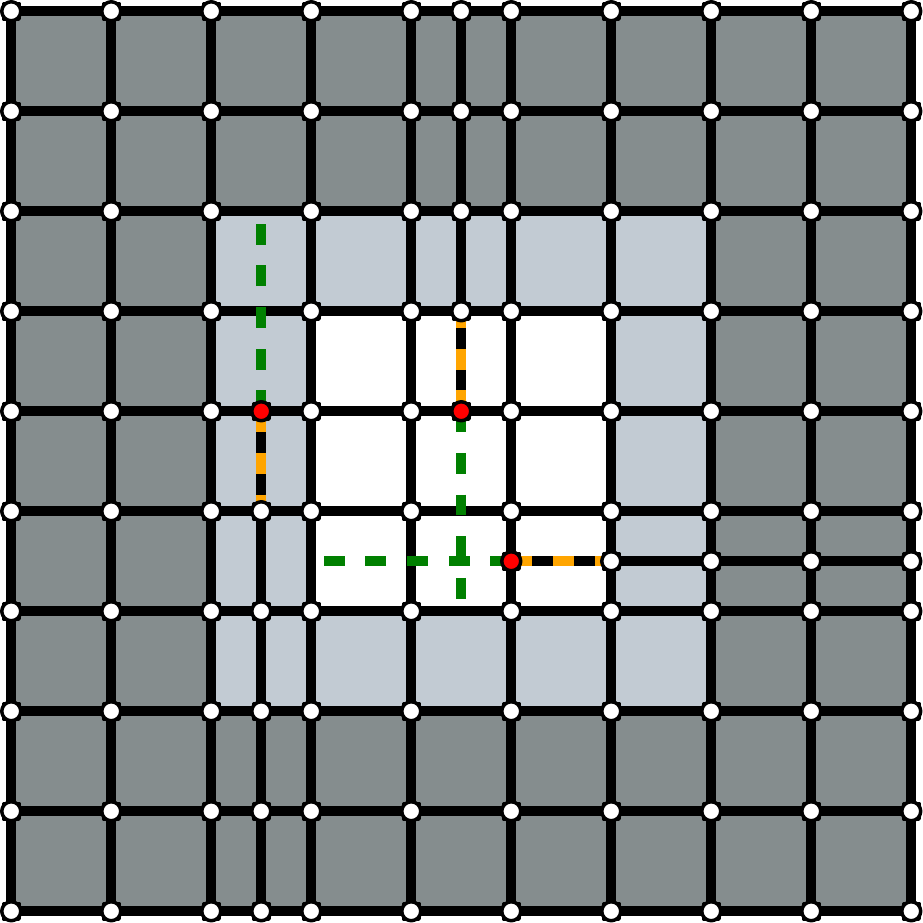}
\end{center}
\end{tabularx}
\caption{T-junction extension in two dimensions for a bicubic
  T-spline.  Face extensions are shown in green and edge extensions
  are shown in orange.  The T-mesh on the left is analysis-suitable
  whereas the T-mesh on the right is not due to intersecting face
  extensions.} 
\label{fig:astmesh}
\end{center}
\end{figure}

\subsection{Anchors}
\label{sec:mbasisdef}
Anchors are used in the construction of T-spline blending
functions. For an analysis-suitable T-mesh the anchors are
located only in the active region and are
defined as follows:
\begin{itemize}
\item if $p$ and $q$ are odd the anchors are vertices. It is written
  as $\{i\} \times \{j\}$ or equivalently $\{i,j\}$.
\item if $p$ is even and $q$ is odd the anchors are horizontal
  edges. It is written as $(i_1, i_2) \times \{j\}$.
\item if $p$ is odd and $q$ is even the anchors are vertical edges. It
  is written as $\{i\} \times (j_1, j_2)$.
\item if $p$ and $q$ are even the anchors are cells. It
  is written as $(i_1, i_2)\times (j_1, j_2)$.
\end{itemize}
\begin{figure}
\begin{center}
\begin{tabularx}{0.9\textwidth}{X}
{\begin{tabularx}{0.9\textwidth}{XX}
\begin{center}
\includegraphics[scale=0.45]{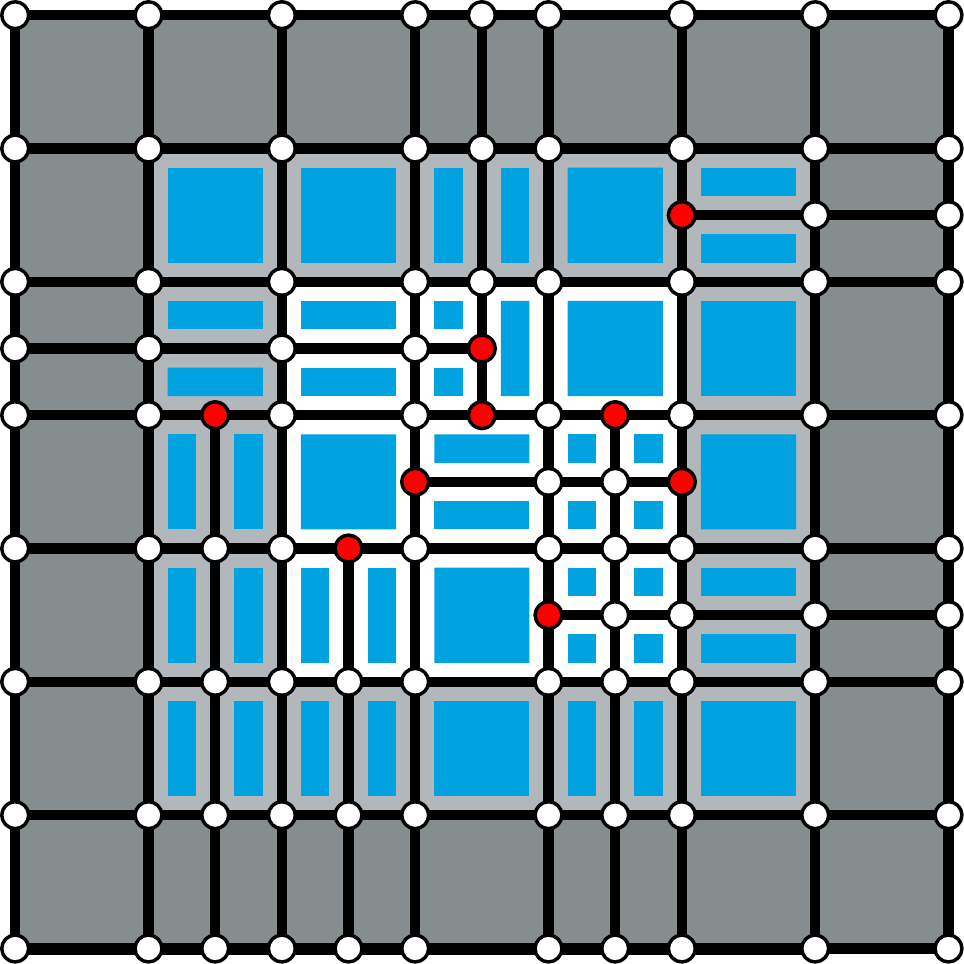}
\end{center} 
&
\begin{center}
\includegraphics[scale=0.45]{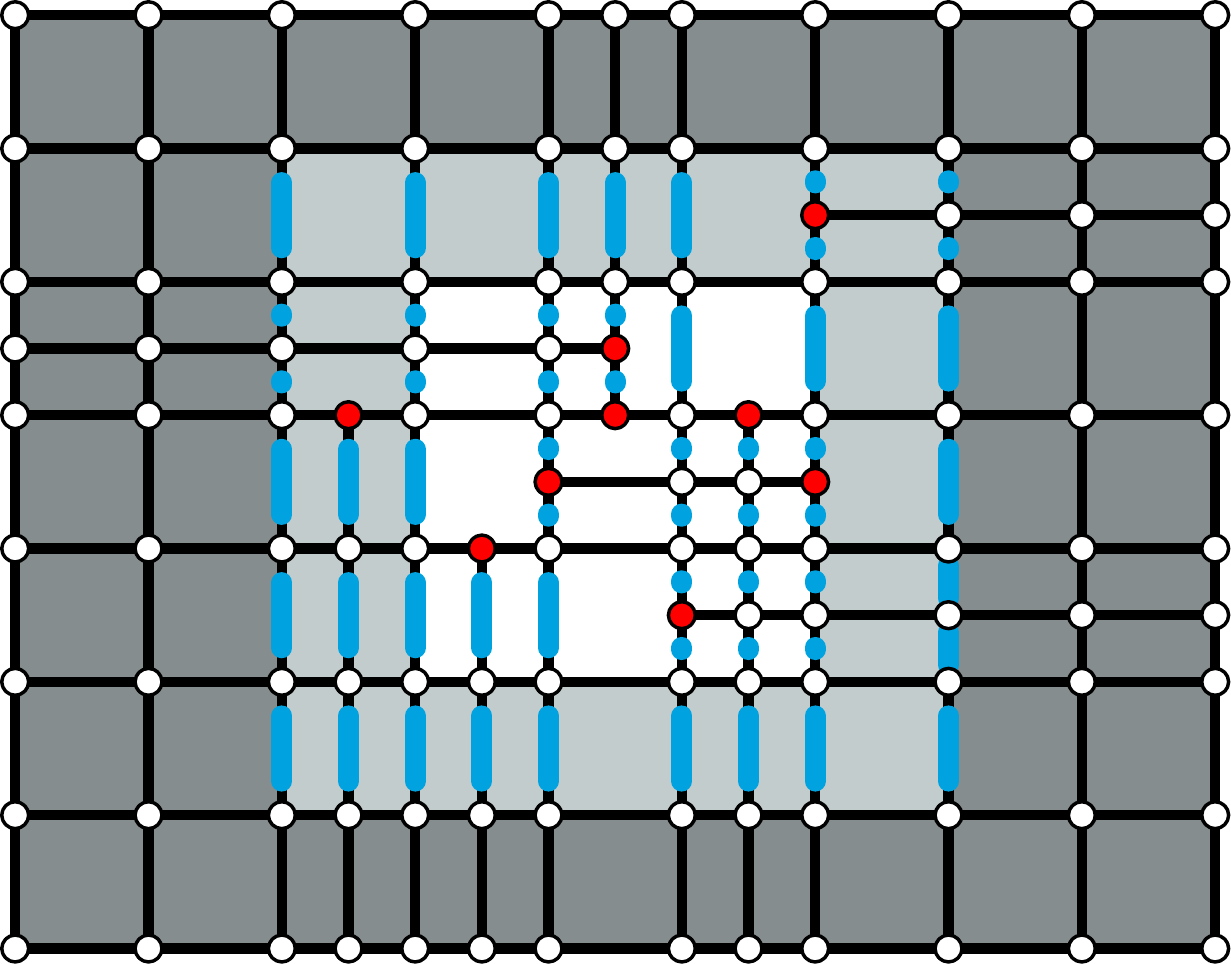}
\end{center}
\\
\vspace{-15pt}
\begin{center}
(a) $p =2,\;q=2$, anchors are the faces denoted by blue squares.
\end{center}
&
\vspace{-15pt}
\begin{center}
(b) $p=3,\;q =2$, anchors are the vertical edges denoted by bold
blue vertical lines.
\end{center}
\\
\begin{center}
\includegraphics[scale=0.6]{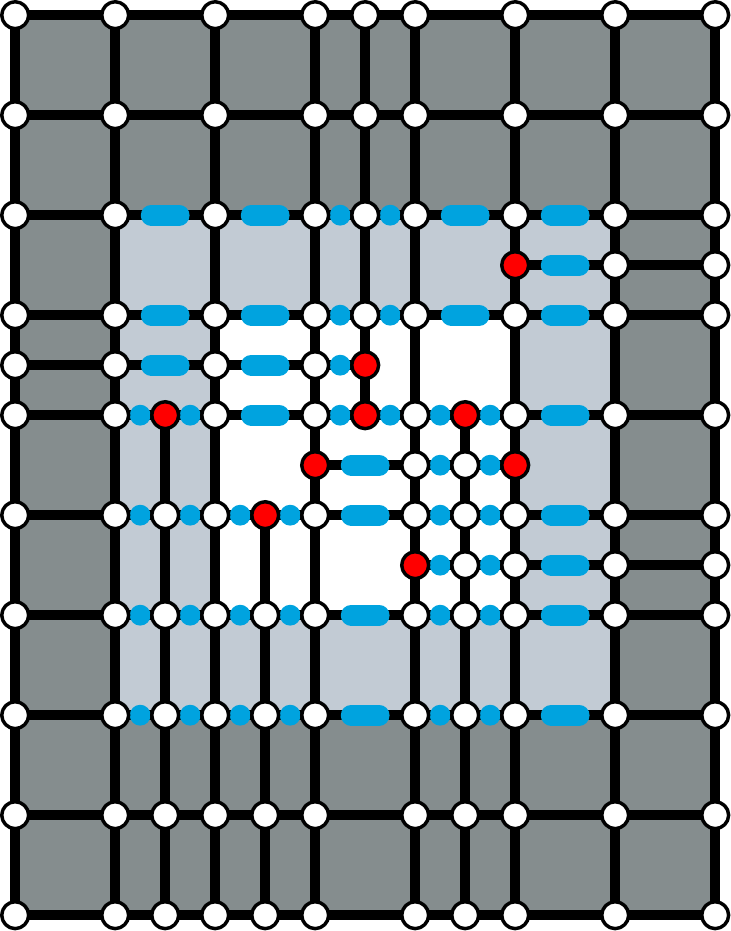}
\end{center} 
&
\begin{center}
\includegraphics[scale=0.6]{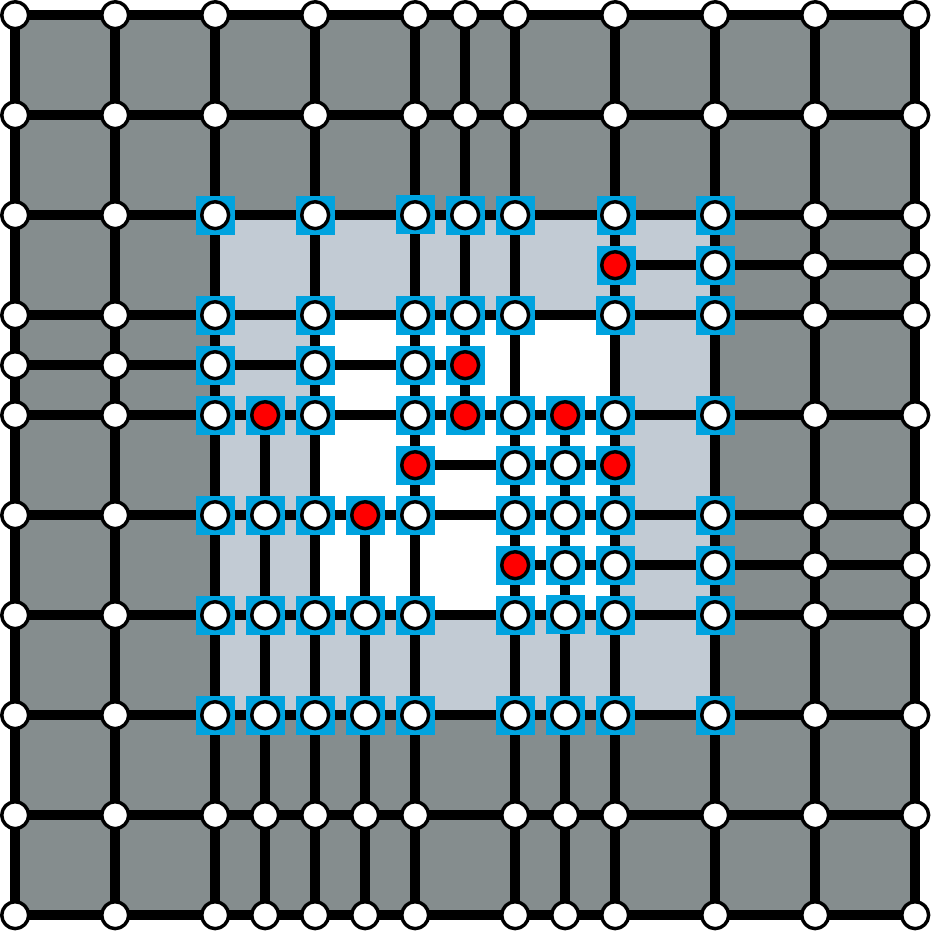}
\end{center}
\\
\vspace{-15pt}
\begin{center}
(c) $p=2,\;q=3$, anchors are the horizontal edges denoted by bold
blue horizontal lines.
\end{center}
&
\vspace{-15pt}
\begin{center}
(d) $p=3,\;q=3$, anchors are the vertices denoted by small blue
hollow squares.
\end{center}
\end{tabularx} }
\end{tabularx}
\caption{The set of anchors for varying values of $p$ and $q$.  In
  this picture, the blue regions denote anchor locations. Note that no
  anchors reside inside the frame region (dark grey region).}
\label{fig:anchorloc}
\end{center}
\end{figure}
The set of all anchors is denoted by $\mathsf{A}$. The set
of anchors for varying values of $p$ and $q$ are shown
in~\Cref{fig:anchorloc}.

\section{Analysis-suitable T-splines}
\label{sec:tspsp}
Analysis-suitable T-splines form a useful subset of T-splines.
ASTS maintain the important mathematical
properties of the NURBS basis while providing an efficient and highly
localized refinement capability. Several important properties of
ASTS have been proven:
\begin{itemize}
\item The blending functions are linearly
  independent for \textit{any} choice of knots~\cite{LiZhSeHuSc10}.
\item The basis constitutes a partition of unity~\cite{LiScSe12}.
\item Each basis function is non-negative.
\item They can be generalized to arbitrary degree~\cite{BeBuSaVa12}.
\item An affine transformation of an analysis-suitable T-spline is
  obtained by applying the transformation to the control points. We
  refer to this as affine covariance. This implies that all ``patch
  tests'' (see~\cite{Hug00}) are satisfied \textit{a priori}.
\item They obey the convex hull property.
\item They can be locally refined~\cite{SeZhBaNa03, ScLiSeHu10,LiScSe12}.
\item A dual basis can be constructed~\cite{BeBuChSa12, BeBuSaVa12}.
\item Optimal approximation~\cite{LiScSe12}.
\end{itemize}
The important properties of ASTS emanate
directly from the topological properties of the underlying
analysis-suitable T-mesh and resulting set of T-spline basis functions
constructed from it.

\subsection{T-spline basis functions, spaces, and geometry}
\label{sec:ts_basis}
Given a parametric domain $\hat{\Omega} = [0,1]^2$ we define global
horizontal and vertical \textit{open} knot vectors $h\mathsf{K} = \{s_{1}, s_{2}, \ldots,
s_{m}\}$ and $v\mathsf{K} = \{t_{1}, t_{2},
\ldots, t_{n}\}$, respectively. In other words, 
$$0=s_{1} = \ldots = s_{p+1} <
s_{p+2} \leq \ldots \leq
s_{m-p-1} < s_{m-p} = \ldots = s_{m} = 1$$
and
$$0=t_{1} = \ldots = t_{q+1} <
t_{q+2} \leq \ldots \leq
t_{n-q-1} < t_{n-q} = \ldots = t_{n} = 1.$$
As a result, every T-mesh vertex $V = \{i,j\} \in \doublehat{\Omega}$
has the parametric representation $\{s_{i}, t_{j}\} 
\in \hat{\Omega}$. For reasons that will become apparent, we refer to a
cell $C \in \mathsf{C}_{ext}$ with \textit{positive parametric area} as a B\'{e}zier 
element. The parametric domain of a B\'{e}zier element is denoted by
$\hat{\Omega}^e$. The set of all B\'{e}zier elements in a T-mesh is
denoted by $\mathsf{E}$. 

For each anchor $A = a \times b \in \mathsf{A}$ we construct horizontal and vertical local index
vectors $\{i_1, \ldots, i_{p+2}\}$ and
$\{j_1, \ldots, j_{q+2}\}$ made up of 
increasing (but not necessarily consecutive) indices in
$h\mathsf{I}$ and $v\mathsf{I}$, respectively. Note that for $p$
odd, $\{i_{(p+3)/2}\} = a$, and for $p$ even,
$(i_{(p/2)+1}, i_{(p/2)+2}) = a$.
Similar relationships hold for $q$. The
procedure for determining local index vectors is shown
in~\cref{fig:local-func-examples} for various polynomial degrees.  To clarify this
procedure we describe the anchors and associated local index vectors.  
In~\Cref{fig:local-func-examples}a, ${p}= 2$ and $q=2$ and thus the example anchor is the cell
$(4, 8) \times (3, 7)$.  
The horizontal local index vector is $\{3, 4, 8, 10\}$ 
and the vertical local index vector is $\{2, 3, 7, 9\}.$  
We observe that subset of the vertical skeleton located at index
 $9$ does not span the entire 
height of the anchor cell, hence it is not included in the horizontal 
local index vector; similarly since subset of the horizontal skeleton located at index $8$ 
does not span the entire width of the cell it not included in 
the vertical local index vector.
In~\Cref{fig:local-func-examples}b, ${p}= 3$ and $q=2$ thus the 
example anchor is the vertical edge
$\{9\} \times (7, j)$.  The horizontal local index vector 
is $\{5, 8, 9, 11, 12\}$ and the vertical local index vector 
is $\{6, 7, 9, 10\}.$ Similar to the prior
example the subset of the vertical skeleton at index $8$ does not span the entire 
height of the anchor edge, hence it is not included in the 
horizontal local index vector. \Cref{fig:local-func-examples}c 
shows the case where ${p}= 2$ and $q=3$ thus the example anchor is the horizontal edge
$(4, 7) \times \{8\}$.  The horizontal local index vector is 
$\{3, 4, 7, 8\}$ and the vertical local index vector is 
$\{3, 4, 8, 9, 10\}.$ In the last case, shown 
in~\Cref{fig:local-func-examples}d, ${p}= 3$ and $q=3$ thus the 
example anchor is the vertex
$\{8\} \times \{8\}$.  The horizontal local 
index vector is $\{4, 5, 8, 9, 11\}$ 
and the vertical local index vector is $\{3, 4, 8, 9, 10\}.$ 

\begin{figure}
  \begin{center}
    \begin{tabularx}{0.9\textwidth}{XX}
      \begin{center}
        \includegraphics[scale=0.6]{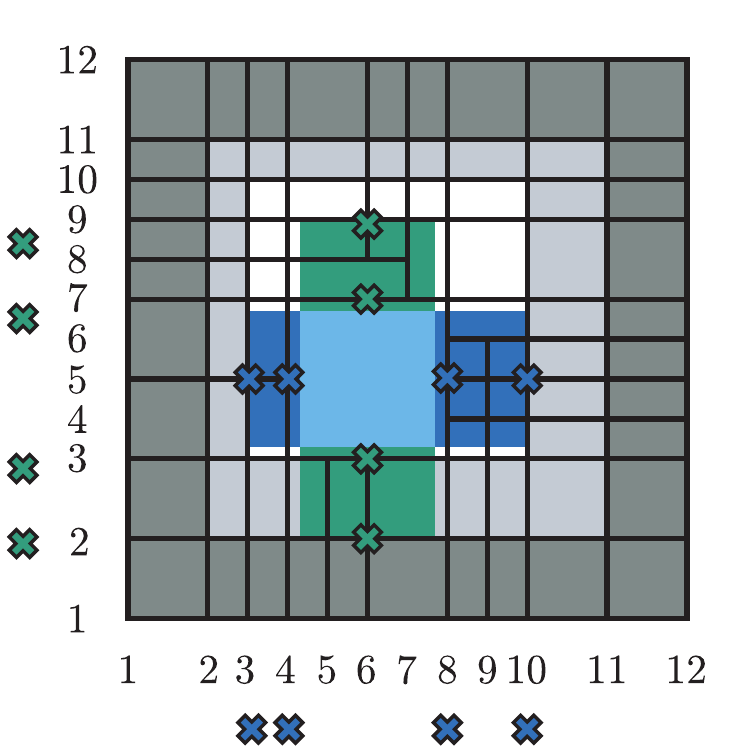}
      \end{center}
      &
      \begin{center}
        \includegraphics[scale=0.6]{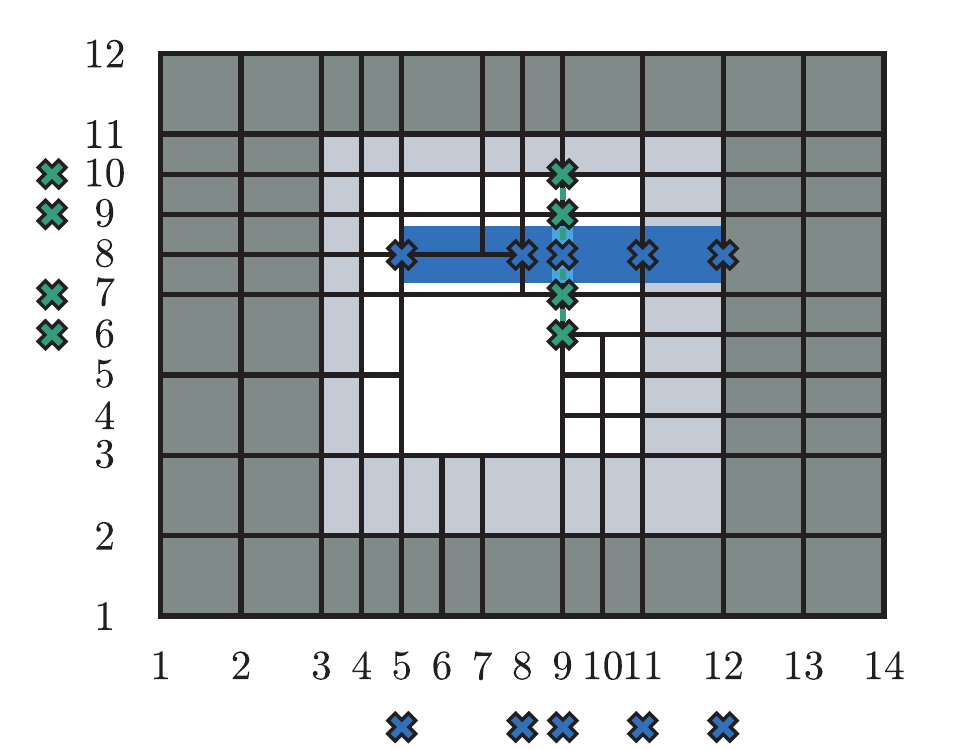}
      \end{center}
      \\
      \vspace{-30pt}
      \begin{center}
        (a) $p =2,\;q=2$
      \end{center}
      &
      \vspace{-30pt}
      \begin{center}
        (b) $p=3,\;q =2$
      \end{center}
      \\
      \vspace{-30pt}
      \begin{center}
        \includegraphics[scale=0.6]{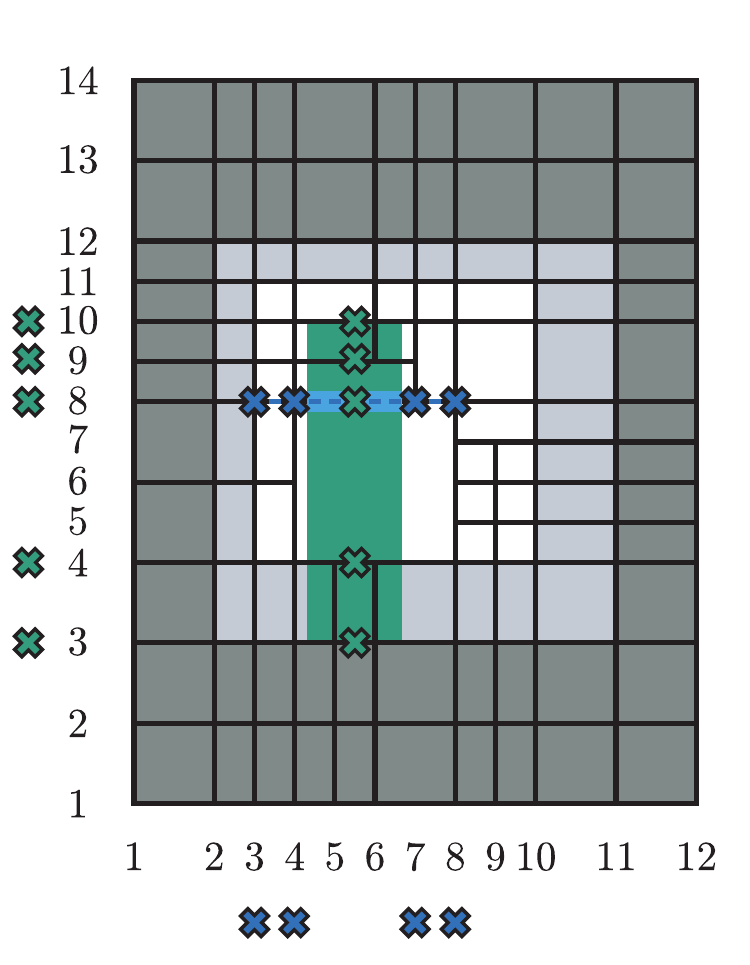}
      \end{center}
      &
      \vspace{-30pt}
      \begin{center}
        \includegraphics[scale=0.6]{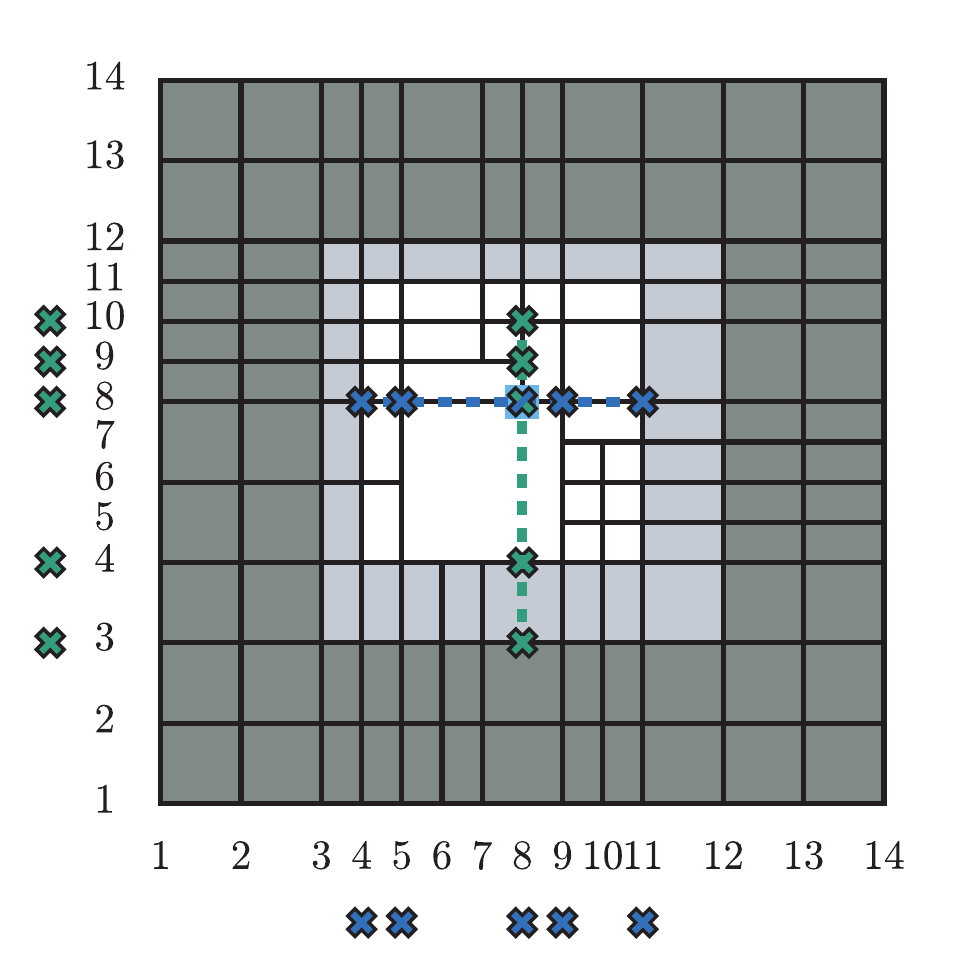}
      \end{center}
      \\
      \vspace{-30pt}
      \begin{center}
        (c) $p=2,\;p=3$
      \end{center}
      &
      \vspace{-30pt}
      \begin{center}
        (d) $p=3,\;q=3$
      \end{center}
    \end{tabularx}
    \caption{Examples for how local index vectors are constructed for
      T-spline basis functions of varying values of the polynomial degrees
      $p$ and $q$.  The function anchors are marked with light blue.
      The horizontal line used to determine the horizontal local index vector is
      indicated with a dark blue dashed line.  The vertical line used to
      calculate the vertical local index vector is marked with a green dashed
      line.  The indices that contribute to the local index vectors are
      marked with $\times$.} 
    \label{fig:local-func-examples}
  \end{center}
\end{figure}
The T-spline blending function $N_{A}^{p, q}(s, t)$ is given by
\begin{equation}
N_{A}^{p, q}(s, t) := N_{A}^{p}
[s_{i_1}, \ldots, s_{i_{p+2}}](s) N_{A}^{q}[t_{j_1}, \ldots, t_{j_{q+2}}](t) \quad 
\forall (s,t) \in \hat{\Omega}
\end{equation}
where $N_{A}^{p}[s_{i_1}, \ldots, s_{i_{p+2}}](s)$ and $N_{A}^{q}[t_{j_1}, \ldots, t_{j_{q+2}}](t)$ are B-spline
basis functions associated with the local knot vectors
$[s_{i_1}, \ldots, s_{i_{p+2}}] \subset h\mathsf{K}$ and
$[t_{j_1}, \ldots, t_{j_{q+2}}] \subset v\mathsf{K}$.

We define $\mathsf{N}$ to be the set of all basis functions associated with a T-mesh.
Given a weight $w_{A} \in \mathbb{R}^+$ for each $A
\in \mathsf{A}$ a
\textit{rational} T-spline basis function $R_{A}^{p,q}
: \hat{\Omega} \rightarrow \mathbb{R}$ can be written as 
\begin{align}
R_{A}^{p,q}(s,t) &= \frac{N_{A}^{p,q}(s,t)}
{\sum_{A \in
    \mathsf{A}} w_{A}
  N_{A}^{p,q}(s,t)} \\
&= \frac{N_{A}^{p,q}(s,t)}{w(s,t)}
\end{align}
where $w(s,t) : \hat{\Omega} \rightarrow \mathbb{R}$ is called a
weight function. For clarity we will often suppress the dependence on
the polynomial degrees $p,q$ and write the basis function as $R_A$.
\Cref{fig:tmesh_cases} shows several
T-spline basis functions plotted in the parametric domain $\hat{\Omega}$. 
\begin{figure}
\begin{center}
\begin{tabularx}{0.9\textwidth}{XX}
\begin{center}
  \includegraphics[scale=0.85]{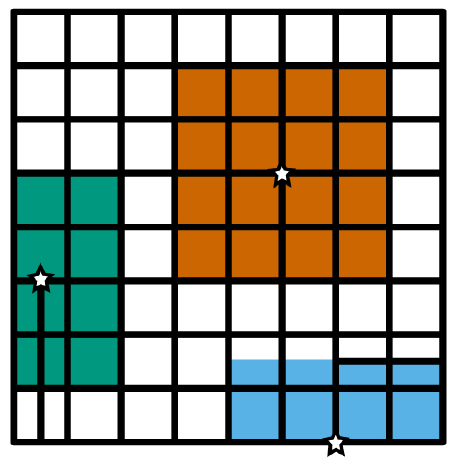}
\end{center} 
&
\begin{center}
 \includegraphics[scale=1.1]{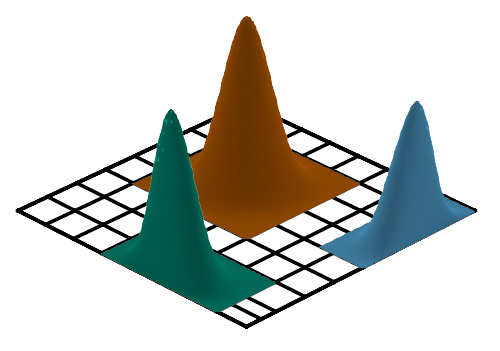}
\end{center}
\end{tabularx}
\caption{T-spline basis functions and supports in the parametric domain
  $\hat{\Omega}$ for $p = q =3$. Anchor locations are denoted by a star.}
\label{fig:tmesh_cases}
\end{center}
\end{figure}
An ASTS space,
denoted by $\mathcal{T}$, is the span of the blending functions in
$\mathsf{N}$ constructed from an analysis-suitable T-mesh.
Given vector valued control points, $\mathbf{P}_{A} \in
\mathbb{R}^{n}$, $n=2$ or $3$, the geometry of a T-spline can be written as
\begin{equation}
\mathbf{x}(s,t) = \sum_{A \in \mathsf{A}}
  \mathbf{P}_{A} w_{A} R_A(s,t).
\end{equation}

\section {Local linear independence of analysis-suitable T-splines}
\label{sec:asts_lli}
The local linear independence of analysis-suitable T-splines is an
important theoretical result in its own right and is critical for our definition of
hierarchical analysis-suitable T-splines in Section~\ref{sec:hts}. The (global)
linear independence of analysis-suitable T-splines was first shown
in~\cite{LiZhSeHuSc10}. Local linear independence is a stronger result
than global linear independence and is the
notion of linear independence enjoyed by standard $C^0$ finite element
bases.  Local linear independence implies that the finite element basis is
linearly independent over \textit{every} element domain. For smooth
locally-refined bases this element-level notion of linear independence
may be lost.
\label{sec:lli}
\subsection{Preliminaries}
Given a knot vector  $\{s_1 \leq s_2 \leq \ldots \leq
s_{m}\}$ and degree $p$ we can recursively define $n = m - p - 1$
B-spline basis functions $N_{i,p}(s)$ as follows:
\begin{align}
  N_{i}^0(s) = 
  \begin{cases}
    1 & s_{i} \leq s < s_{i+1} \\
    0 & \text{otherwise}.
  \end{cases}
\end{align}
\begin{align}
  N_i^p(s) = \frac{s - s_{i}}{s_{i+p} - s_{i}}N_i^{p-1}(s) 
  + \frac{s_{i+p+1} - s}{s_{i+p+1}-s_{i+1}}N_{i+1}^{p-1}(s).
\end{align}
The de Boor algorithm~\cite{Farin99} provides a standard method for
 evaluating a B-spline, although other possibilities
exist~\cite{Se10,Ra89}. A B-spline basis function $N_i^p(s)$ can
also be denoted by $N[s_i, \ldots,
s_{i+p+1}](s)$. Following~\cite{Schu93}, there exist dual
functionals $\lambda_{j,p} = \lambda[s_j, \ldots, s_{j+p+1}]$ such that
$\lambda_{j,p}(N_i^p(s)) = \delta_{j,i}$ where the Kronecker delta
$\delta_{i,j}$ is zero when $i \neq j$ and one otherwise. We have the following two results:
\begin{lemma}
\label{lem:db1} 
Suppose $f(s) = \sum_{i = 1}^n c_i N_i^p(s)$. 
Then $c_j = \lambda_j^p(f(s)).$
\end{lemma}
\begin{lemma}
\label{lem:db2} 
For any function $f(s) = \sum_{i = 1}^n c_i N_i^p(s)$, if $f(s) =
0$, $s_k \leq s \leq s_{k+1}$ then for $j=k-p, k-p+1, \ldots, k$ we
have that $\lambda_j^p(f(s)) = 0.$
\end{lemma} 
As shown in~\cite{BeBuSaVa12, BeBuChSa12} the notion of a dual basis
can be extended to analysis-suitable T-splines.
\begin{theorem}
\label{thm:glidb}
Given an analysis-suitable T-mesh and associated basis functions
$\{N_{A}^{p, q}: {A} \in \mathsf{A}_{p,q}\}$
the set of functionals $\{\lambda_{A}^{p,q}: {A} \in
\mathsf{A}_{p,q}\}$ form a dual basis. Specifically, we have that
$$\lambda_A^{p,q} = \lambda_A[s_{i_1}, \ldots, s_{i_{p+2}}] \otimes
\lambda_A[ t_{j_1}, \ldots, t_{j_{q+2}}]$$
where $\lambda_A[s_{i_1}, \ldots, s_{i_{p+2}}]$ and $\lambda_A[t_{j_1},
\ldots, t_{j_{q+2}}]$ are dual basis functions corresponding to
univariate B-splines~\cite{Schu93} with local knot vectors
$\{s_{i_1}, \ldots, s_{i_{p+2}}\}$ and $\{t_{j_1},
\ldots, t_{j_{q+2}}\}$, respectively.
\end{theorem}

\subsection{Proof of local linear independence}
\begin{lemma}
\label{lem:bez1}
Let $C \in \mathsf{C}_{ext}$ be a cell from an analysis-suitable
extended T-mesh with vertices $ \{i_l, j_b\},$  $\{i_r, j_b\},$  $\{i_r, j_t\},$ and
$\{i_l, j_t\}$. If the basis function anchored at ${A} \in
\mathsf{A}_{p,q}$ with local index vectors $\{i_{1}, \ldots,
i_{p+2}\}$ and 
$\{j_{1}, \ldots, j_{q+2}\}$ is non-zero over $C$ then there must either exist an
integer $k$, $1 \leq k \leq p+1$, such that $i_l = i_{k}$ and
$i_r = i_{k+1}$ or an integer $\ell$, $1 \leq \ell \leq q+1$,
such that $j_b = j_{\ell}$ and $j_t = j_{\ell+1}$. 
\end{lemma}
\begin{proof} Suppose the lemma is false, then there exists at least one cell in 
$\mathsf{C}_{ext}$ that violates the condition. We denote this cell by be $C^{*}$. Since $C^{*}$ 
violates the condition there exists at least one corner of $C^{*}$
that lies in $\{i_{k}, i_{k+1}\} \times \{j_{\ell},
j_{\ell+1}\}$ where $1 \leq k \leq p+1$ and $1 \leq \ell \leq
q+1$. Without loss of generality we may assume that
$\lfloor\frac{p}{2} \rfloor + 3 \leq k \leq p + 1$ and
$\lfloor\frac{q}{2} \rfloor + 3 \leq \ell \leq q + 1$.    
We have following three cases:
\begin{enumerate}
  \item The corner is a vertex of the original T-mesh. This
    violates Lemma 3.2(a) in~\cite{BeBuSaVa12}.
  \item The corner is the result of the intersection of two perpendicular
    T-junction extensions. This violates the assumption that the
    T-mesh is analysis-suitable.
  \item The corner is the result of the intersection of one
    T-junction extension and a T-mesh edge. Without loss  
    of generality, we assume the edge is a horizontal edge and the
    T-junction extension is vertical and is associated with a
    T-junction $T_1$. As the edge cannot intersect the vertical line
    $i_{\lceil p/2 \rceil+1}$, it must terminate in a T-junction
    $T_{2}$. Examining the extensions associated with $T_1$ and $T_2$
    we find that they must intersect. This violates the assumption
    that the T-mesh is analysis-suitable.
\end{enumerate}
Hence, such cell cannot exist.
\end{proof}

\begin{theorem}\label{thm:lli}
The basis for an analysis-suitable T-spline is locally linearly independent.
\end{theorem}
\begin{proof}
Let an arbitrary B\'{e}zier element, $E \in \mathsf{E}$,
from an analysis-suitable T-mesh be given. We denote the element vertices
by $\{i_l, j_b\},$ $\{i_r, j_b\},$  $\{i_r, j_t\}, $ and  $\{i_l, j_t\}$. Let
$\mathsf{A}^e$ be the set of anchors whose corresponding basis functions 
are non-zero over $\hat{\Omega}^e$.  Assume that $f(s,t) = \sum_{A \in \mathsf{A}^e}
c_A N_{A}(s,t) = 0$ for all $s,t \in \hat{\Omega}^e$. 
By Lemma \ref{lem:bez1}, since $N_{A}$ is non-zero over $E$,
there exists either a $k$, $1 \leq k \leq p+1$, such that 
$i_l = i_{k}$ and $i_r = i_{k+1}$ or an $\ell$, $1 \leq \ell \leq q+1$, 
such that $j_b = j_{\ell}$ and $j_t = j_{\ell+1}.$ By
Lemma~\ref{thm:glidb}
\[\begin{split}c_A &= \lambda_{A}^{p, q}(f(s,t))\\
 &= \lambda_{A}[s_{i_1}, \ldots, s_{i_{p+2}}] 
\otimes \lambda_{A}[t_{j_1}, \ldots, t_{j_{q+2}}]
(f(s,t)).\end{split}\]
Since $f(s,t) =0$ 
for $s_{i_k} \leq s \leq s_{i_{k+1}}$ or 
$f(s,t) = 0$ for $t_{j_{ \ell}} \leq t \leq t_{j_{ \ell+1}}$
we have by Lemma~\ref{lem:db2} that $$\lambda_{A}[s_{i_1}, \ldots, s_{i_{p+2}}] 
(f(s,t)) = 0$$ or $$ 
\lambda_{A}[t_{j_1}, \ldots, t_{j_{q+2}}](f(s,t)) = 0$$
which completes the proof.
\end{proof}

\section{Hierarchical analysis-suitable T-splines}
\label{sec:hts}
A hierarchical T-spline space is constructed from a finite sequence of $N$ nested
ASTS spaces, $\mathcal{T}^{\alpha} \subset \mathcal{T}^{\alpha+1}$, $\alpha = 1,
\ldots, N-1$, and $N$ bounded open index domains, $\doublehat{\Omega}^{N}
\subseteq \doublehat{\Omega}^{N-1} \subseteq 
\cdots \subseteq \doublehat{\Omega}^1$,
which define the nested domains for the hierarchy. Two
important theoretical results for ASTS will be used in the construction of hierarchical analysis-suitable T-splines:
\begin{theorem}
\label{thm:nest_result}
Given two analysis-suitable T-meshes with non-overlapping T-junction extensions, 
$\mathsf{T}^1$ and $\mathsf{T}^2$ , if $\mathsf{T}_{ext}^1 \subseteq
\mathsf{T}_{ext}^2$, then $\mathcal{T}^1 \subseteq \mathcal{T}^2$.
\end{theorem}
\begin{theorem}
\label{thm:lli2}
Analysis-suitable T-splines are locally linear independent.
\end{theorem}
We note that to accommodate overlapping T-junction extensions requires
a minor generalization of~\cref{thm:nest_result} which is not reproduced here to maintain
clarity of exposition. For a complete description of the underlying
theory we refer the interested reader to~\cite{LiScSe12}. The local
linear independence of ASTS is proven in~\cref{sec:lli}.

\subsection{Sequences of analysis-suitable T-meshes }
We construct a sequence of $N$ analysis-suitable T-meshes such
that $\mathcal{T}^{\alpha} \subseteq
\mathcal{T}^{\alpha+1}$, $\alpha=1,\ldots,N-1$, as follows:
\begin{enumerate}
\item\label{alt:cry} Create $\mathsf{T}^{\alpha+1}$ from
  $\mathsf{T}^{\alpha}$ by subdividing each cell in $\mathsf{E}^{\alpha}$ into four congruent cells.  
\item Extend T-junctions in $\mathsf{T}^{\alpha+1}$ until it is
  analysis-suitable and $\mathsf{T}_{ext}^{\alpha} \subseteq
\mathsf{T}_{ext}^{\alpha+1}$.
\end{enumerate}
This algorithm is graphically demonstrated in
Figure~\ref{fig:refinement} for a particular T-mesh. For an efficient and general
algorithm to produce nested analysis-suitable T-spline spaces
see~\cite{ScLiSeHu10}.

\begin{figure}
\begin{center}
\begin{tabularx}{0.9\textwidth}{XX}
{\begin{tabularx}{0.9\textwidth}{X}
\begin{center}
\includegraphics[scale=0.5]{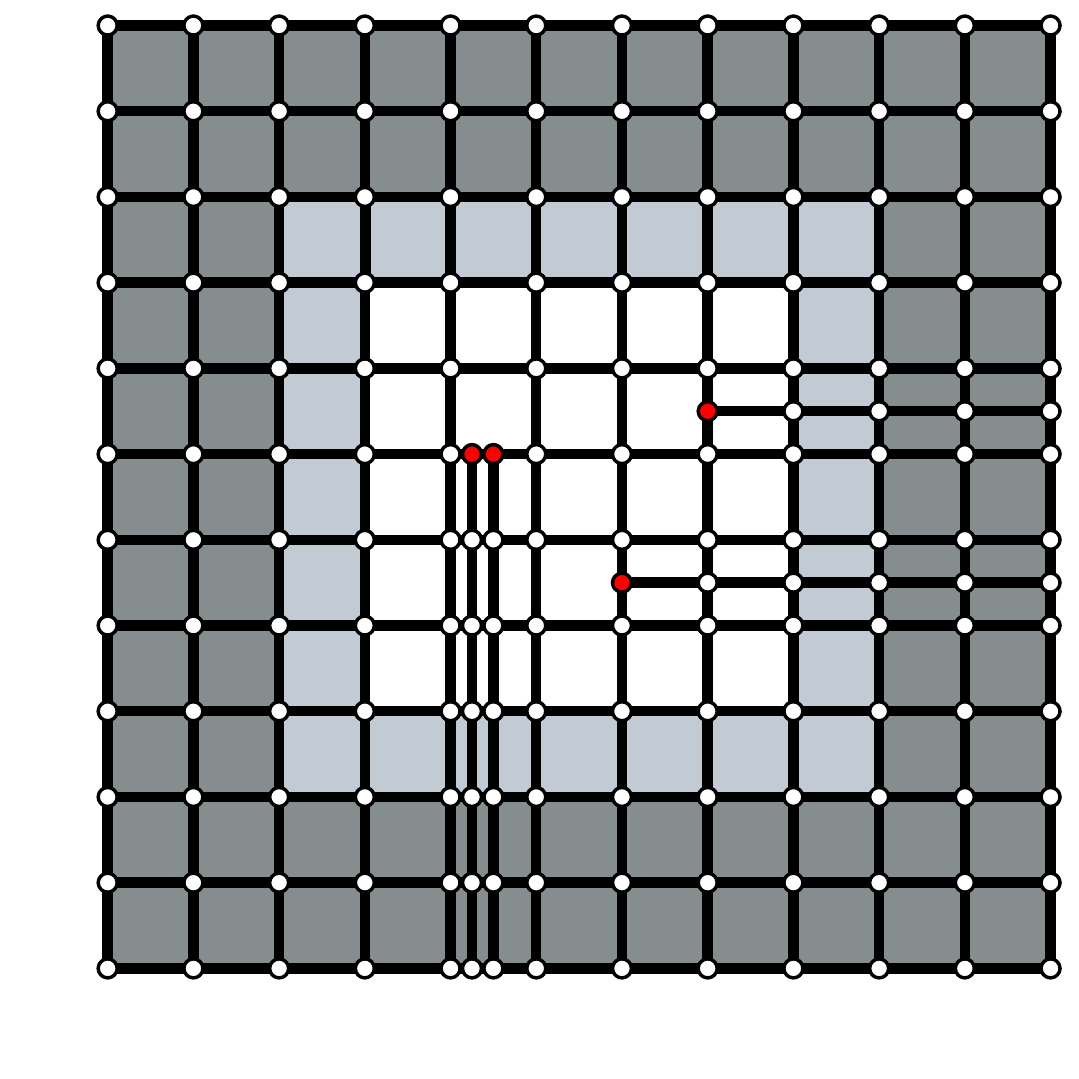}
\end{center} 
\\
\vspace{-15pt}
\begin{center}
(a) The initial T-mesh, $\mathsf{T}^{\alpha}$.
\end{center}
\end{tabularx}}
\\
\begin{center}
\includegraphics[scale=0.4]{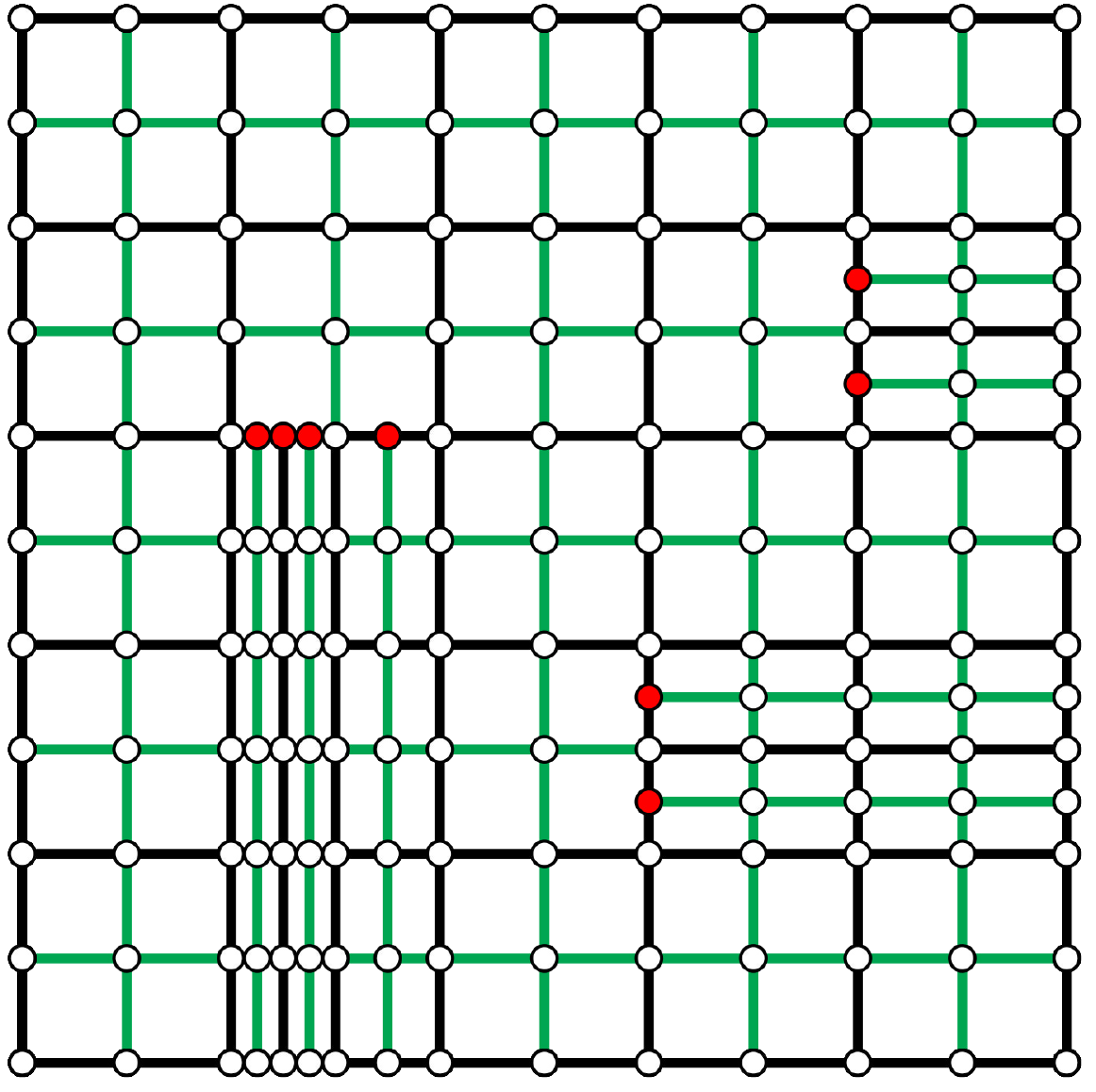}
\end{center} 
&
\begin{center}
\includegraphics[scale=0.4]{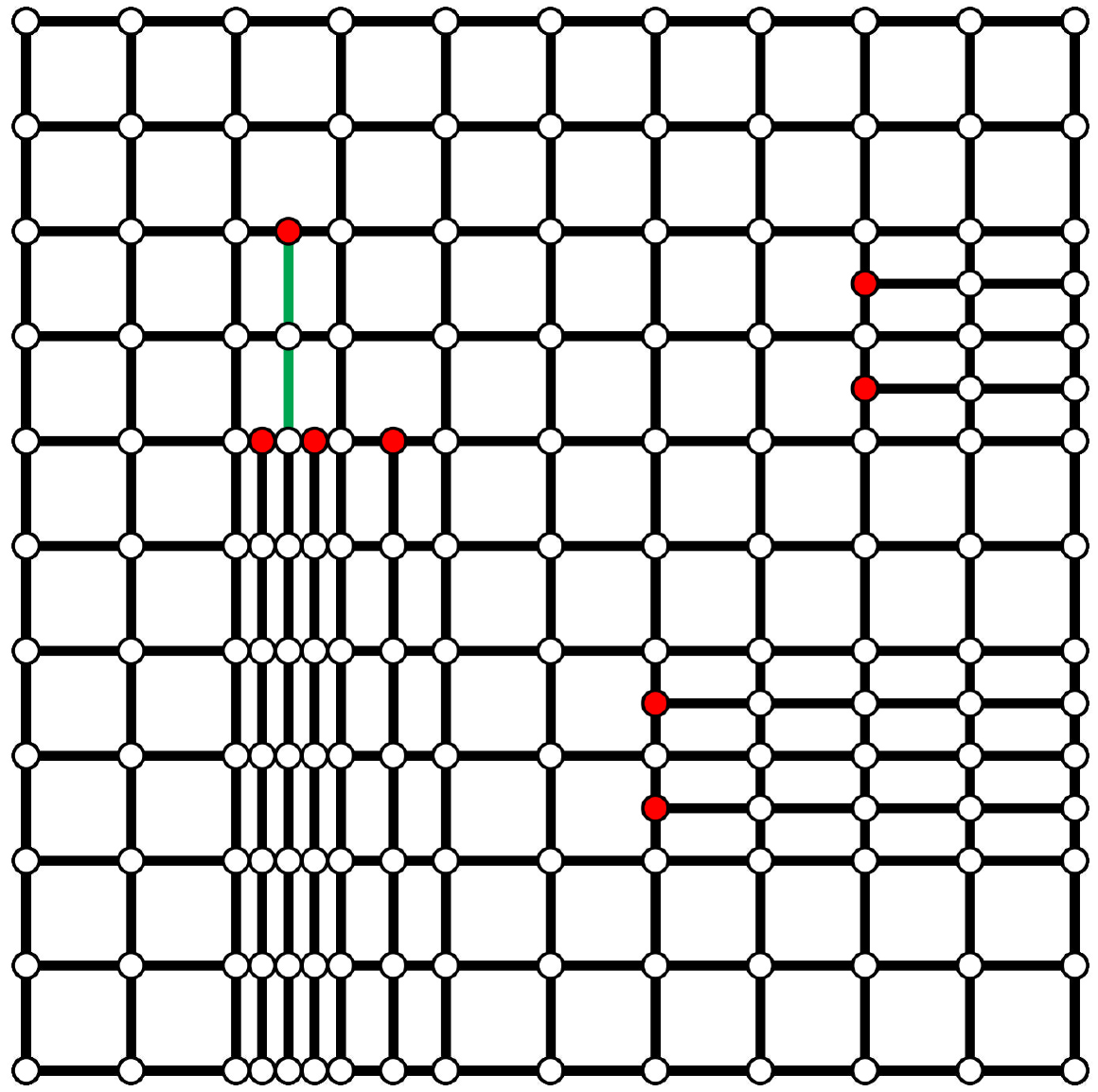}
\end{center}
\\
\vspace{-15pt}
\begin{center}
(b) Create $\mathsf{T}^{\alpha+1}$ from
$\mathsf{T}^{\alpha}$ by subdividing B\'{e}zier elements. 
\end{center}
&
\vspace{-15pt}
\begin{center}
(c) Extend T-junctions until $\mathsf{T}^{\alpha+1}$ is
analysis-suitable and
$\mathsf{T}^{\alpha} \subseteq \mathsf{T}^{\alpha+1}$. 
\end{center}
\end{tabularx}
\caption{Generating nested analsysis-suitable T-meshes. In Figures (b) and (c) the
zero parametric area region is omitted for clarity.} 
\label{fig:refinement}
\end{center}
\end{figure}

\subsection{Hierarchical T-spline spaces}
The hierarchical analysis-suitable T-spline basis can be
constructed recursively in a manner analogous to that used
for hierarchical B-splines~\cite{VuGiJuSi11}:
\begin{enumerate}
\item  Initialize $\mathsf{H}^1 =\mathsf{N}^1.$
\item Recursively construct $\mathsf{H}^{\alpha + 1}$ from $\mathsf{H}^{\alpha}$ by setting 
\[\mathsf{H}^{\alpha+1} = \mathsf{H}^{\alpha+1}_{coarse} \cup
\mathsf{H}^{\alpha+1}_{fine}, \alpha = 1, \ldots ,N-1,\] 
where
\[\mathsf{H}^{\alpha + 1}_{coarse} = \{N \in \mathsf{H}^{\alpha}:
\text{supp}(N) \nsubseteq \hat{\Omega}^{\alpha + 1}\},\] 
and
\[\mathsf{H}^{\alpha + 1}_{fine} = \{N \in \mathsf{N}^{\alpha+1}:
\text{supp}(N) \subseteq \hat{\Omega}^{\alpha + 1}\}.\]
\item Set $\mathsf{H} = \mathsf{H}^N.$
\end{enumerate}
We denote the number of functions in $\mathsf{H}$ by $n_f$.
We call the space spanned by the functions in $\mathsf{H}$ a
hierarchical analysis-suitable T-spline space and denote it by
$\mathcal{H}$. To make the ideas concrete a univariate hierarchical
spline space is shown in Figure~\ref{fig:multi_level_basis}.

\begin{figure}
  \centering
  \includegraphics[width=5in]{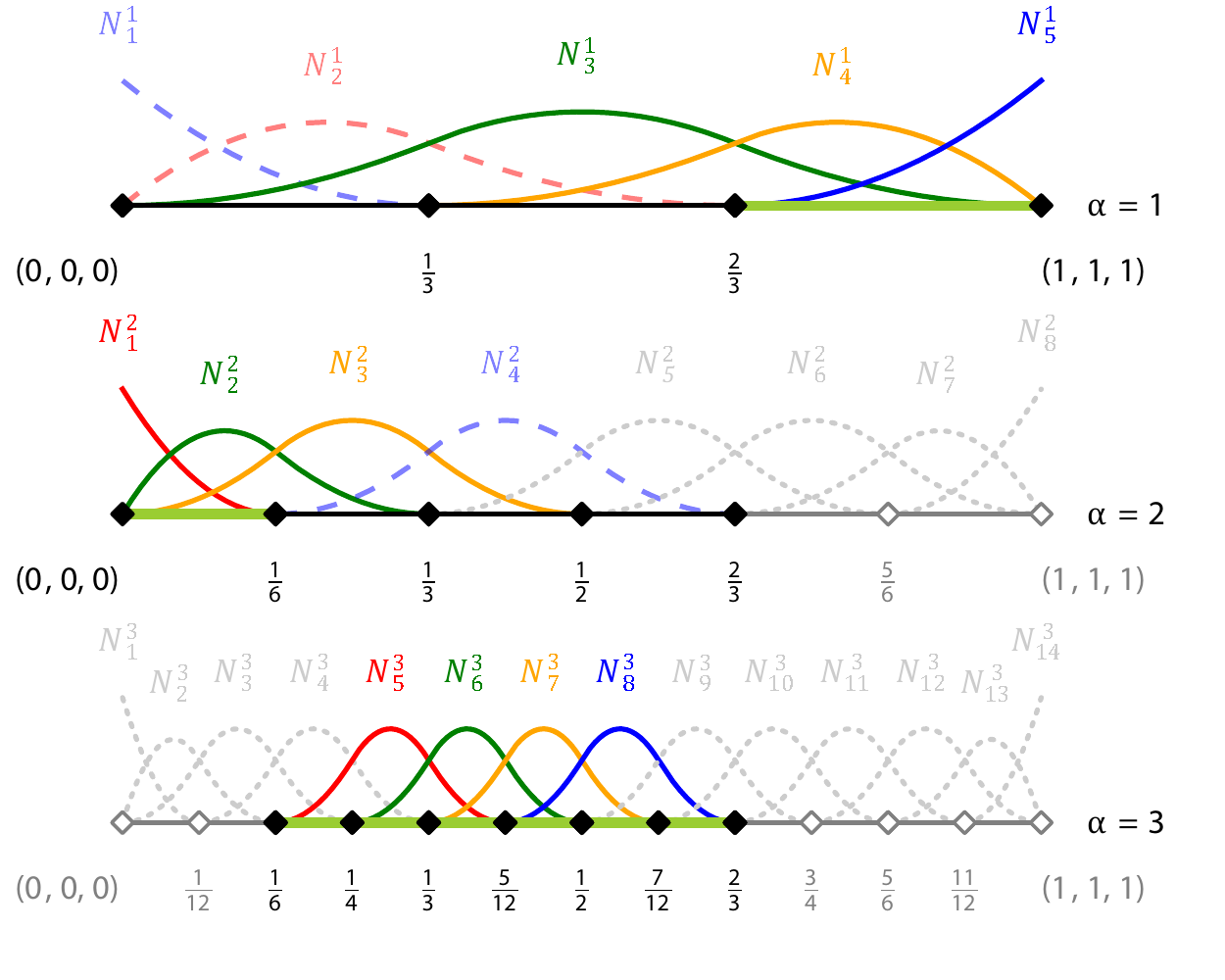}
  \caption{Basis functions for a univariate hierarchical spline space.
    Functions in $\mathsf{H}$ are indicated by solid colored
    lines.  Functions dependent on higher-level functions are
    indicated by dashed colored lines. Functions which are not fully supported in
    the hierarchy of domains are
    indicated by grey dotted lines. Hierarchical domains are black and
    B\'{e}zier elements are green. Note that functions in
    $\mathsf{H}$ are supported entirely by B\'{e}zier elements.}
  \label{fig:multi_level_basis}
\end{figure}

The linear independence of the functions in $\mathsf{H}$ follows 
immediately from the definition of hierarchical T-splines and the local linear 
independence of ASTS (see~\Cref{sec:lli}).
\begin{lemma} The functions in the hierarchical basis $\mathsf{H}$ are
  linearly independent.
\end{lemma}
\begin{proof}
See Lemma 2 in~\cite{VuGiJuSi11}
\end{proof}

\begin{lemma}
\label{lem:span}
Given $\mathsf{H}^1, \ldots, \mathsf{H}^{N}$, a sequence of hierarchical
analysis-suitable T-spline bases,
$\operatorname{span} \mathsf{H}^{\alpha} \subseteq
\operatorname{span} \mathsf{H}^{\alpha + 1}$, $\alpha = 1, \ldots, N-1$.
\end{lemma}
\begin{proof}
See Lemma 3 in~\cite{VuGiJuSi11} 
\end{proof}
By construction, $\mathcal{T}^1 \subseteq \mathcal{H}$, thus the approximation
properties of analysis-suitable T-splines are inherited by their
hierarchical counterpart. In particular, constants are exactly
represented and all patch tests are exactly
satisfied~\cite{LiScSe12,Hug00}.

\section{B\'{e}zier extraction of hierarchical analysis-suitable
  T-splines} 
\label{sec:extract}
The B\'{e}zier extraction framework~\cite{ScBoHu10,Borden:2010nx,
  ScThEv13} can be extended to HASTS in a straightforward fashion.
Using B\'{e}zier extraction, the spline hierarchy is collapsed onto a
single level finite element mesh which can then be processed by
standard finite element codes without any explicit knowledge of
HASTS algorithms or data structures.

\subsection{Bernstein basis functions}
The univariate Bernstein basis functions are written as
\begin{equation}
B_{i,p}(\xi) =\frac{1}{2^p}{p \choose
  i-1}(1-\xi)^{p-(i-1)}(1 + \xi)^{i-1}
\end{equation}
where $\xi \in [-1,1]$ and the binomial coefficient $p \choose i-1$$ =
\frac{p!}{(i-1)!(p+1-i)!}$, $1 \le i \le p + 1$. In CAGD, the
Bernstein polynomials are usually defined over the unit interval
$[0,1]$, but in finite element analysis the biunit interval is
preferred to take advantage of the usual domains for Gauss quadrature.
The univariate Bernstein basis has the following properties: 
\begin{itemize}
\item {\textit{Partition of unity}}. $$\sum_{i=1}^{p+1} B_{i,p}(\xi)=1 \quad \forall
\xi \in [-1, 1]$$
\item {\textit{Pointwise nonnegativity}}. $$B_{i,p}(\xi) \ge 0 \quad \forall \xi \in
[-1,1]$$
\item {\textit{Endpoint interpolation}}. $$B_{1,p}(-1)=B_{p+1,p}(1)=1$$
\item {\textit{Symmetry}}. $$B_{i,p}(\xi) = B_{p+1-i, p}(-\xi) \quad \forall
\xi \in [-1, 1]$$
\end{itemize}
\Cref{fig:bernstein} shows the Bernstein basis for polynomial degrees $p=1,2,3$.
\begin{figure}
  \centering
  \includegraphics[scale=1]{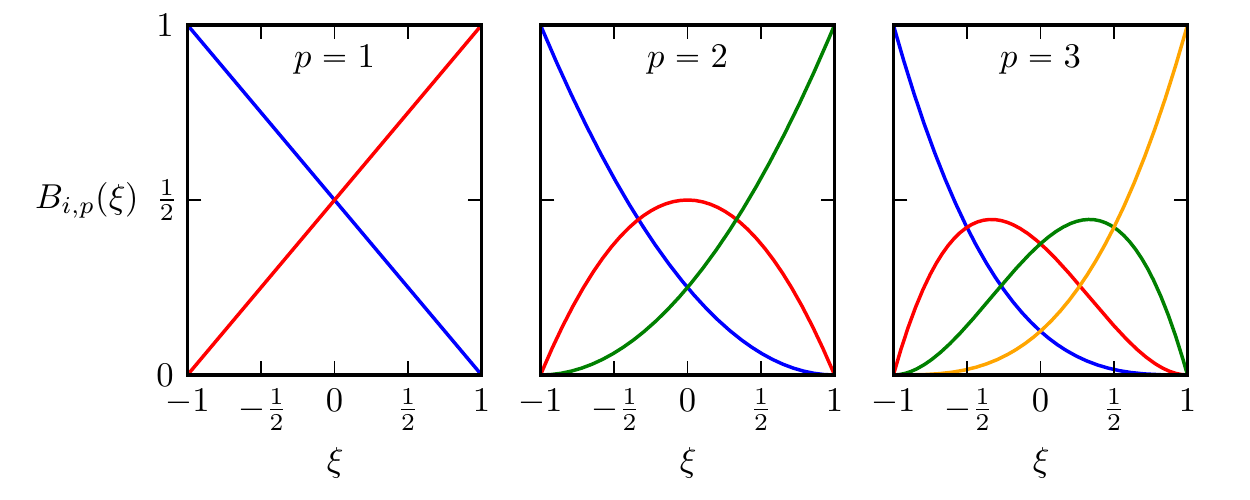}
  \caption{The Bernstein basis for polynomial degrees $p=1,2,3$.} 
  \label{fig:bernstein}
\end{figure}
We construct a bivariate Bernstein basis function of degree
$\mathbf{p}=\{p, q\}$ by
$B_{a,\mathbf{p}}:\bar{\Omega} \rightarrow \mathbb{R}^+ \cup 0$ where
$a = 1,\ldots, n_b$, $n_b =(p+1)(q+1)$, and
$\bar{\Omega}=[-1,1]^2$, as the 
tensor product of univariate 
basis functions
\begin{equation}
B_{a(i,j),\mathbf{p}} (\xi,\eta) =  B_{i,p}(\xi) B_{j,q} (\eta )
\end{equation}
with
\begin{equation}
\label{eq:bern_map_2d}
a\left(i,j\right) = (p+1)(j-1) + i.
\end{equation}

\subsection{The geometry of a hierarchical representation}
\label{sec:hbs_geom}
In a single level T-spline, basis functions and control points have a
one-to-one relationship and each control point influences
the geometry in a similar manner. In a hierarchical context it is
common to only associate control points with the functions in
$\mathsf{N}^1$. This is the convention adopted in this
paper. Note that by construction every blending function in
$\mathsf{N}^1$ can be written in terms of basis functions in
$\mathsf{H}$ (see~\Cref{lem:span}). We call the functions in $\mathsf{N}^1$
\textit{geometric} blending functions. We use $n_g$ to denote the
number of geometric blending functions.

Given vector valued control points, $\mathbf{P}_{G} \in
\mathbb{R}^{n}$, $n=2$ or $3$, and weights $w_G$, the geometry of a hierarchical
representation $\mathbf{x}:\hat{\Omega} \rightarrow \Omega$ can be written as 
\begin{align}
\mathbf{x}(s,t) &= \frac{\sum_{G=1}^{n_g} \mathbf{P}_G w_G N_{G}(s,t)}{\sum_{G=1}^{n_g} w_{G}
  N_{G}(s,t)} \\
&= \frac{\sum_{G=1}^{n_g} \mathbf{P}_G w_G N_{G}(s,t)}{w(s,t)}
\end{align}
where $(s,t) \in \hat{\Omega}$, $G$ is used to index the geometric
blending functions, and $w(s,t)$ is the weight function. The
decoupling of geometry from the basis functions in $\mathsf{H}$ is an
additional complexity unique to hierarchical representations which is
elegantly addressed via B\'{e}zier extraction.

\subsection{B\'{e}zier Elements}
\label{sec:bez}
The set of B\'{e}zier elements underlying a hierarchical T-spline are
determined recursively in a manner similar to the basis. We denote the
set of B\'{e}zier elements in a hierarchy by $\mathsf{HE}$. We construct
$\mathsf{HE}$ as follows:
\begin{enumerate}
\item  Initialize $\mathsf{HE}^1 =\mathsf{E}^1.$
\item Recursively construct $\mathsf{HE}^{\alpha + 1}$ from $\mathsf{HE}^{\alpha}$ by setting 
\[\mathsf{HE}^{\alpha+1} = \mathsf{HE}^{\alpha+1}_{coarse} \cup
\mathsf{HE}^{\alpha+1}_{fine}, \alpha = 1, \ldots ,N-1,\] 
where
\[\mathsf{HE}^{\alpha + 1}_{coarse} = \{e \in \mathsf{HE}^{\alpha}:
\hat{\Omega}^e \nsubseteq \hat{\Omega}^{\alpha + 1}\},\] 
and
\[\mathsf{HE}^{\alpha + 1}_{fine} = \{e \in \mathsf{E}^{\alpha+1}:
\hat{\Omega}^e \subseteq \hat{\Omega}^{\alpha + 1}\}.\]
\item Set $\mathsf{HE} = \mathsf{HE}^N.$
\end{enumerate}
We denote the number of B\'{e}zier elements in $\mathsf{HE}$ by $n_e$.

\subsection{Element localization}
\label{sec:local}
Using standard techniques~\cite{ScBoHu10,Borden:2010nx} it is possible
to determine the set of functions in $\mathsf{H}$ which are nonzero
over any element in $\mathsf{HE}$. This gives rise to a standard element
connectivity map which, given an element index $e$ and local function
index $a$, returns a global function index $A$. In other words $A =
\operatorname{IEN}(e, a)$. The reader is referred to~\cite{Hug00} for
additional details on common approaches to finite element localization
and the $\operatorname{IEN}$ array. Note that 
$A$ can indicate an anchor or a global 
function index.

We write a rational hierarchical T-spline basis function, restricted
to element $e$, as
\begin{align}
R_{a}^e(s,t) &= \frac{N_{a}^{e}(s,t)}{w^e(s,t)}
\end{align}
where $(s,t) \in \hat{\Omega}^e$ and $w^e(s,t)$ is the
element weight function restricted to element $e$. The element geometric map $\mathbf{x}^e :
\hat{\Omega}^e \rightarrow \Omega^e$ is the restriction of
$\mathbf{x}(s,t)$ to element $e$.
 
\subsection{B\'{e}zier extraction}
To present the basic ideas, B\'{e}zier extraction for a B-spline curve
is shown graphically in Figure~\ref{fig:bez-ext}. B\'{e}zier
extraction constructs a linear transformation defined by a matrix referred to
as the extraction operator. The extraction operator maps a Bernstein 
polynomial basis defined on B\'{e}zier elements to the global spline
basis. The transpose of the extraction operator maps the control
points of the spline to the B\'{e}zier control points.

\label{sec:extraction}
\begin{figure}[htb]
  \centering
  \includegraphics[width=4.5in]{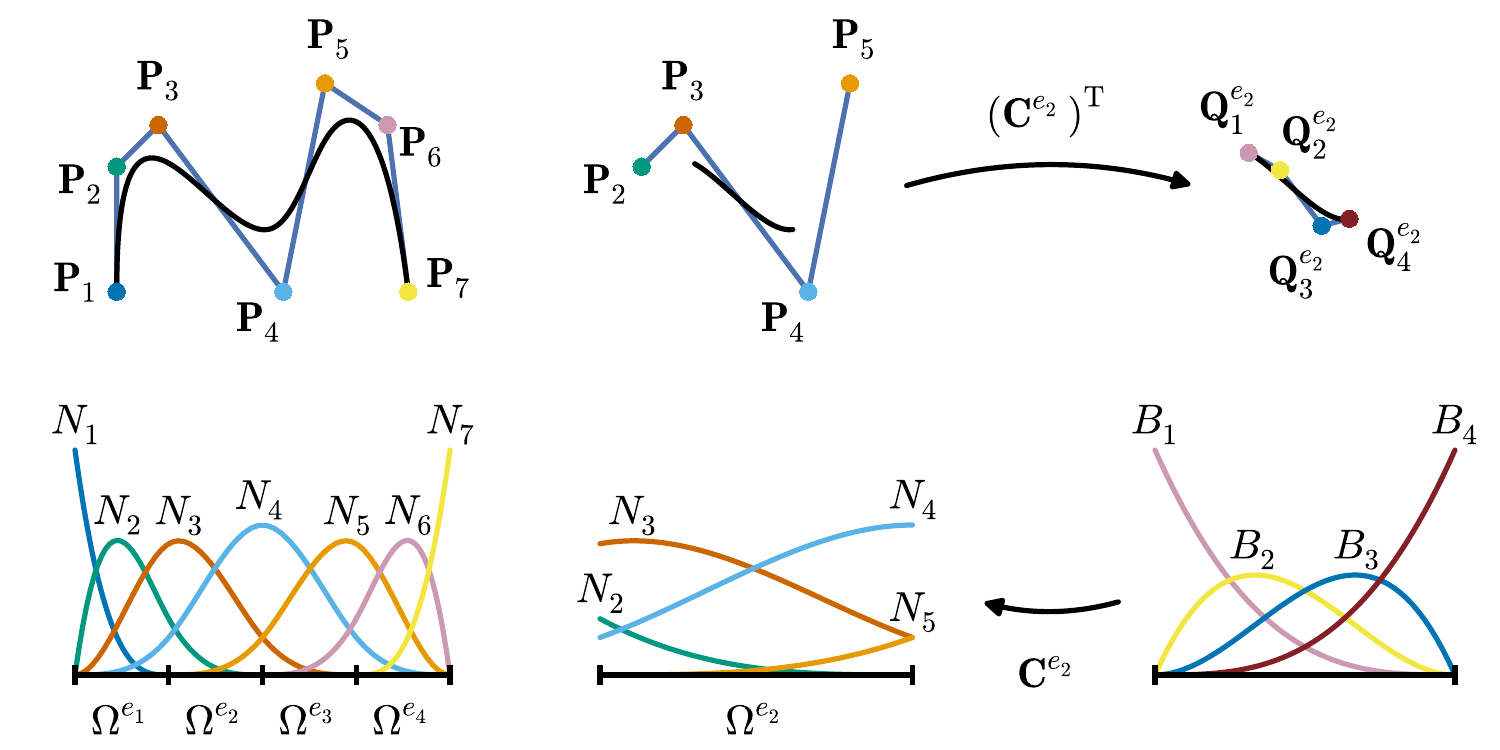}
  \caption{\label{fig:bez-ext}Illustration of the B\'{e}zier
    extraction operator $\mathbf{C}^e$ for a spline of degree 3.}
\end{figure}

Each hierarchical basis function supported
by element $e$ can be written in B\'{e}zier form as
\begin{align}
N_a^e(s({\xi}), t(\eta)) &= \sum_{b=1}^{n_b} c^a_b
B_b(\xi,\eta) \\
&= \bar{N}_a^e(\xi,\eta)
\end{align}
where the dependence of the Bernstein polynomial
$B_b(\xi,\eta)$ on the polynomial
degrees $p$ and $q$ has been suppressed for clarity. The overbar
will be used to denote a quantity written in terms of the
Bernstein basis defined over the element domain
$\bar{\Omega}$. The
B\'{e}zier coefficients $c^a_b$ are computed using standard knot
insertion techniques~\cite{ScBoHu10}. We denote the vector of hierarchical basis functions
supported by element $e$ by $\mathbf{H}^e(\xi,\eta)$ and the vector of
Bernstein basis functions by $\mathbf{B}(\xi,\eta)$. We then have
that
\begin{align}
  \mathbf{N}^e(s(\xi),t(\eta)) &= \mathbf{C}^e\mathbf{B}(\xi,\eta) \\
&= \bar{\mathbf{N}}^e(\xi,\eta)
\end{align}
where $\mathbf{C}^e$ is the \textit{element extraction
  operator} (see ~\cite{ScThEv13}). In other words, the element extraction operator is composed of the
B\'{e}zier coefficients $c_b^a$.

We write the element weight function as
\[\begin{aligned}
w^e(s(\xi), t(\eta)) &= \sum_{g=1}^{n_g^e} w_{g}^{e}
  N_{g}^{e}(s(\xi),t(\eta)) \\
  &= \sum_{g=1}^{n_g^e} w_{g}^{e}
  \sum_{b=1}^{n_b} c^g_b B_b(\xi,\eta) \\
&= \sum_{b=1}^{n_b} \left(\sum_{g=1}^{n_g^e} w_{g}^{e}c^g_b\right)
B_b(\xi,\eta) \\
&= \sum_{b=1}^{n_b} w_b^e
B_b(\xi,\eta) \\
&= \bar{w}^e(\xi, \eta).
\end{aligned}\]
where $n_g^e$ is the number of geometric basis functions which are
non-zero over element $e$. We may also write the rational hierarchical
basis functions as
\[\begin{aligned}
R_{a}^e(s(\xi), t(\eta)) &=
\frac{N_{a}^{e}(s(\xi), t(\eta))}{w^e(s(\xi), t(\eta))}
\\
&= \frac{\bar{N}_a^e(\xi, \eta)}{\bar{w}^e(\xi, \eta)}
\\
&= \bar{R}_a^e(\xi, \eta).
\end{aligned}\]
Finally, the element geometric map can be written as
\[\begin{aligned}
\mathbf{x}^e(s(\xi), t(\eta)) &= \frac{\sum_{g=1}^{n^e_g} \mathbf{P}_g^e w_g^e
N_{g}^{e}(s(\xi), t(\eta))}{w^e(s(\xi), t(\eta))}
\\ &= \frac{\sum_{g=1}^{n^e_g} \mathbf{P}_g^e w_g^e
\sum_{b=1}^{n_b} c^g_b
B_b(\xi, \eta)}{\bar{w}^e(\xi, \eta)} \\
&= \frac{\sum_{b=1}^{n_b} \left(\sum_{g=1}^{n^e_g} \mathbf{P}_g^e w_g^e
c^g_b\right) B_b(\xi, \eta)}{\bar{w}^e(\xi, \eta)} \\
&= \frac{\sum_{b=1}^{n_b} \mathbf{Q}_b^e w_b^e
B_b(\xi, \eta)}{\bar{w}^e(\xi, \eta)} \\
&= \bar{\mathbf{x}}^e(\xi, \eta).
\end{aligned}\]
The implementation of a finite element framework based on B\'{e}zier
extraction is described in detail in~\cite{ScBoHu10,Borden:2010nx}.

\section{Computational Results}\label{sec:iga}
We illustrate the use of hierarchical T-splines in the context of
isogeometric analysis. We consider problems that highlight the unique
attributes of both hierarchical refinement and T-splines.  The examples used are inspired by
those found in~\cite{ScLiSeHu10,VuGiJuSi11, ScThEv13}.

\subsection{A comparison between ASTS and HASTS local refinement}
\label{sec:apps}
We compare local refinement of ASTS to local
refinement of HASTS. When working with ASTS all refinement is performed on a single
level whereas when working with HASTS this constraint is relaxed. For additional
algorithmic details on local refinement of ASTS see~\cite{ScLiSeHu10}. We locally refine the T-spline
ship hull design shown in~\cref{fig:hull_geom} using both methods.
The geometry is constructed using the Autodesk T-spline plugin for
Rhino3d~\cite{tspline_rhino}. T-splines are popular in 
ship hull design because an entire hull can be modeled by a single
watertight surface with a minimal number of control
points~\cite{SeSe10}. T-junctions can be used to efficiently model
local features. Note that the initial T-spline of the hull contains
just $75$ control points and $36$ B\'{e}zier elements.

\begin{figure}[htb]
\begin{center}
\begin{tabularx}{0.9\textwidth}{X}
\begin{center}
\includegraphics[scale=0.65]{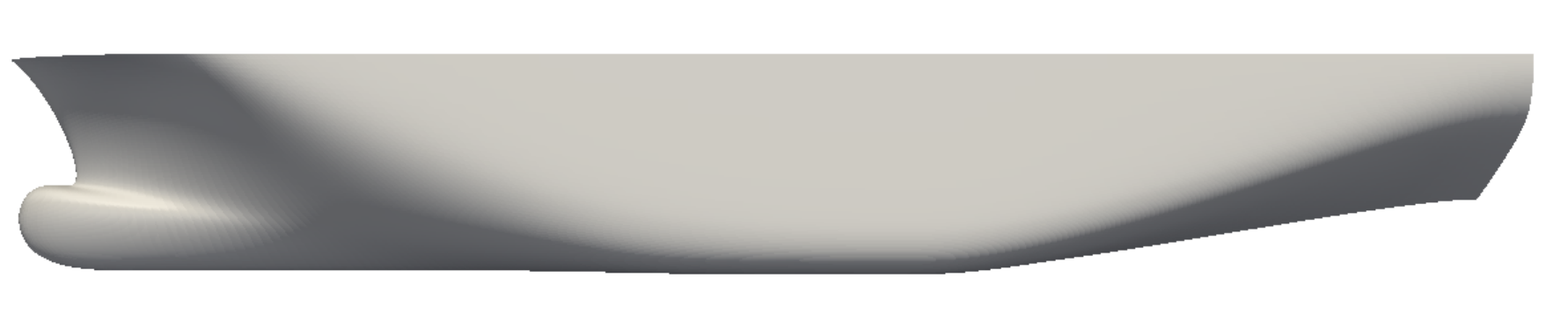}
\end{center}
\\
\begin{center}
\includegraphics[scale=0.3]{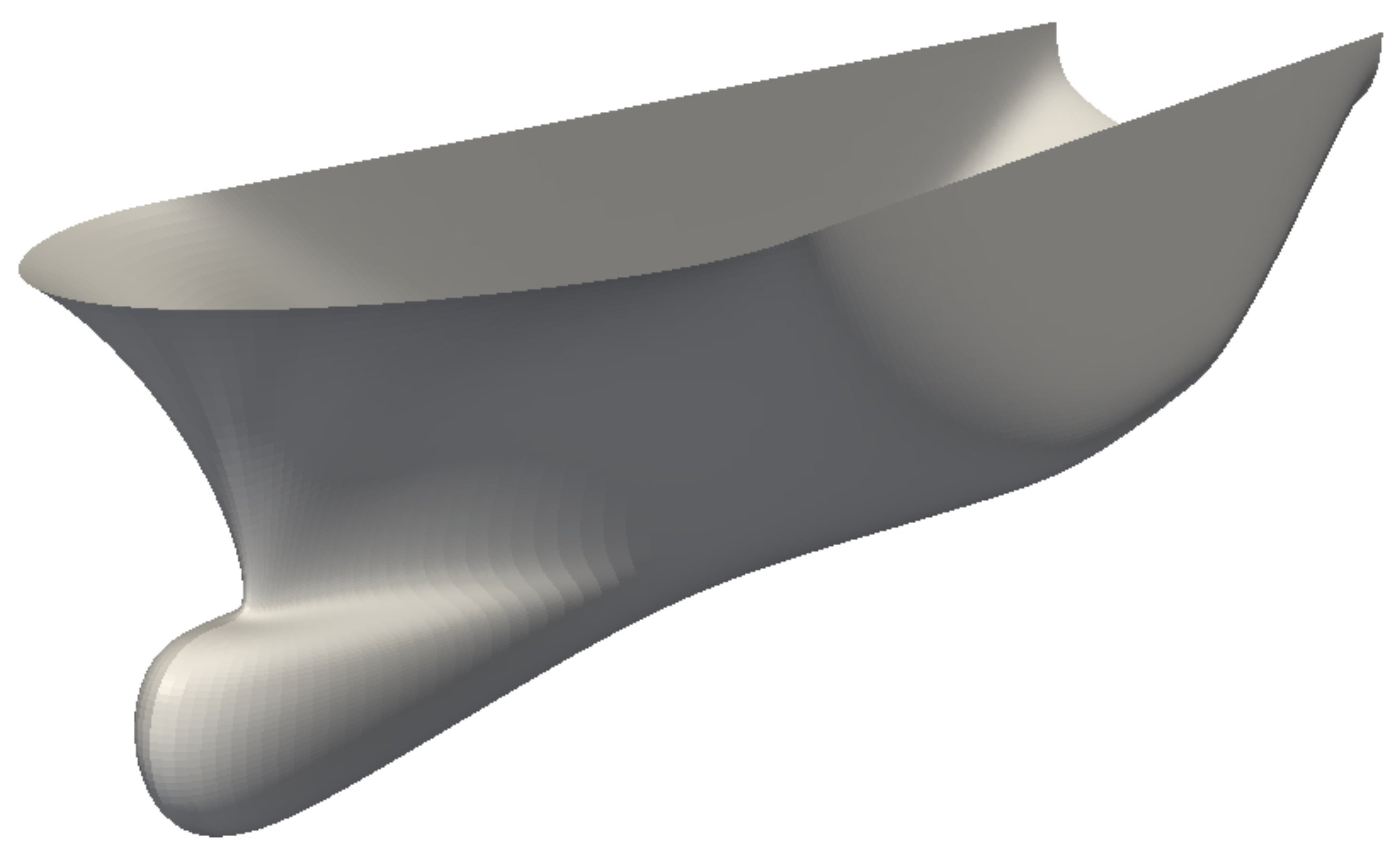}
\end{center}
\end{tabularx}
\caption{A T-spline container ship hull. The surface is
  $C^2$-continuous everywhere.} 
\label{fig:hull_geom}
\end{center}
\end{figure}

We restrict the refinement region to the locations detailed
in~\cref{fig:hull_effects}. It is assumed that the original design is
too coarse to be used as a basis for analysis and additional
resolution is required in the rectangular region followed by highly
localized refinements along the region corresponding to the curve. 

\begin{figure}[htb]
\begin{center}
\includegraphics[scale=0.65]{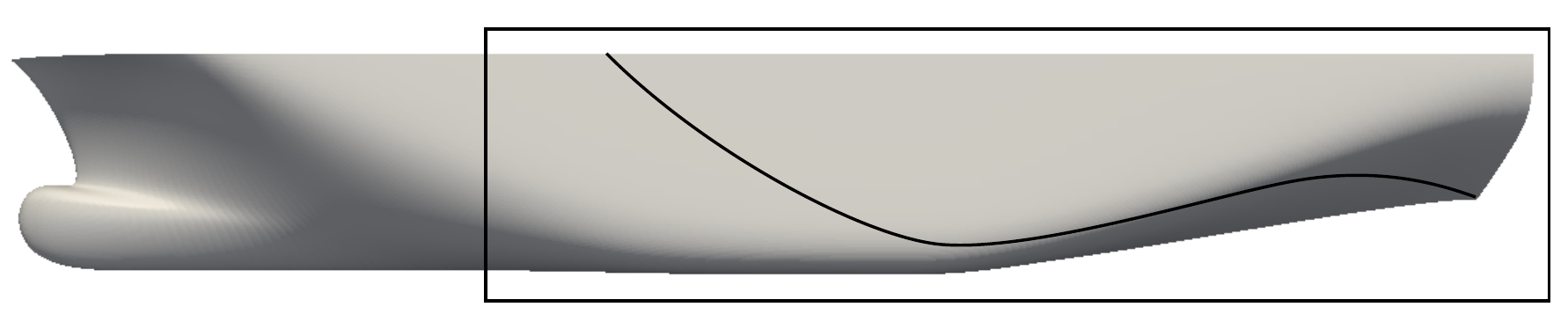}
\caption{The regions of the container ship hull where
  local refinement will be performed. Refinement is first performed in
  the rectangular region followed by highly
  localized refinement along the curve.} 
\label{fig:hull_effects}
\end{center}
\end{figure}

The HASTS refinement algorithm is based on the algorithm presented
in~\cite{ScThEv13} for spline forests. The algorithm is element-based,
meaning refinement is driven by the subdivision of B\'{e}zier
elements. The hierarchical basis is then reextracted into the new
hierarchical T-mesh topology to generate the
new set of B\'{e}zier elements. A detailed description of the underlying
algorithms, in the context of HASTS, will be postponed to a future
publication. Figure~\ref{fig:hrefine2} shows three HASTS local
refinements along the curve shown in Figure~\ref{fig:hull_effects}. 
The elements are colored according 
to their level, $\alpha$, in the hierarchy. Note that \textit{no} nonlocal
propagation of local refinement occurs for HASTS. Only those
elements specified for refinement are subdivided. This is possible due
to the relaxation of the single level constraint inherent in ASTS. The
refinements form a nested sequence of $C^2$-continuous 
hierarchical analysis-suitable T-spline spaces. The geometry of the
hull is exactly preserved during refinement. The final HASTS is composed of $1857$ B\'{e}zier
elements and $1193$ basis functions. However, only $75$ geometric
blending functions and control points are used to define the hull
geometry.

\begin{figure}[htb]
\begin{center}
\begin{tabularx}{1\textwidth}{X}
\begin{center}
\includegraphics[scale=0.4]{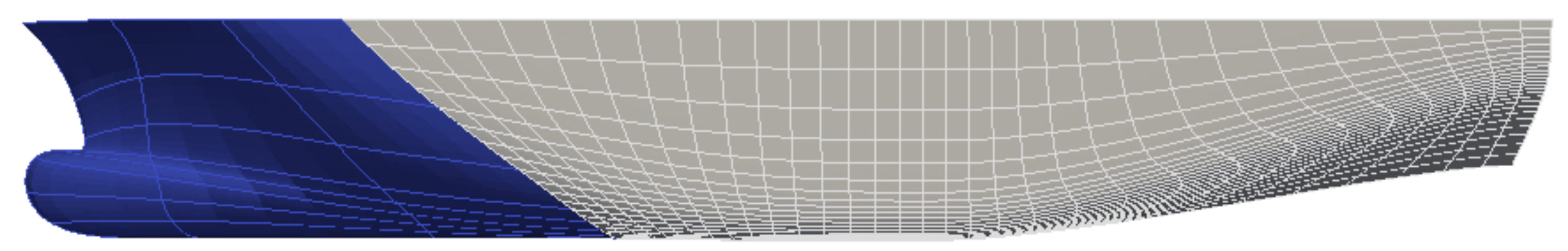}
\end{center}
\\
\begin{center}
\includegraphics[scale=0.4]{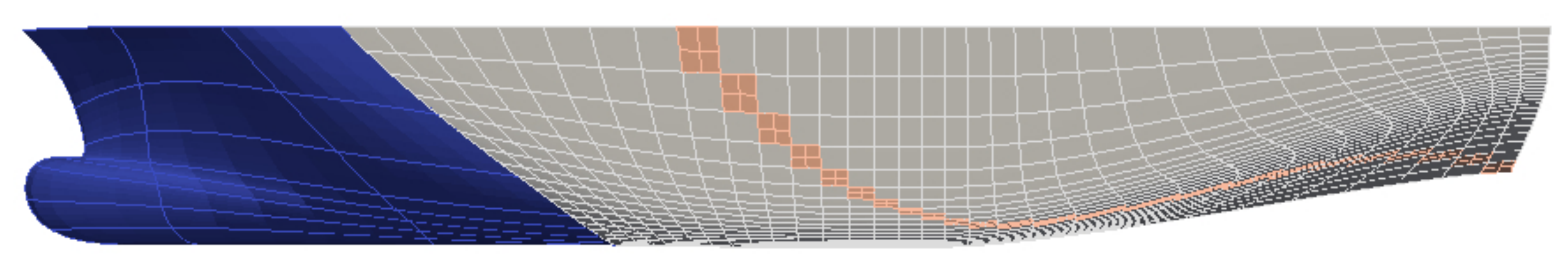}
\end{center}
\\
\begin{center}
\includegraphics[scale=0.4]{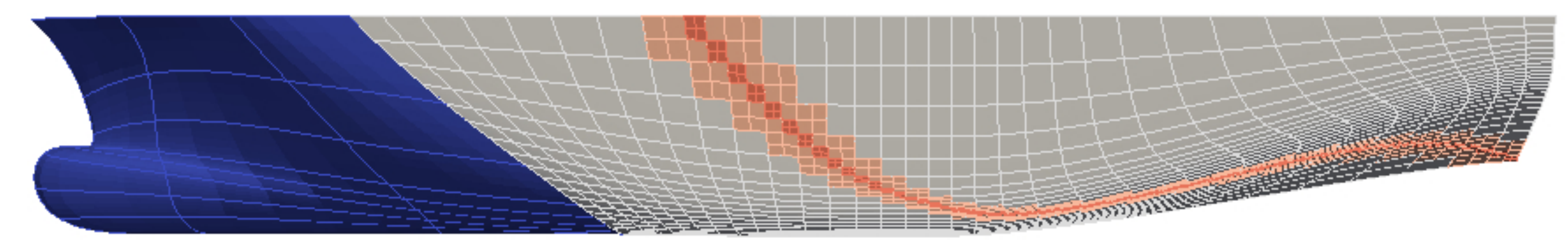}
\end{center}
\\
\begin{center}
\includegraphics[scale=0.4]{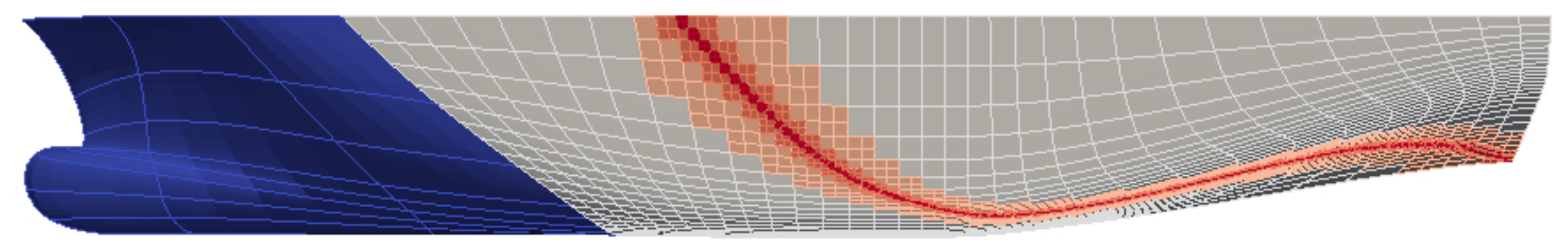}
\end{center}
\end{tabularx}
\caption{Three iterations of HASTS local refinement along the curve
  for the container ship hull in Figure~\ref{fig:hull_geom}.  The
  elements are colored according 
  to their level, $\alpha$, in the hierarchy. The
  refinements form a nested sequence of $C^2$-continuous 
  hierarchical analysis-suitable T-spline spaces. The geometry of the
  hull is exactly preserved during refinement.}
\label{fig:hrefine2}
\end{center}
\end{figure}

As a comparison, Figure~\ref{fig:refine6} shows the results of ASTS local
refinement using the algorithm from~\cite{ScLiSeHu10}. The top figure
shows the control points added during local
refinement (black dots) along the curve. The region selected
for refinement is shown in red. Observe the 
propagation of the control points away from the selected refinement
region. The bottom figure shows the resulting B\'{e}zier elements
after refinement. Superfluous control points and elements are added just to satisfy the
single level constraint inherent in the definition of ASTS.

\begin{figure}[htb]
\begin{center}
\begin{tabularx}{0.9\textwidth}{X}
\begin{center}
\includegraphics[scale=0.3]{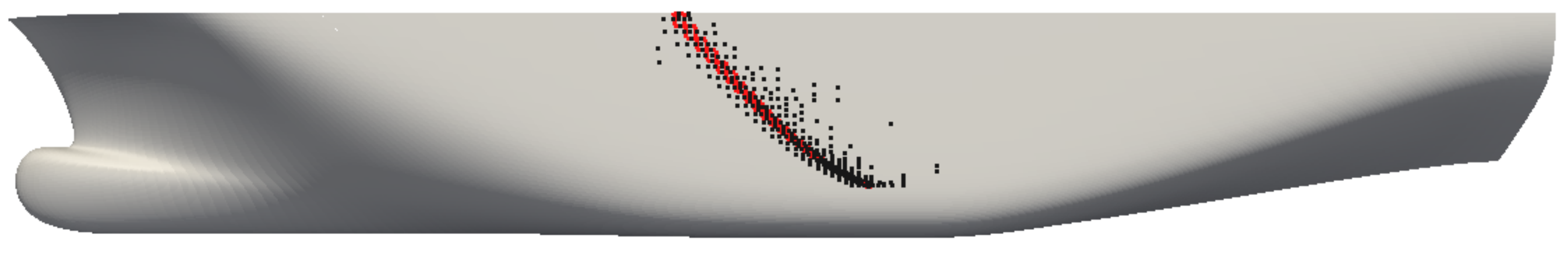}
\end{center}
\\
\begin{center}
\includegraphics[scale=0.3]{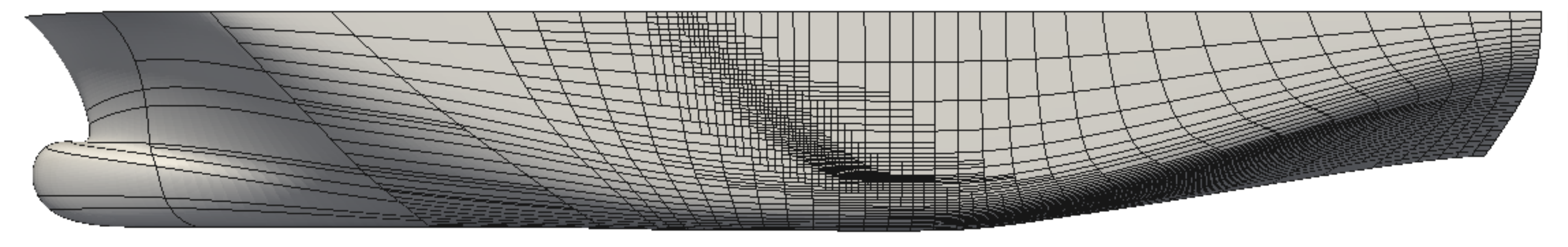}
\end{center}
\end{tabularx}
\caption{The results of ASTS local
refinement using the algorithm from~\cite{ScLiSeHu10}. The top figure
shows the control points added during local
refinement (black dots) along the curve. Observe the 
propagation of the control points away from the selected refinement
region. The bottom figure shows the resulting B\'{e}zier elements
after refinement. Superfluous control points and elements are added just to satisfy the
single level constraint inherent in the definition of ASTS.}
\label{fig:refine6}
\end{center}
\end{figure}
\clearpage
\subsection{HASTS as an adaptive basis}
\label{sec:skew}
We now consider HASTS as an adaptive basis for isogeometric analysis. We
choose as a benchmark the advection skew to the mesh problem
shown in~\Cref{fig:skew}. This problem is advection
dominated, with diffusivity of $10^{-6}$. Along the external boundary,
the boundary conditions are selected such that sharp interior and
boundary layers are present in the solution. In this case, $\theta =
45$ degrees.

\begin{figure}
  \centering
  \includegraphics[width=3in]{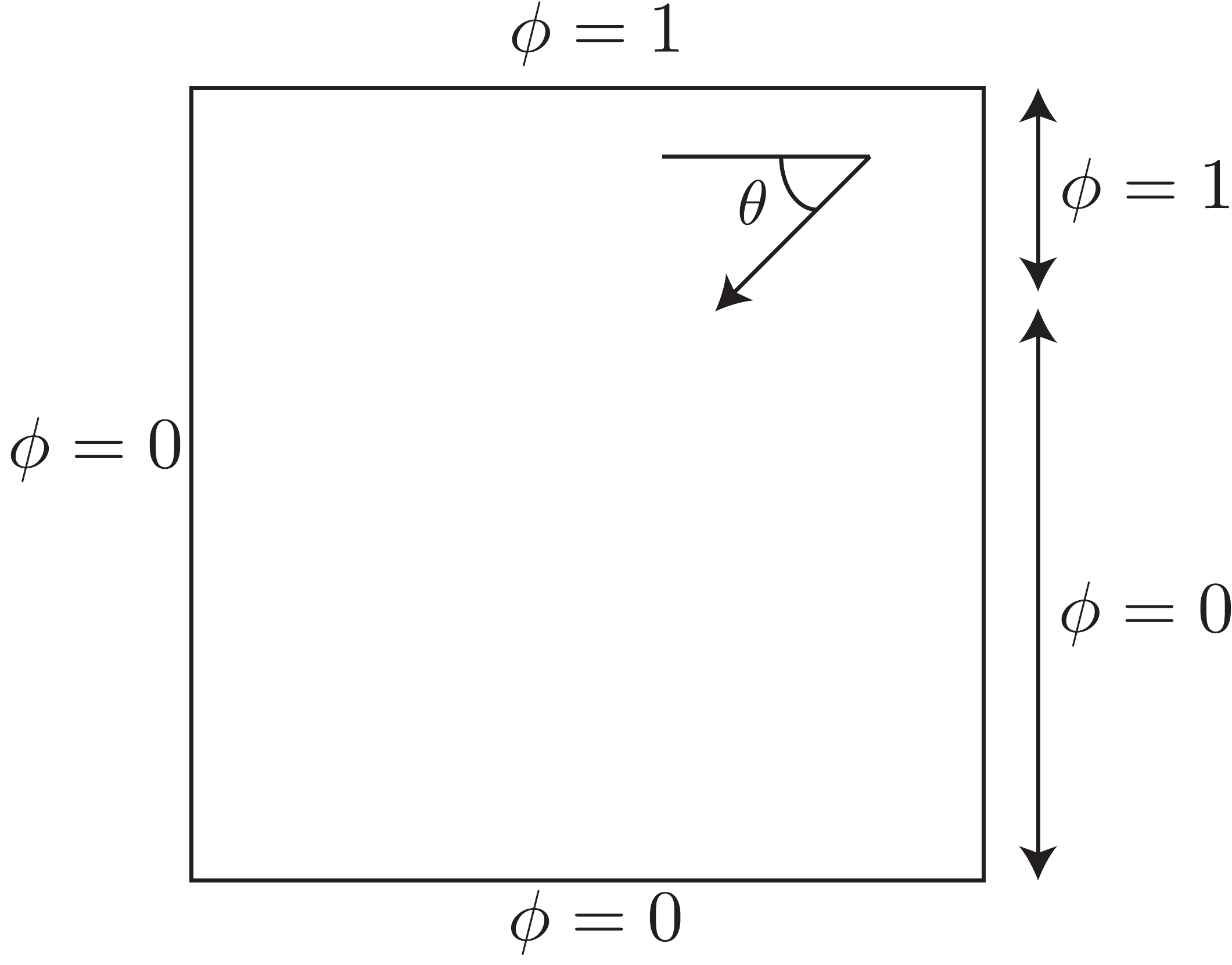}
  \caption{The advection skew to the mesh problem statement.} 
  \label{fig:skew}
\end{figure}
 
\subsubsection{Problem Statement}
Let $\Omega$ be a bounded region in $\mathbb{R}^{2}$ and assume
$\Omega$ has a piecewise smooth boundary $\Gamma$. Let
$\mathbf{x}=\{x_i\}_{i=1}^{2}$ denote a general point in
$\overline{\Omega}$, and let the temperature at a point $\mathbf{x} \in \Omega$ 
be denoted by $\phi(\mathbf{x}) \in \mathbb{R}$. Given
Dirichlet boundary data, $g : \Gamma \rightarrow \mathbb{R}$, 
the steady-state advection-diffusion boundary value
problem consists of finding the temperature $\phi$ such that
\begin{equation}
  \begin{aligned}
    \mathbf{u} \cdot \nabla \phi - \nabla \cdot
          \left(\boldsymbol{\kappa} \nabla \phi\right) &= 0 \text{  on  $\Omega$} \\
    \phi &= g \text{  on  $\Gamma $} \\
    \label{eq:sstatement}
  \end{aligned}
\end{equation}
where $\mathbf{u}: \Omega \rightarrow \mathbb{R}^2$ and $\boldsymbol{\kappa} :
\Omega \rightarrow \mathbb{R}^{2 \times 2}$ are the spatially
varying solenoidal velocity vector and symmetric, positive-definite, diffusivity
tensor, respectively. Note that in this paper we define
$\boldsymbol{\kappa} = \kappa \delta_{ij}$ where $\kappa$ is a
positive constant called the diffusivity coefficient.
We employ SUPG~\cite{BroHug82} with a standard definition for the
element stabilization parameter, $\tau^e$.

\subsubsection{A residual based error estimator} 
\label{sec:error}
To estimate the error we employ a simple residual-based explicit
estimator based on the variational multiscale theory for
fluids~\cite{Hug95,LaMa05, HaDoMi06,HaDoFu12}. It is given by
\begin{equation}
||\phi'||_{\Omega^e} \approx \tau^e || \mathbf{u} \cdot \nabla \phi - \nabla \cdot
\boldsymbol{\kappa} \nabla \phi ||_{\Omega^e}.
\end{equation}
Note that this error estimator underestimates the error for
diffusion-dominated flows but is adequate for the advection-dominated
benchmark presented in this paper. Using standard
techniques~\cite{Hug00} we use the element scaling
\begin{equation}
r = \frac{h^e_{\alpha+1}}{h^e_{\alpha}} =
\left(\frac{tol}{||\phi'||_{\Omega^e}}\right)^{\sfrac{1}{\beta}} 
\end{equation}
where $h^e_{\alpha}, h^e_{\alpha+1}$ are the mesh size distributions for
$\mathsf{T}^{\alpha}$ and $\mathsf{T}^{\alpha+1}$, respectively, and $\beta$ is
the order of convergence of the method. The element size, $h^e$, is
the square root of the element area.  We flag elements
for refinement if $r < 1$. The adaptive process is repeated until a 
specified convergence
tolerance is attained or a maximum number of hierarchical levels are
introduced.

\subsection{Results}
We solve the problem with $C^1$ biquadratic and $C^2$ bicubic 
hierarchical T-splines. The initial T-mesh for both cases
is shown in~\Cref{fig:skew_static_mesh}. Note that the initial
T-mesh is locally refined to accommodate the presence of sharp boundary
layers in the solution. Note that this refinement is \textit{not}
hierarchical. We have found that judiciously performing local refinement of the first
level of a T-spline
hierarchy to accommodate geometric features or boundary conditions
leads to smaller hierarchies and more efficient solution 
procedures.

During each adaptive step the error is assessed as described in~\cref{sec:error}, elements are flagged
for refinement and subdivided, and a new hierarchical basis is then extracted into the new
hierarchical T-mesh topology. This generates a refined set of B\'{e}zier elements. The sequence
of B\'{e}zier mesh refinements is shown for both the biquadratic and 
bicubic case in~\Crefrange{fig:skew_static_mesh_bq}{fig:skew_static_mesh_bc}. 
The sequence of biquadratic refinements form a nested sequence of $C^1$-continuous 
HASTS spaces, whereas the sequence of bicubic refinements form a
nested sequence of $C^2$-continuous HASTS spaces. Note
that fewer elements are required for convergence as the smoothness and order of the
basis increases~\cite{ScThEv13}.

To illustrate the structure and distribution of the hierarchical basis the Greville
abscissae~\cite{ScSiEvLiBoHuSe12} are plotted
in~\Crefrange{fig:skew_static_mesh_funconly_bq}{fig:skew_static_mesh_funconly_bc}. Note
that a linear parameterization was employed for all meshes and the
level zero control points and blending functions define the
geometry. The dots are scaled according to their level in the
hierarchy; a larger dot denotes a lower level. The sequence of solutions are shown
in~\Crefrange{fig:skew_static_sol_bq}{fig:skew_static_sol_bc}.

\begin{figure}
\begin{center}
\includegraphics[scale=0.37]{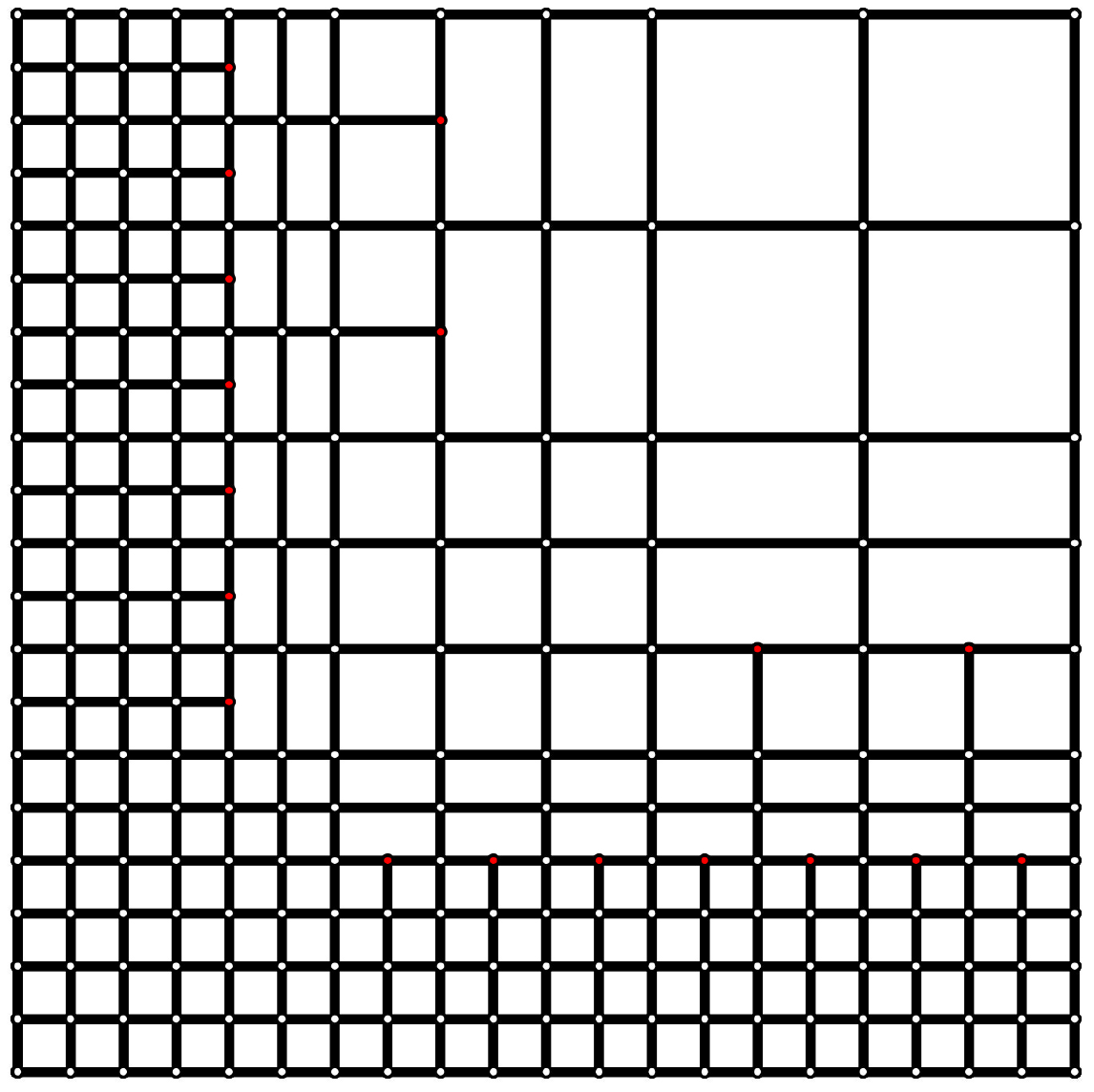}
\end{center}
\caption{Initial T-mesh for the static advection skew to
  the mesh problem.} 
\label{fig:skew_static_mesh}
\end{figure}

\begin{figure}
\begin{center}
{\begin{tabularx}{0.9\textwidth}{XX}
\begin{center}
\includegraphics[scale=0.17]{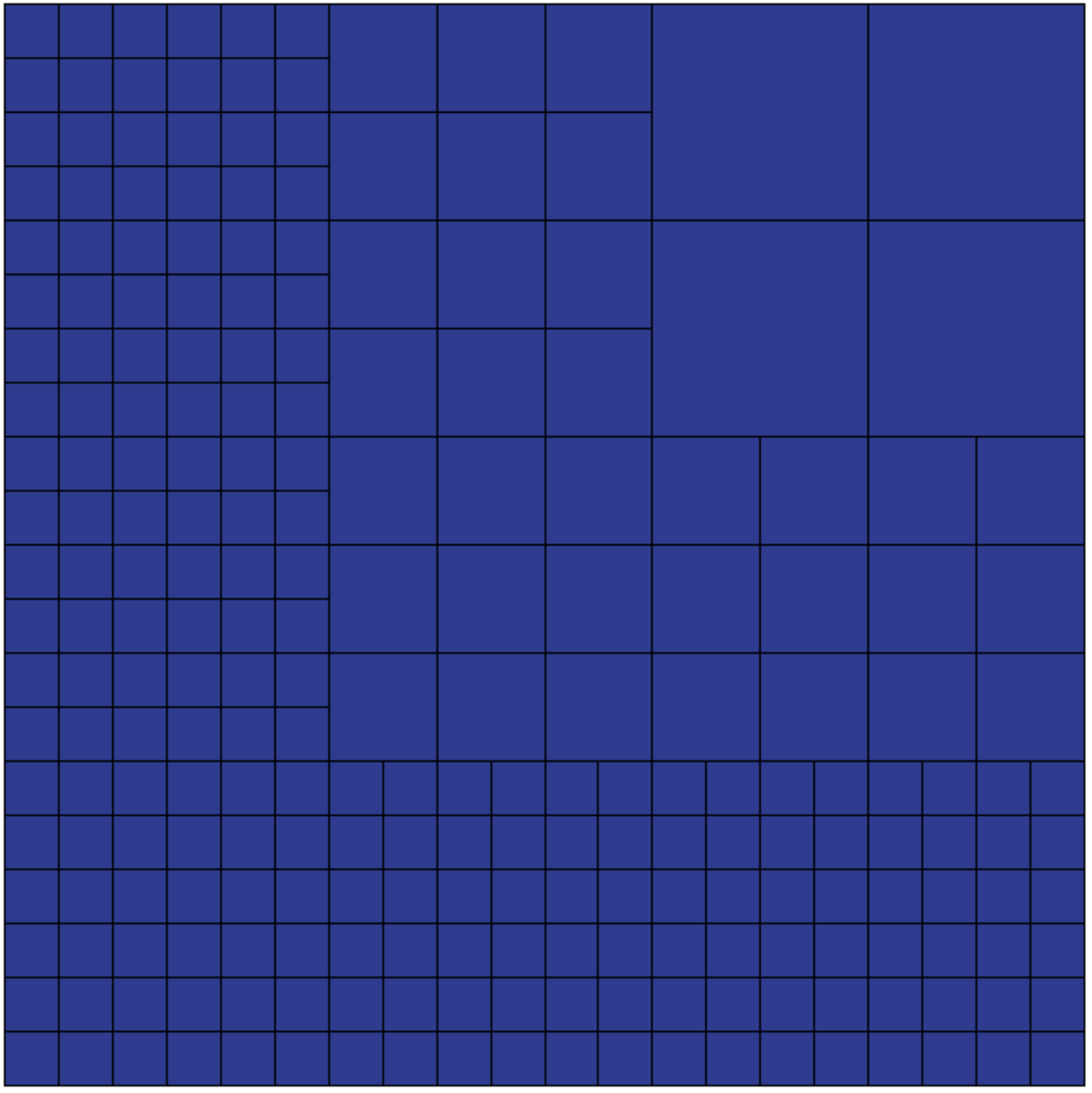}
\end{center} 
&
\begin{center}
\includegraphics[scale=0.17]{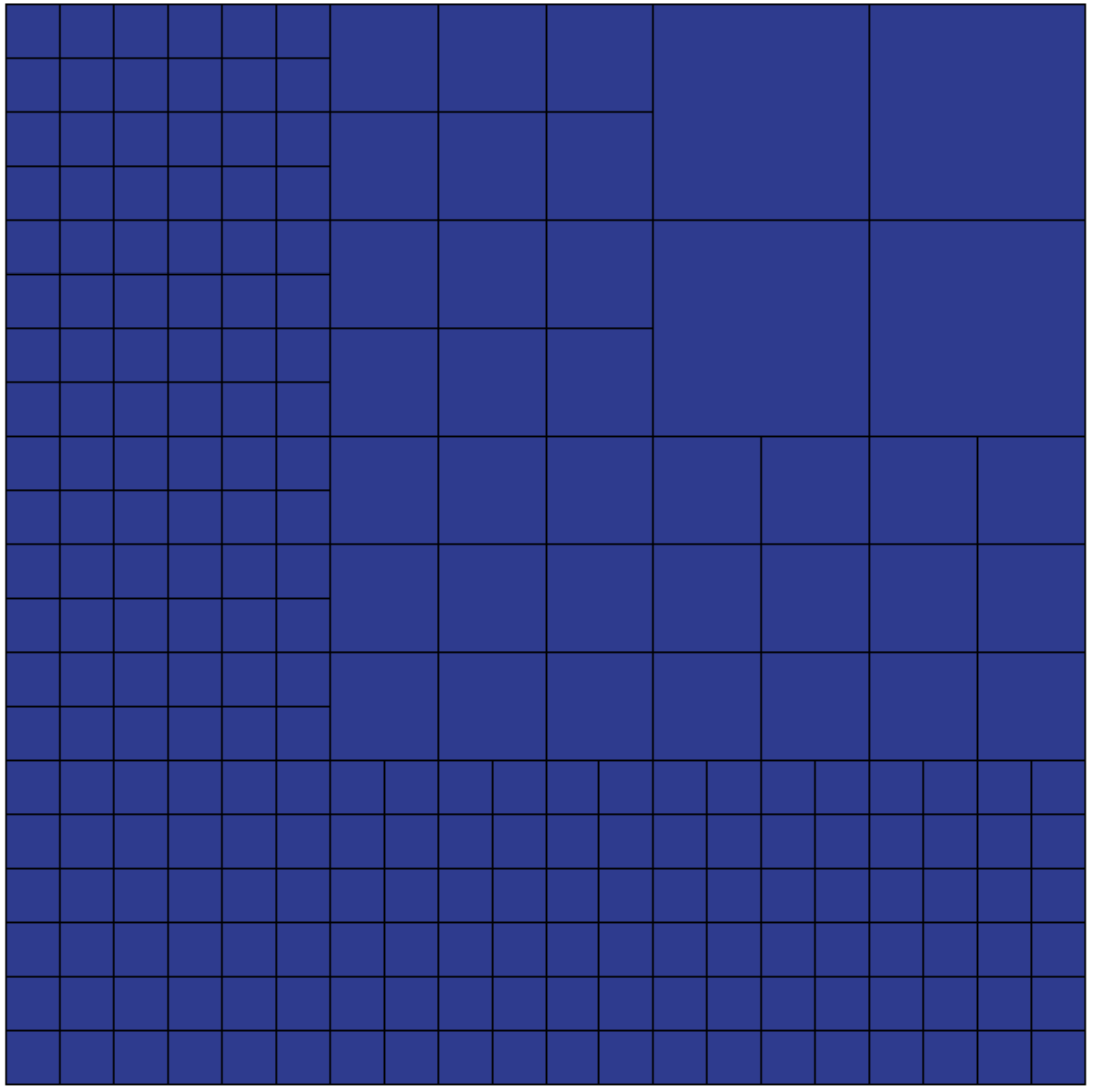}
\end{center}
\\
\vspace{-15pt}
\begin{center}
(a) Initial biquadratic B\'{e}zier mesh
\end{center}
&
\vspace{-15pt}
\begin{center}
(b) Initial bicubic B\'{e}zier mesh
\end{center}\\
\begin{center}
\includegraphics[scale=0.17]{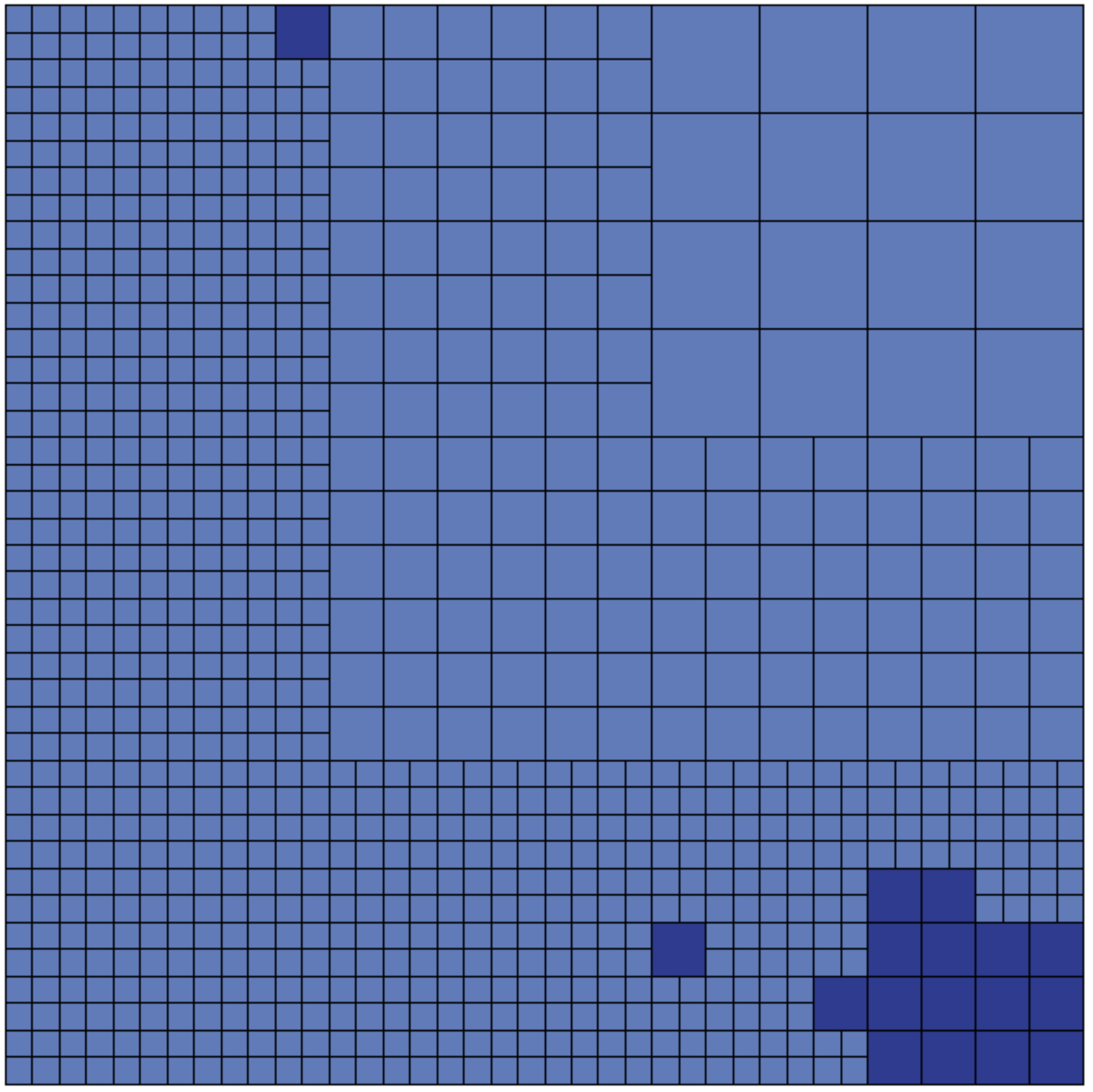}
\end{center} 
&
\begin{center}
\includegraphics[scale=0.17]{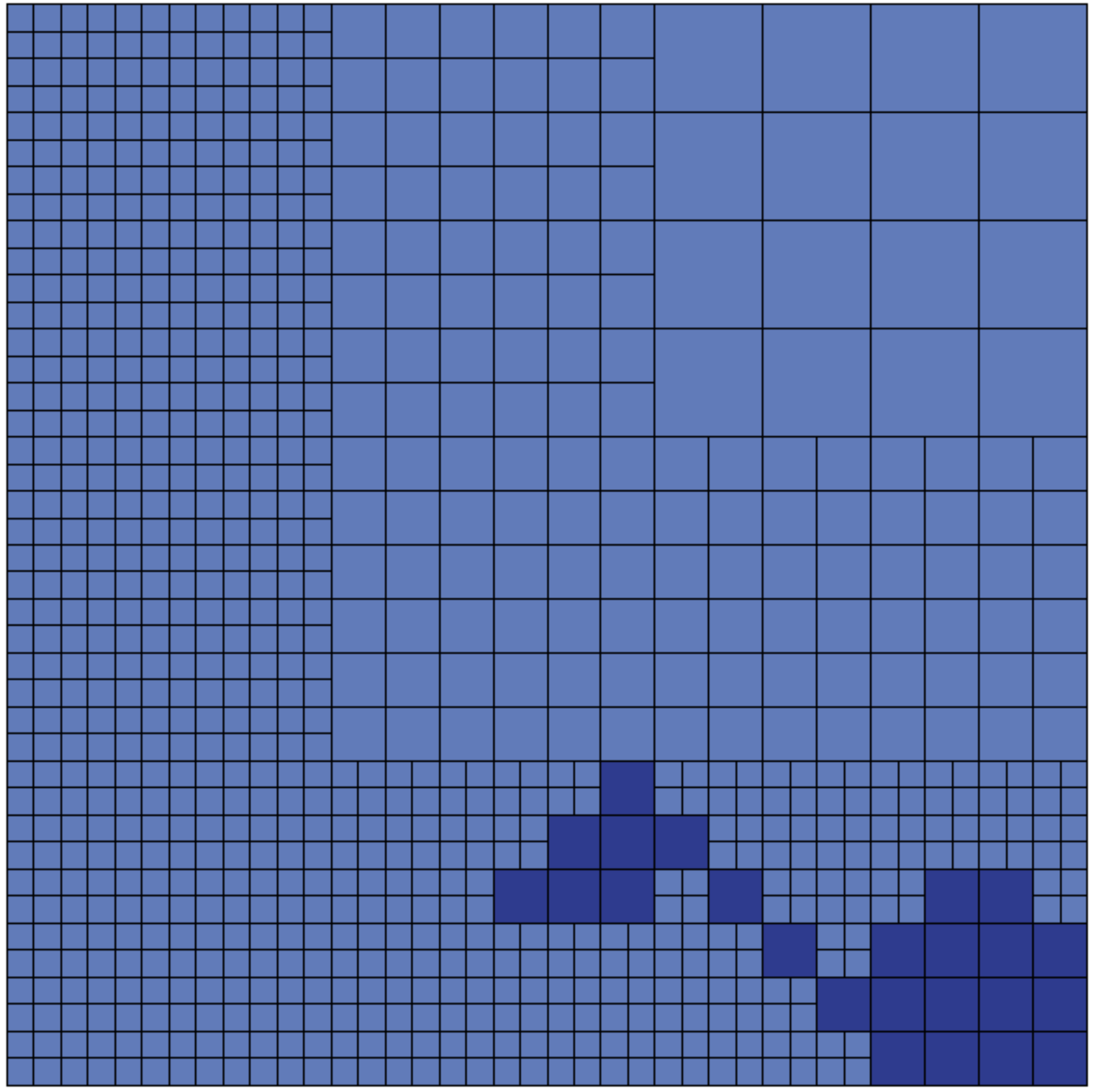}
\end{center}
\\
\vspace{-15pt}
\begin{center}
(c) Second biquadratic B\'{e}zier mesh\end{center}
&
\vspace{-15pt}
\begin{center}
(d) Second bicubic B\'{e}zier mesh
\end{center}\\
\begin{center}
\includegraphics[scale=0.17]{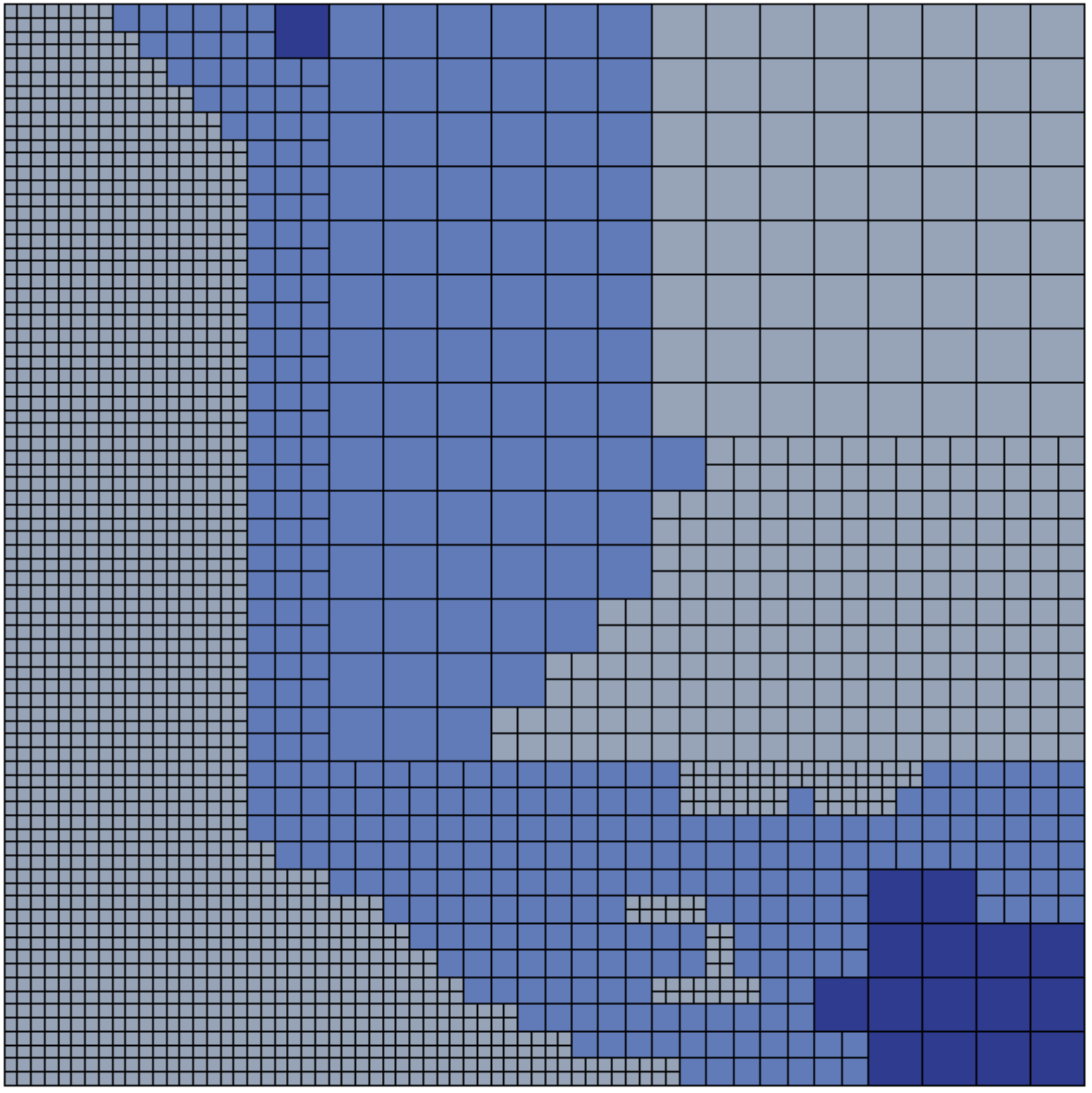}
\end{center} 
&
\begin{center}
\includegraphics[scale=0.17]{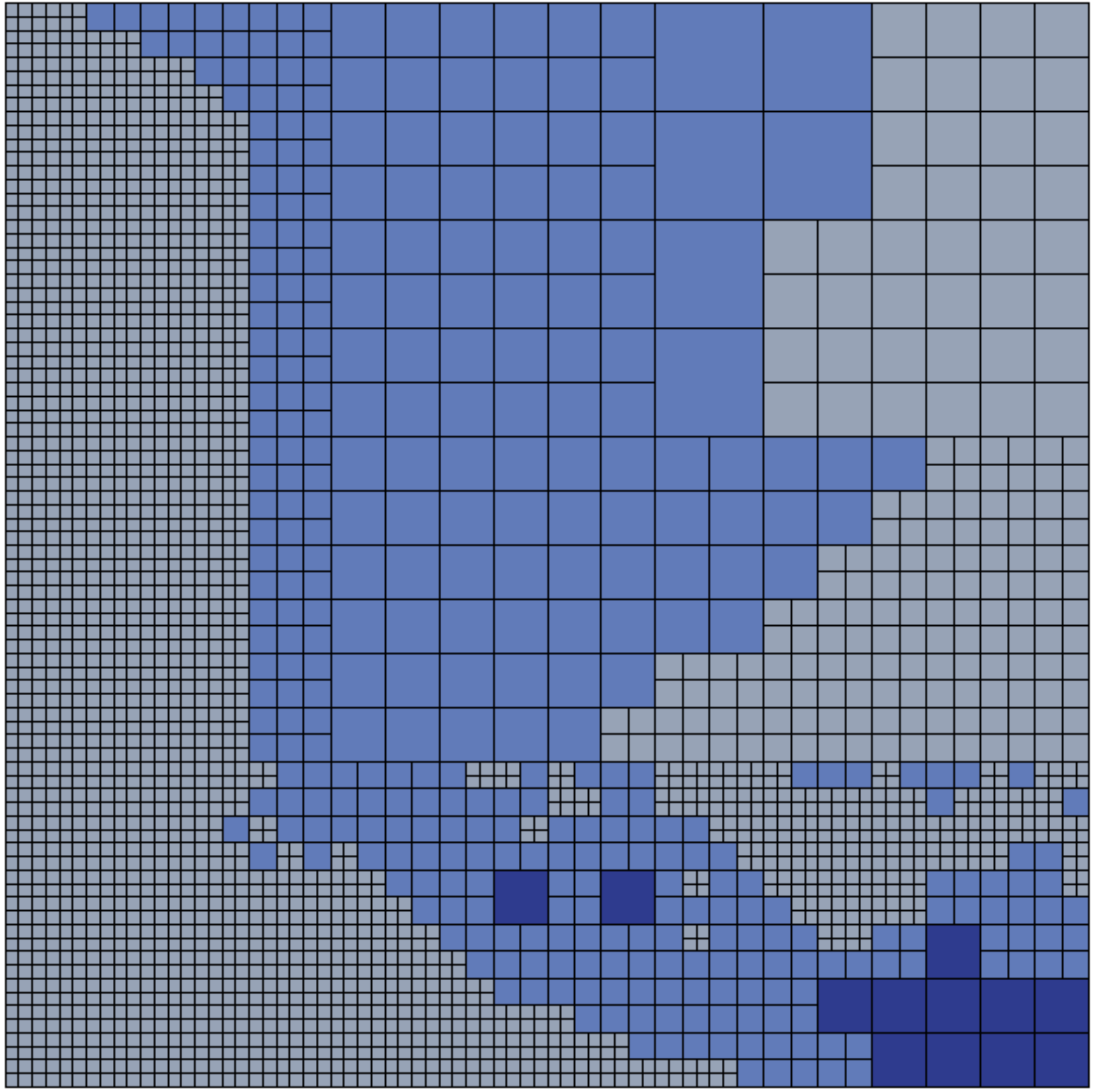}
\end{center}
\\
\vspace{-15pt}
\begin{center}
(e) Third biquadratic B\'{e}zier mesh
\end{center}
&
\vspace{-15pt}
\begin{center}
(f) Third bicubic B\'{e}zier mesh
\end{center}
\end{tabularx} }
\caption{The first two iterations of HASTS local refinement for the 
	biquadratic (left column) and bicubic (right column)
	 advection skew to the mesh problem.  The
  elements are colored according 
  to their level, $\alpha$, in the hierarchy. The
  biquadratic refinements form a nested sequence of $C^1$-continuous 
  HASTS spaces and the
  bicubic refinements form a
  nested sequence of $C^2$-continuous 
  HASTS spaces. }
\label{fig:skew_static_mesh_bq}
\end{center}
\end{figure}

\begin{figure}
\begin{center}
{\begin{tabularx}{0.9\textwidth}{XX}
\begin{center}
\includegraphics[scale=0.17]{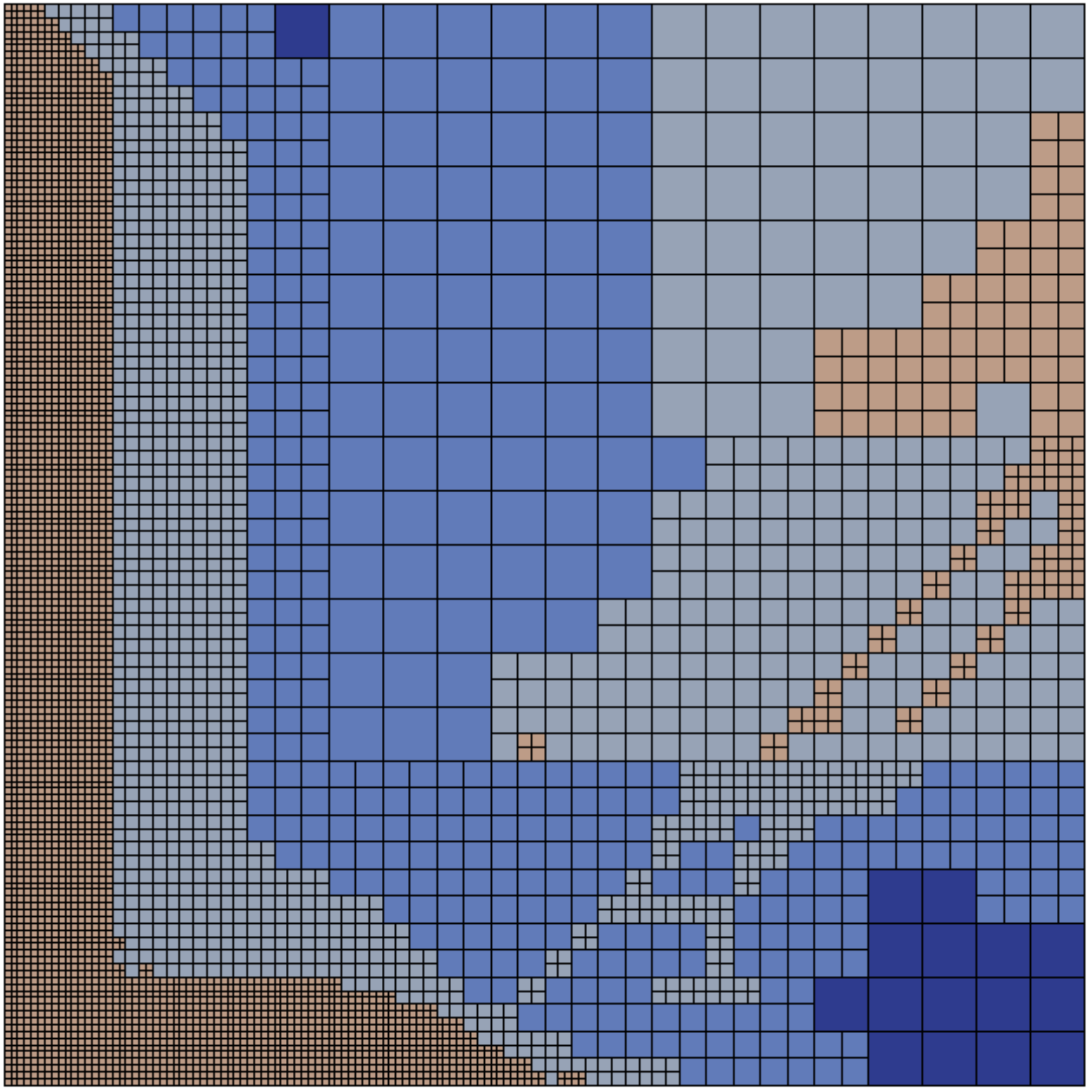}
\end{center} 
&
\begin{center}
\includegraphics[scale=0.17]{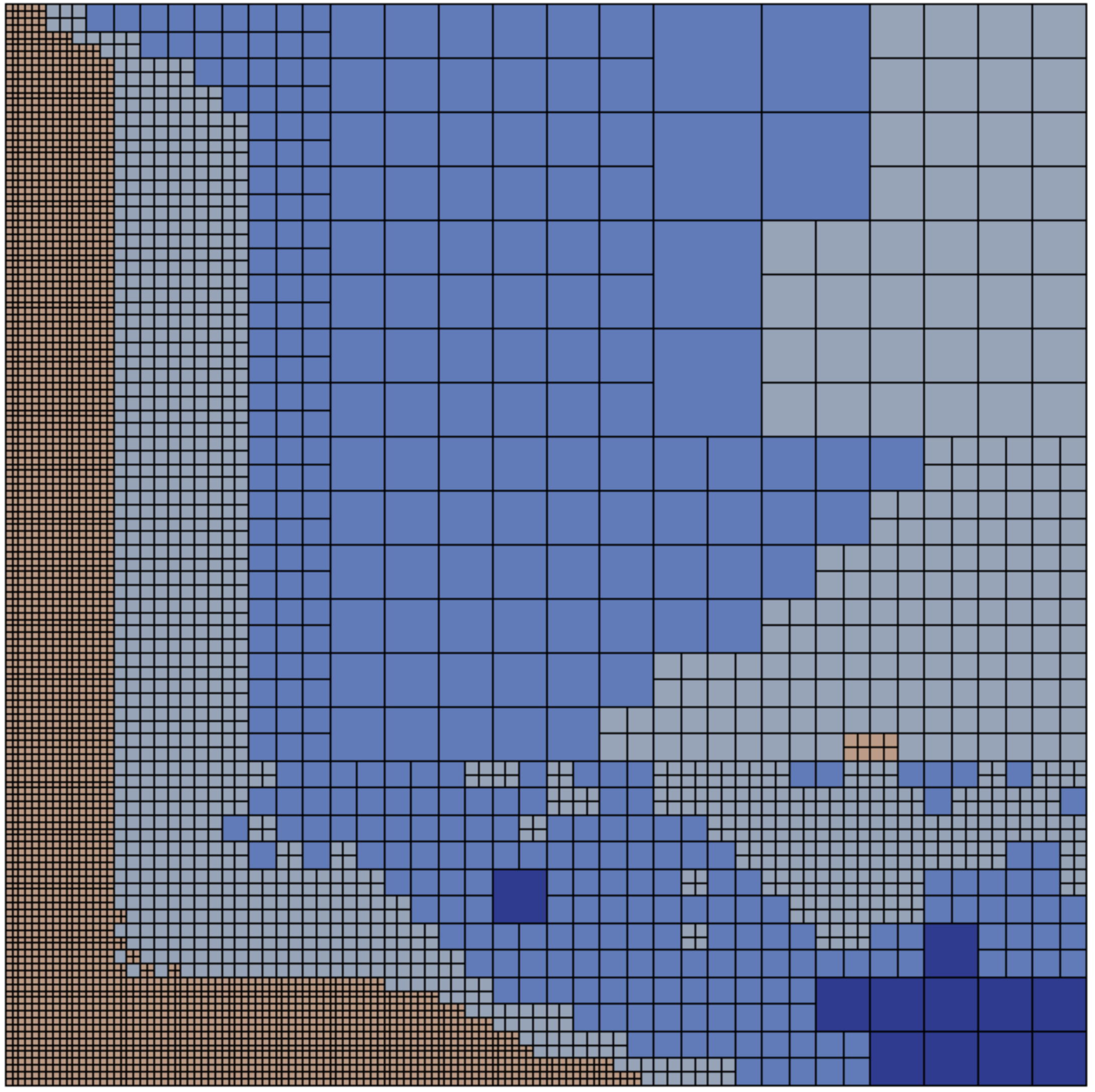}
\end{center}
\\
\vspace{-15pt}
\begin{center}
(a) Fourth biquadratic B\'{e}zier mesh
\end{center}
&
\vspace{-15pt}
\begin{center}
(b) Fourth bicubic B\'{e}zier mesh
\end{center}\\
\begin{center}
\includegraphics[scale=0.17]{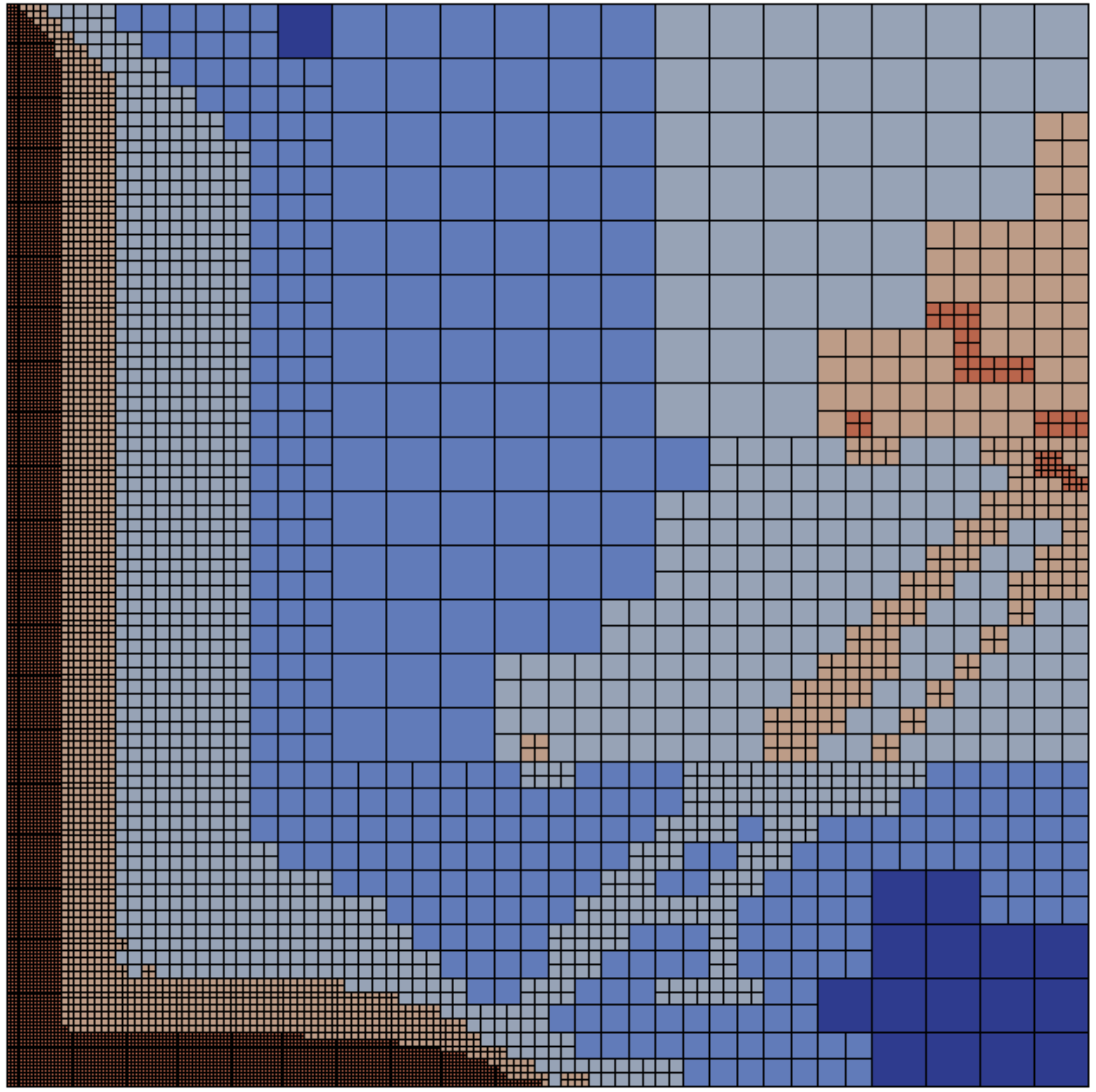}
\end{center} 
&
\begin{center}
\includegraphics[scale=0.17]{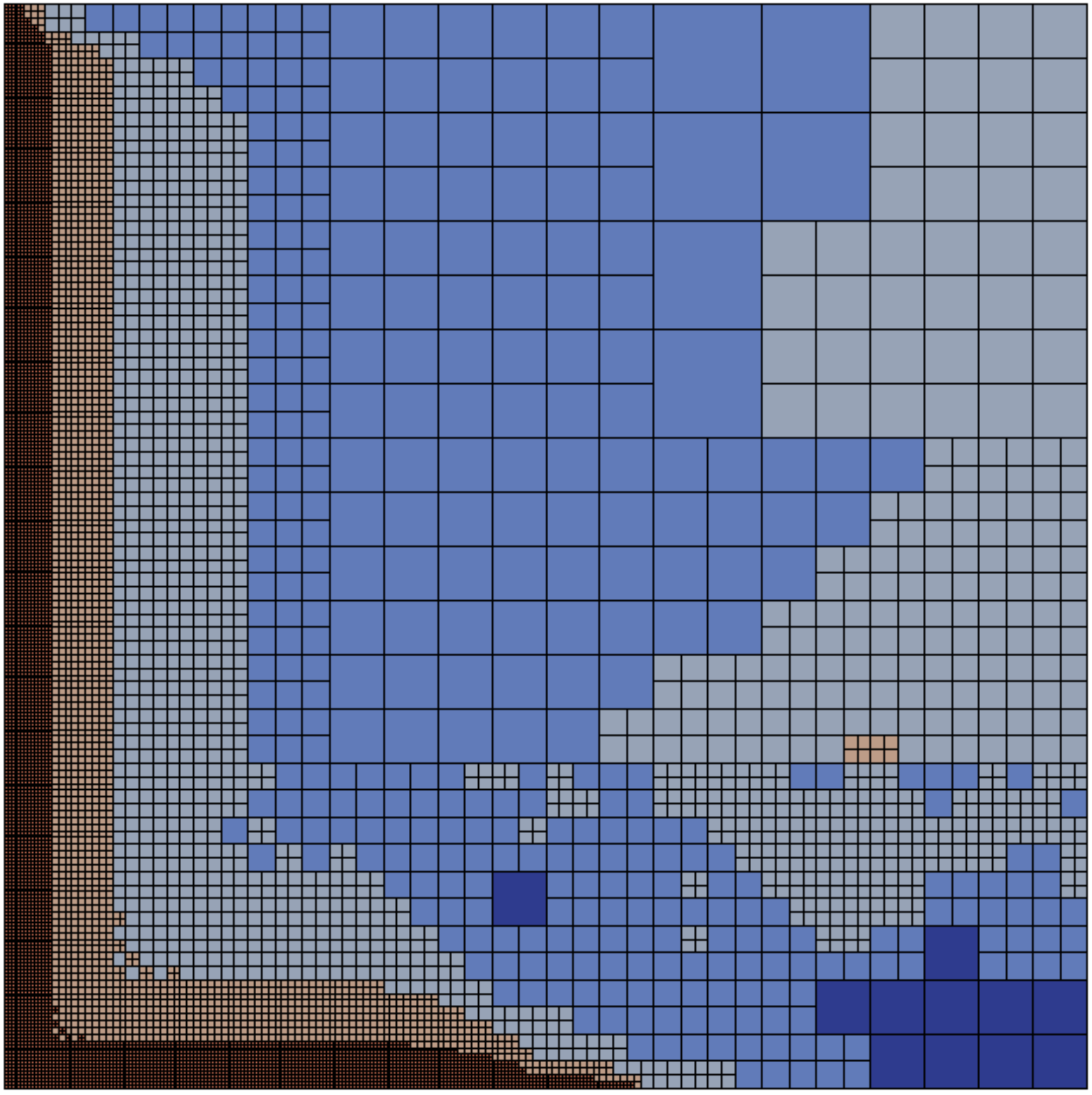}
\end{center}
\\
\vspace{-15pt}
\begin{center}
(c) Fifth biquadratic B\'{e}zier mesh\end{center}
&
\vspace{-15pt}
\begin{center}
(d) Fifth bicubic B\'{e}zier mesh
\end{center}\\
\begin{center}
\includegraphics[scale=0.17]{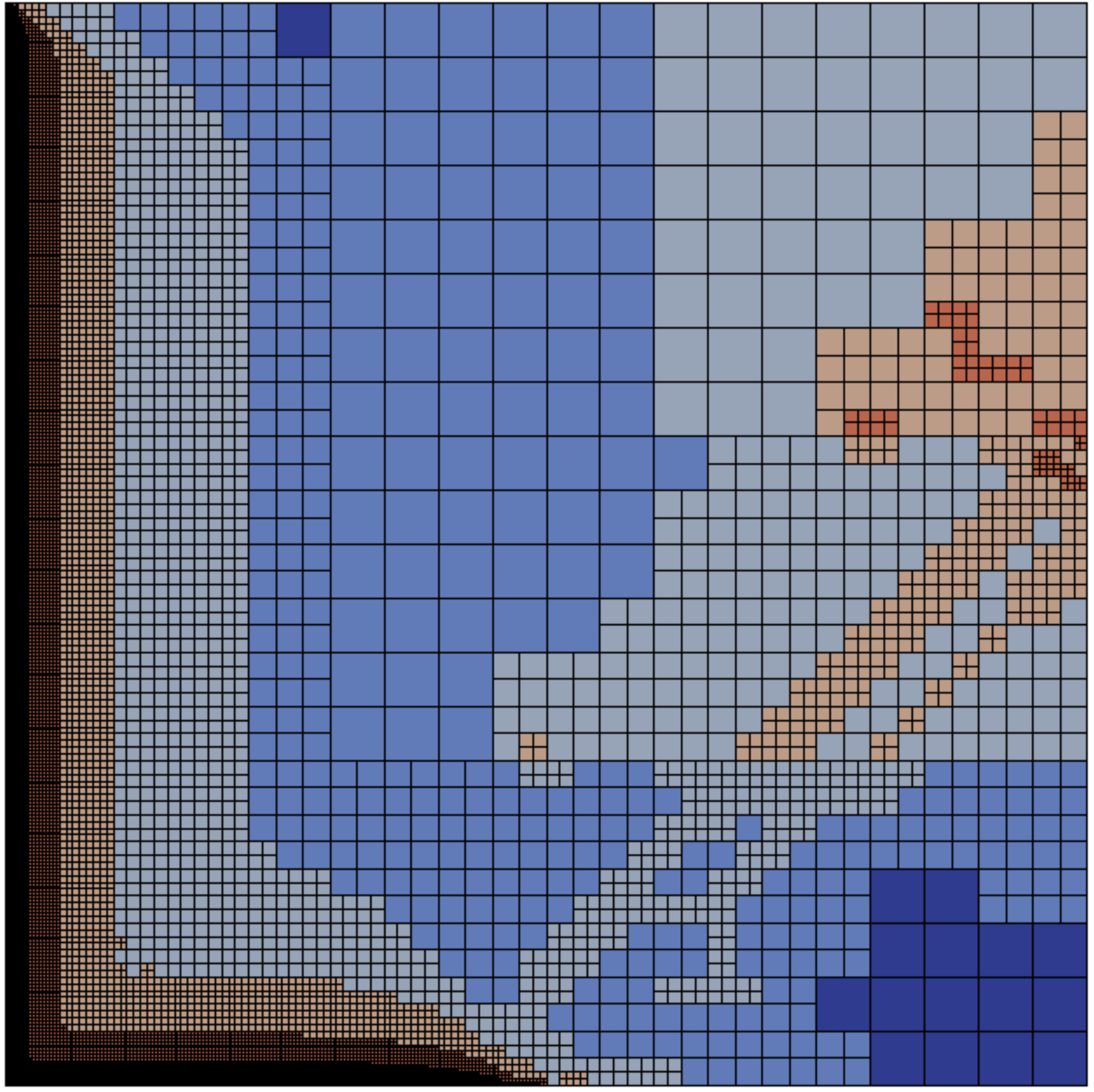}
\end{center} 
&
\begin{center}
\includegraphics[scale=0.17]{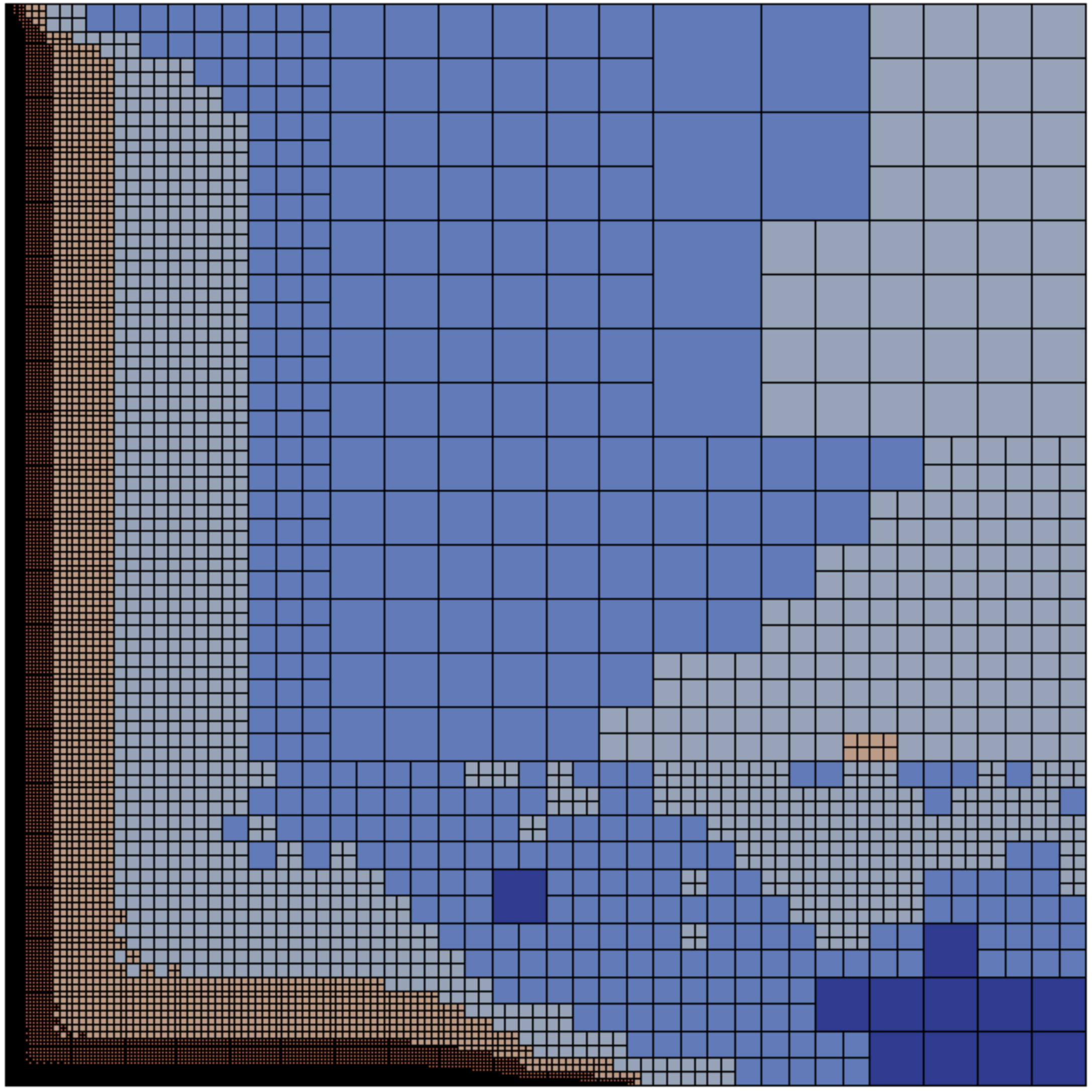}
\end{center}
\\
\vspace{-15pt}
\begin{center}
(e) Sixth biquadratic B\'{e}zier mesh\end{center}
&
\vspace{-15pt}
\begin{center}
(f) Sixth bicubic B\'{e}zier mesh
\end{center}
\end{tabularx} }
\caption{The third through fifth iterations of HASTS local refinement for the 
	biquadratic (left column) and bicubic (right column)
	 advection skew to the mesh problem.  The
  elements are colored according 
  to their level, $\alpha$, in the hierarchy. }
\label{fig:skew_static_mesh_bc}
\end{center}
\end{figure}

\begin{figure}
\begin{center}
{\begin{tabularx}{1\textwidth}{XX}
\begin{center}
\includegraphics[scale=0.23]{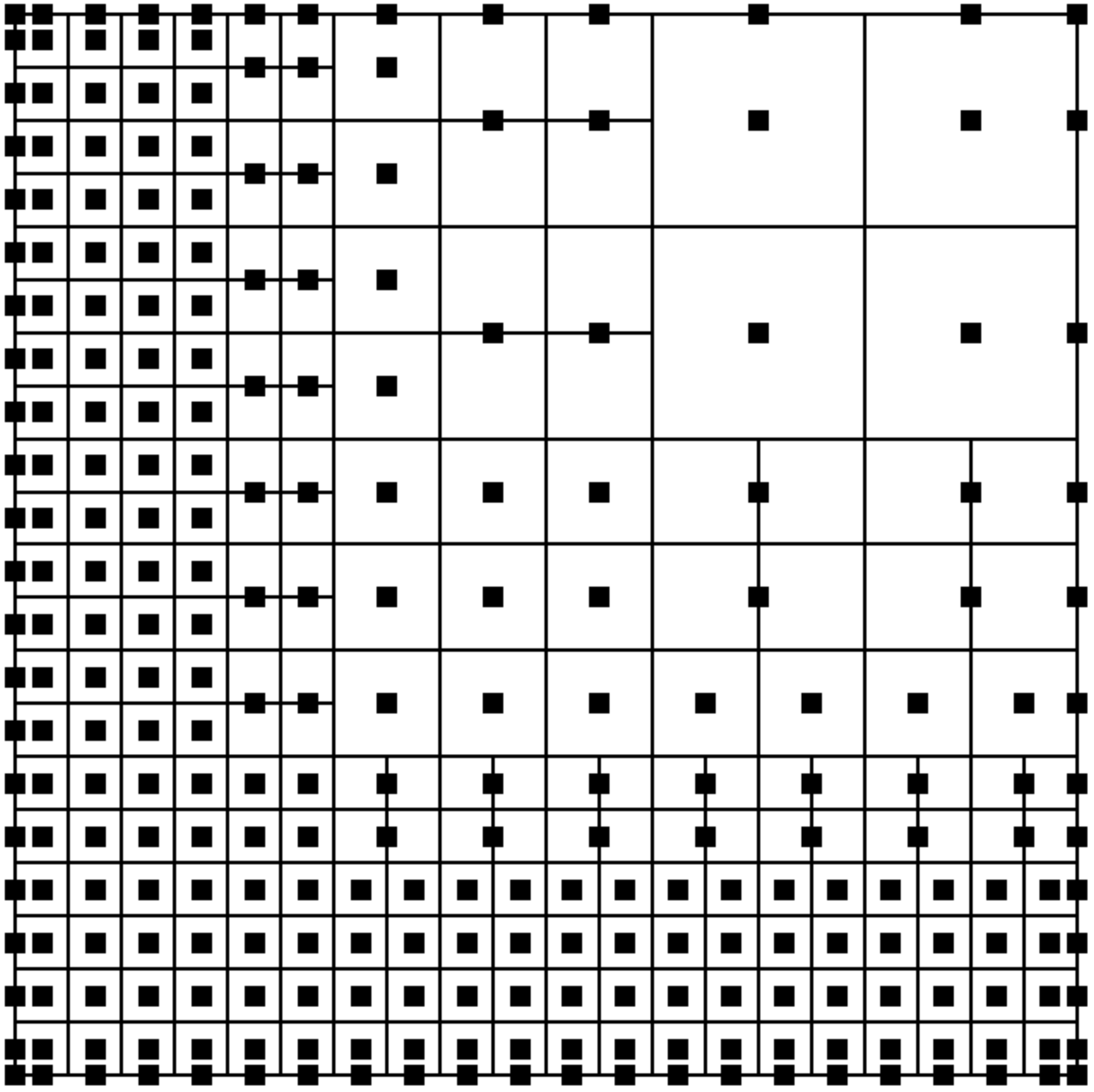}
\end{center} 
&
\begin{center}
\includegraphics[scale=0.23]{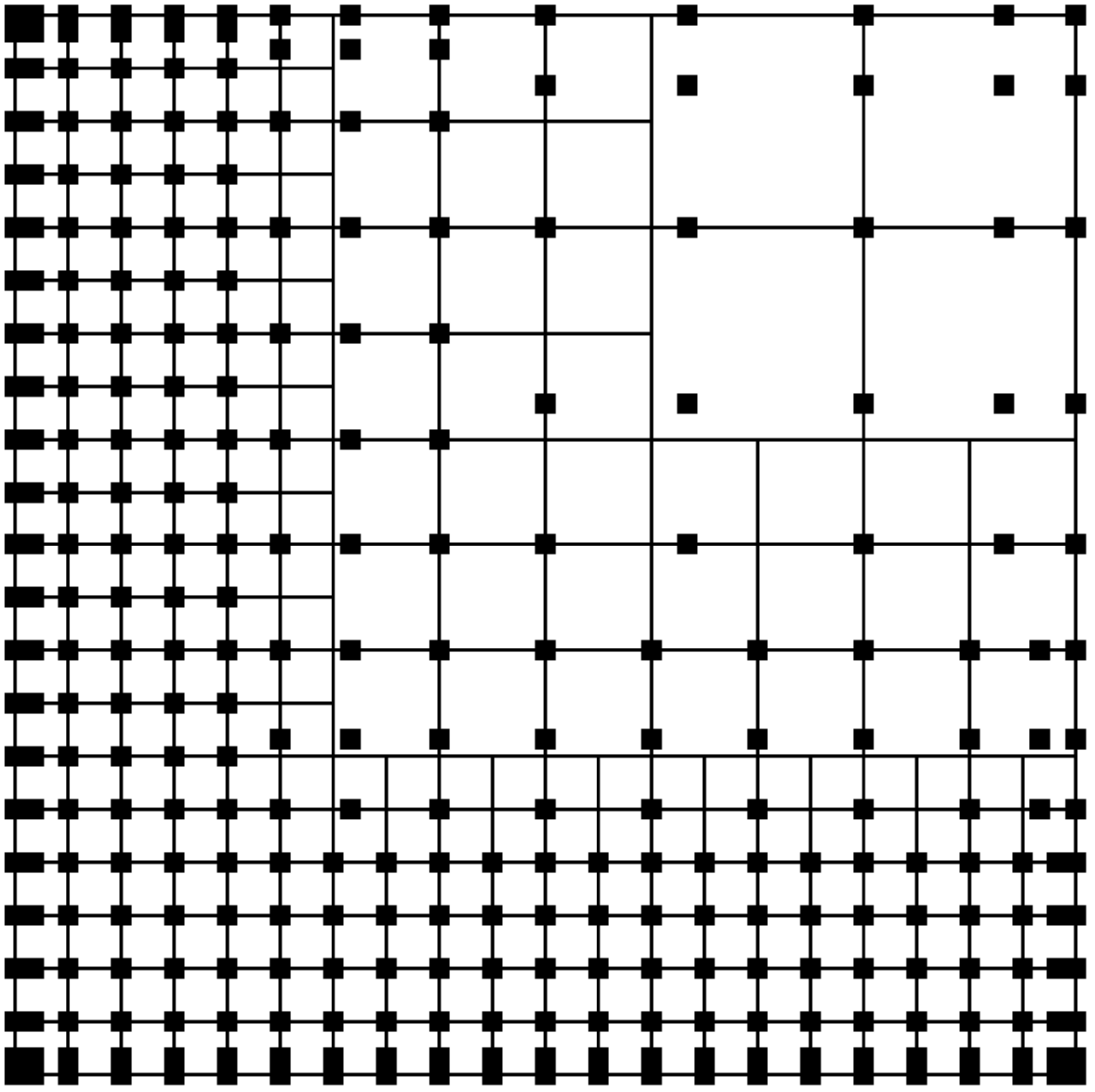}
\end{center}
\\
\vspace{-15pt}
\begin{center}
(a) Initial quadratic Greville abscissae
\end{center}
&
\vspace{-15pt}
\begin{center}
(b) Initial bicubic Greville abscissae
\end{center}
\\
\begin{center}
\includegraphics[scale=0.23]{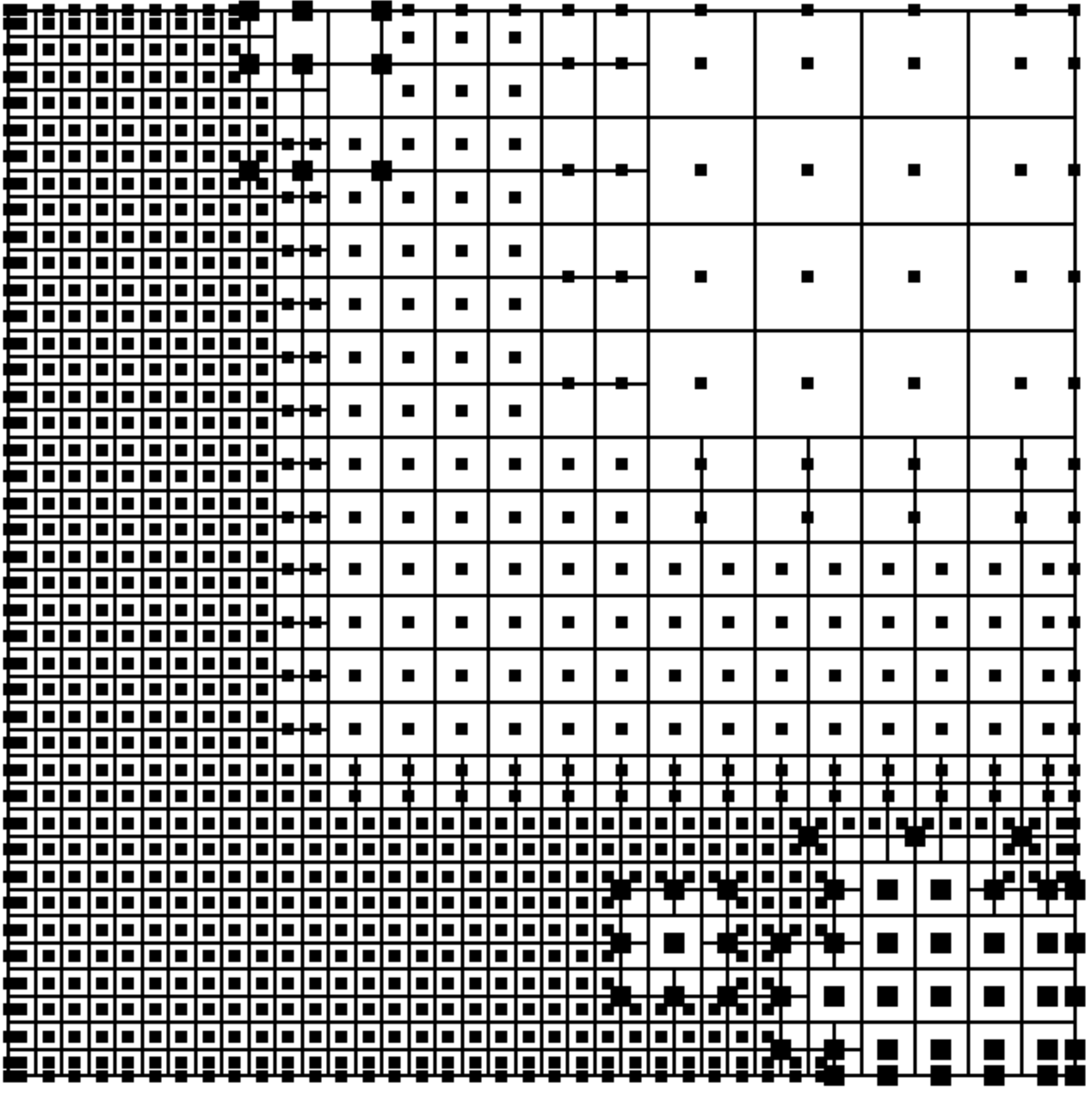}
\end{center} 
&
\begin{center}
\includegraphics[scale=0.23]{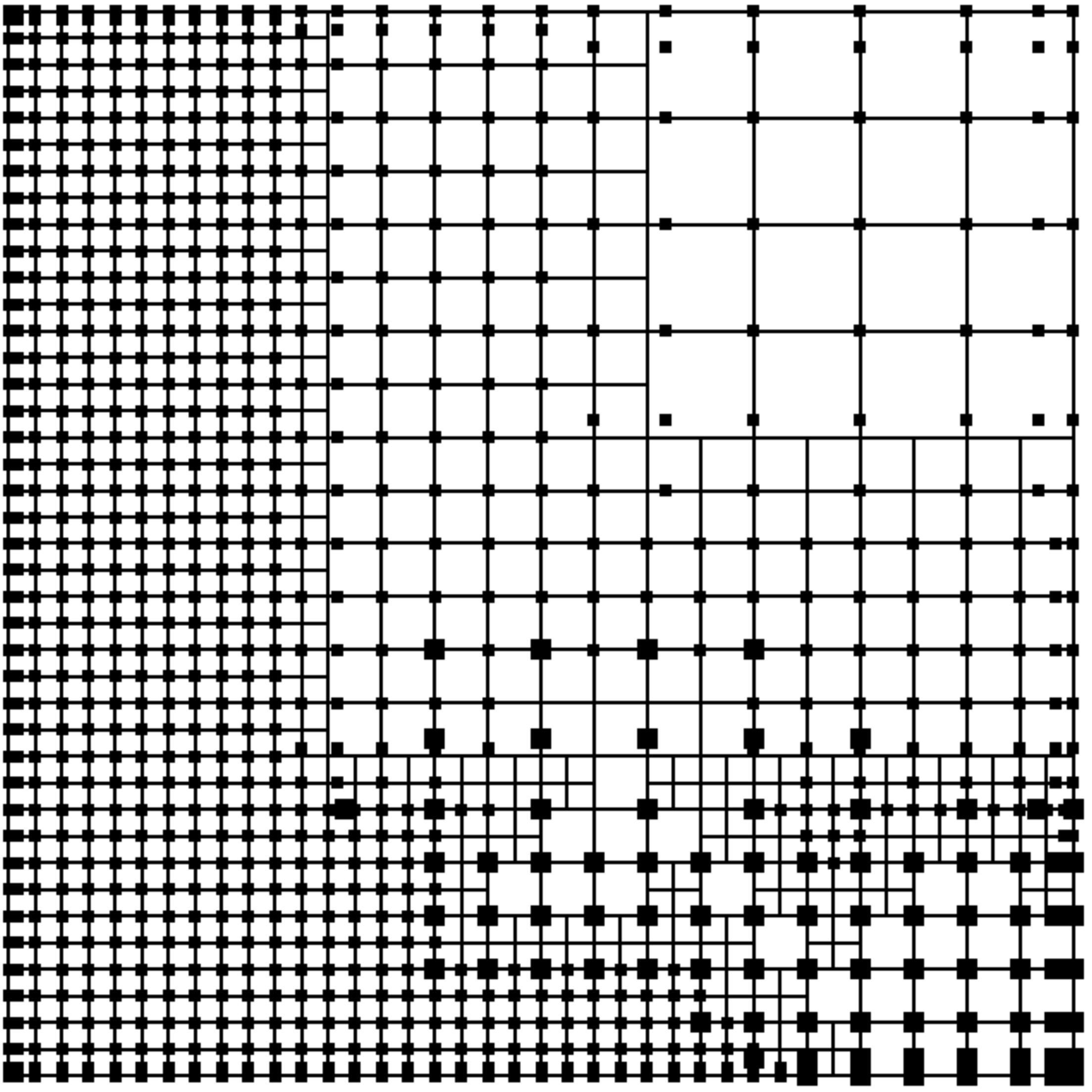}
\end{center}
\\
\vspace{-15pt}
\begin{center}
(c) Greville abscissae for second biquadratic mesh
\end{center}
&
\vspace{-15pt}
\begin{center}
(d) Greville abscissae for second bicubic mesh
\end{center}
\end{tabularx} }
\caption{The Greville abscissae (black dots) corresponding to the basis functions for
  the first two iterations of HASTS local refinement for the 
  biquadratic (left column) and bicubic (right column)
  advection skew to the mesh problem.  The size of the dot corresponds to the 
  basis function level with larger dots corresponding to lower levels. }
\label{fig:skew_static_mesh_funconly_bq}
\end{center}
\end{figure}
\begin{figure}
\begin{center}
{\begin{tabularx}{1\textwidth}{XX}
\begin{center}
\includegraphics[scale=0.23]{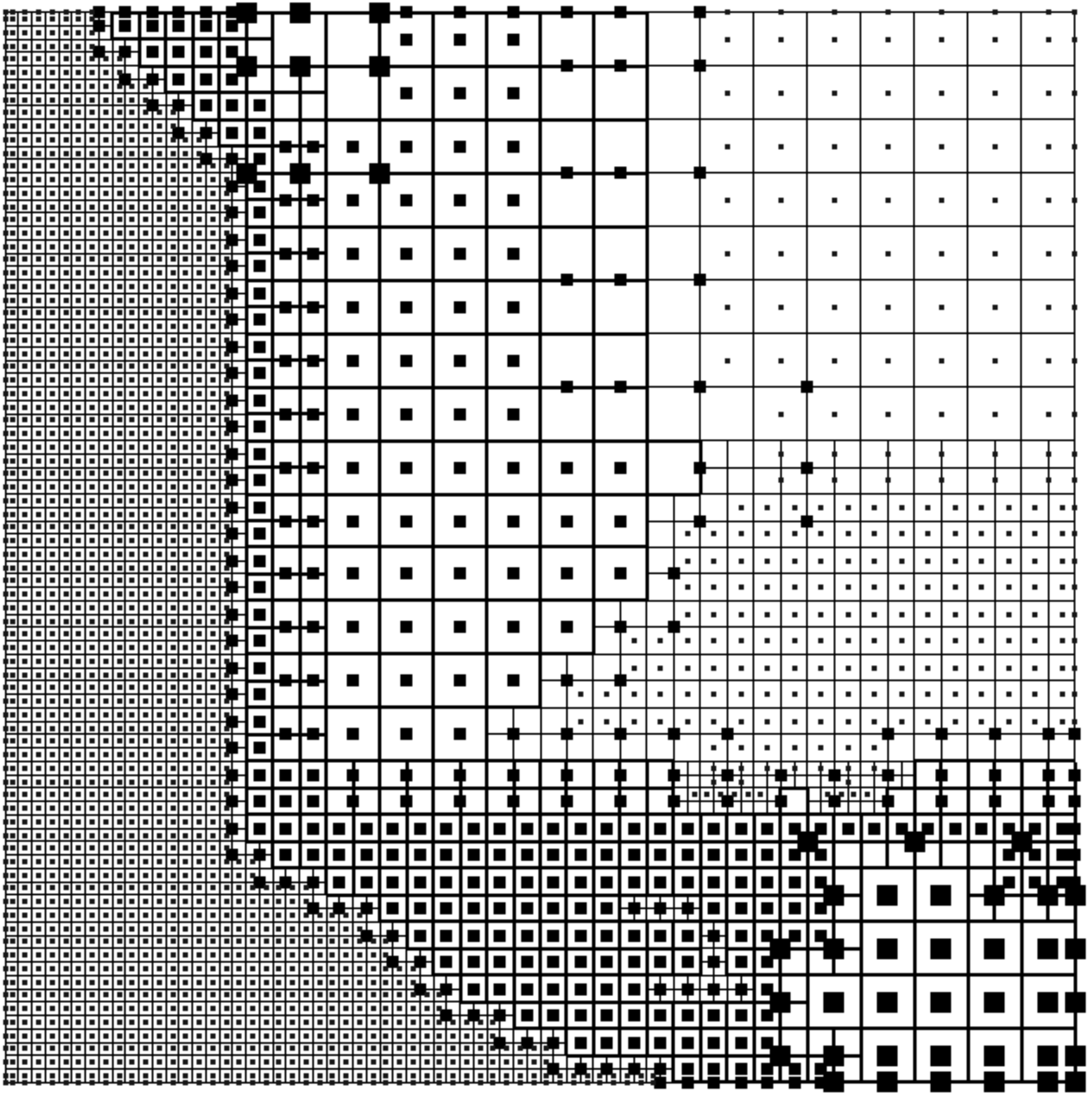}
\end{center} 
&
\begin{center}
\includegraphics[scale=0.23]{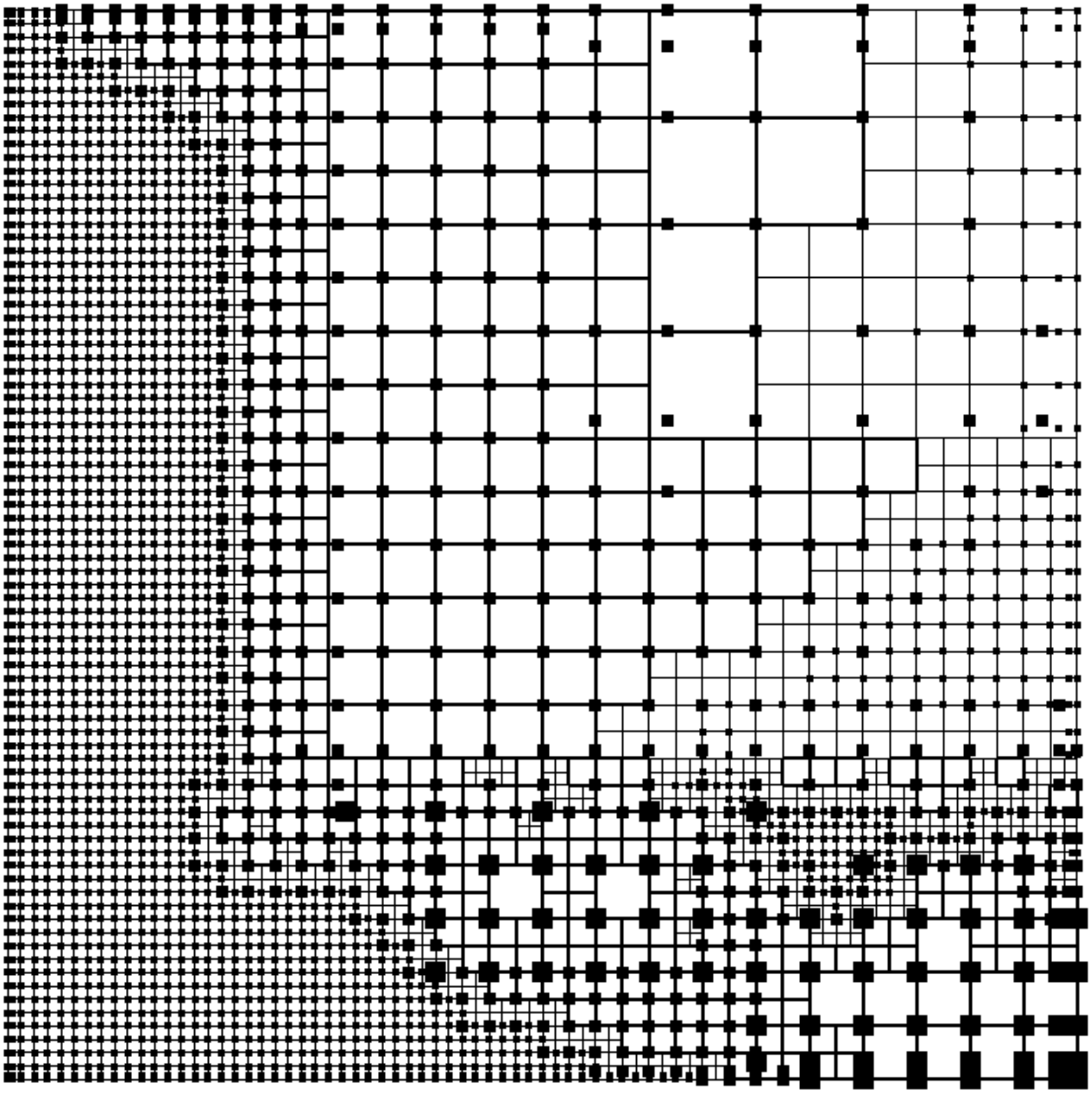}
\end{center}
\\
\vspace{-15pt}
\begin{center}
(a) Greville abscissae for third biquadratic mesh
\end{center}
&
\vspace{-15pt}
\begin{center}
(b) Greville abscissae for third bicubic mesh\end{center}
\\
\begin{center}
\includegraphics[scale=0.23]{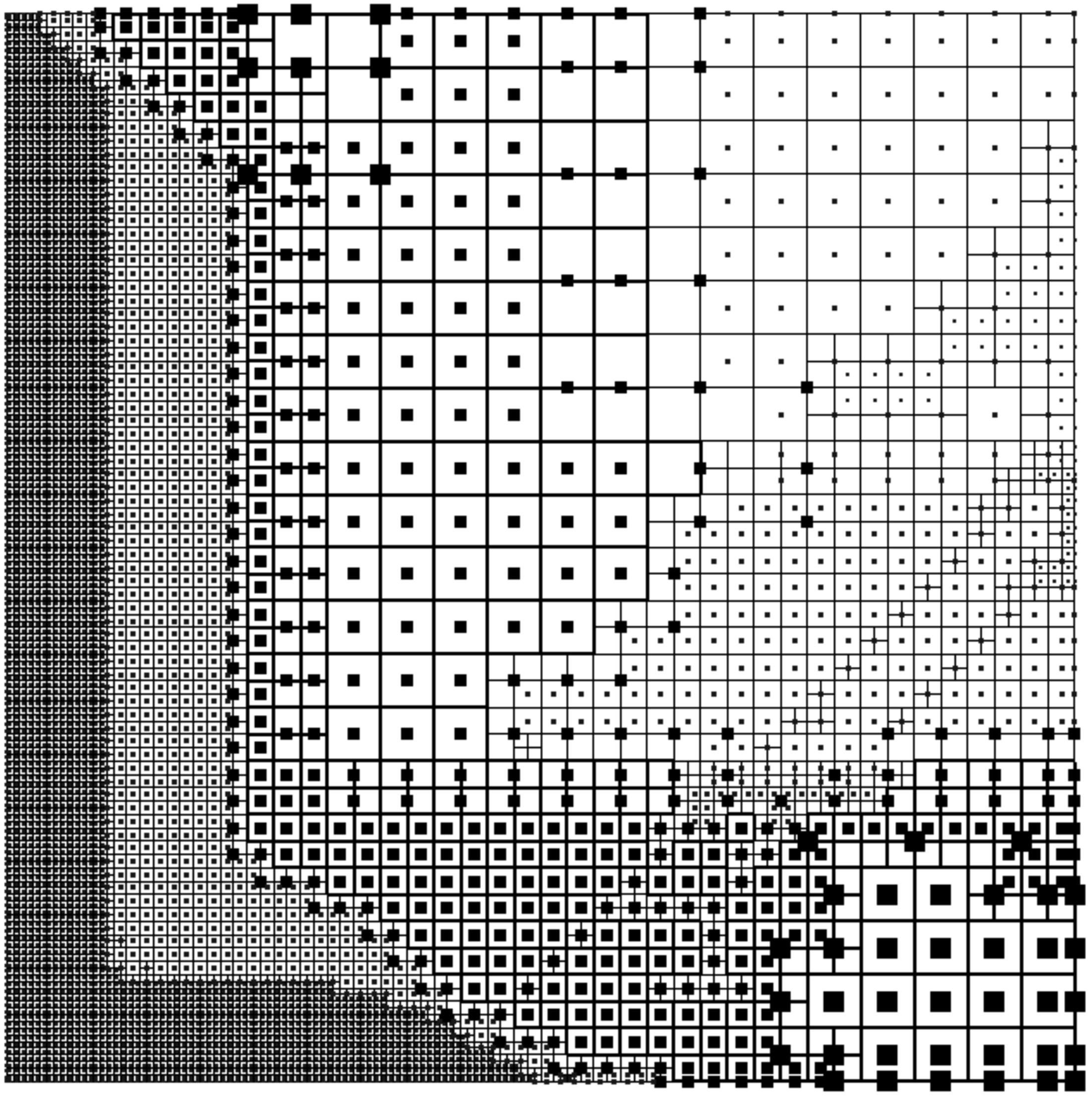}
\end{center} 
&
\begin{center}
\includegraphics[scale=0.23]{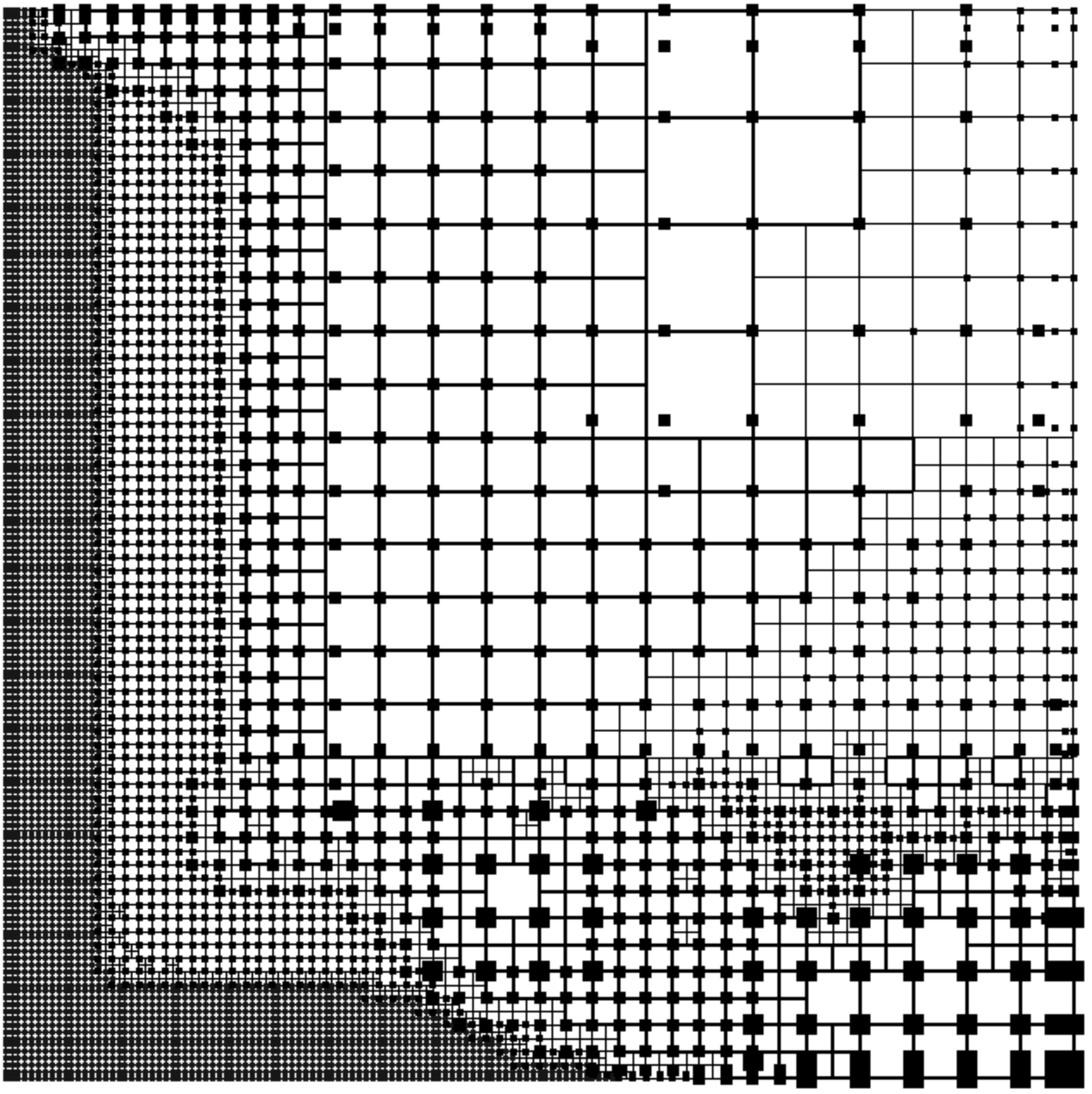}
\end{center}
\\
\vspace{-15pt}
\begin{center}
(c) Greville abscissae for fourth biquadratic mesh
\end{center}
&
\vspace{-15pt}
\begin{center}
(d) Greville abscissae for fourth bicubic mesh
\end{center}
\end{tabularx} }
\caption{The Greville abscissae (black dots) corresponding to the basis functions for
the third and fourth iterations of HASTS local refinement for the 
biquadratic (left column) and bicubic (right column)
advection skew to the mesh problem.  The size of the dot corresponds to the 
basis function level with larger dots corresponding to lower levels. }
\label{fig:skew_static_mesh_funconly_bm}
\end{center}
\end{figure}

\begin{figure}
\begin{center}
{\begin{tabularx}{1\textwidth}{XX}
\begin{center}
\includegraphics[scale=0.23]{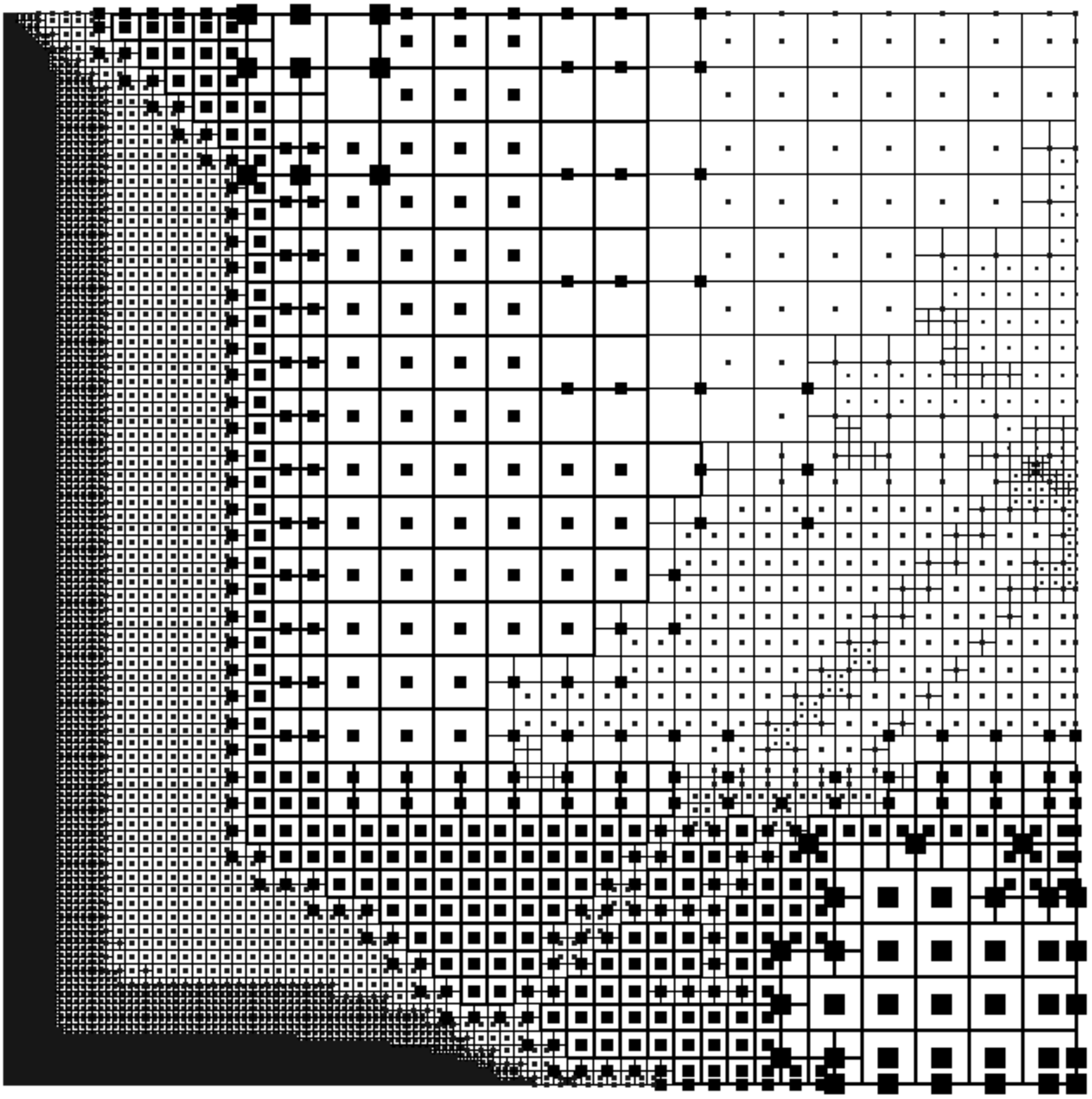}
\end{center} 
&
\begin{center}
\includegraphics[scale=0.23]{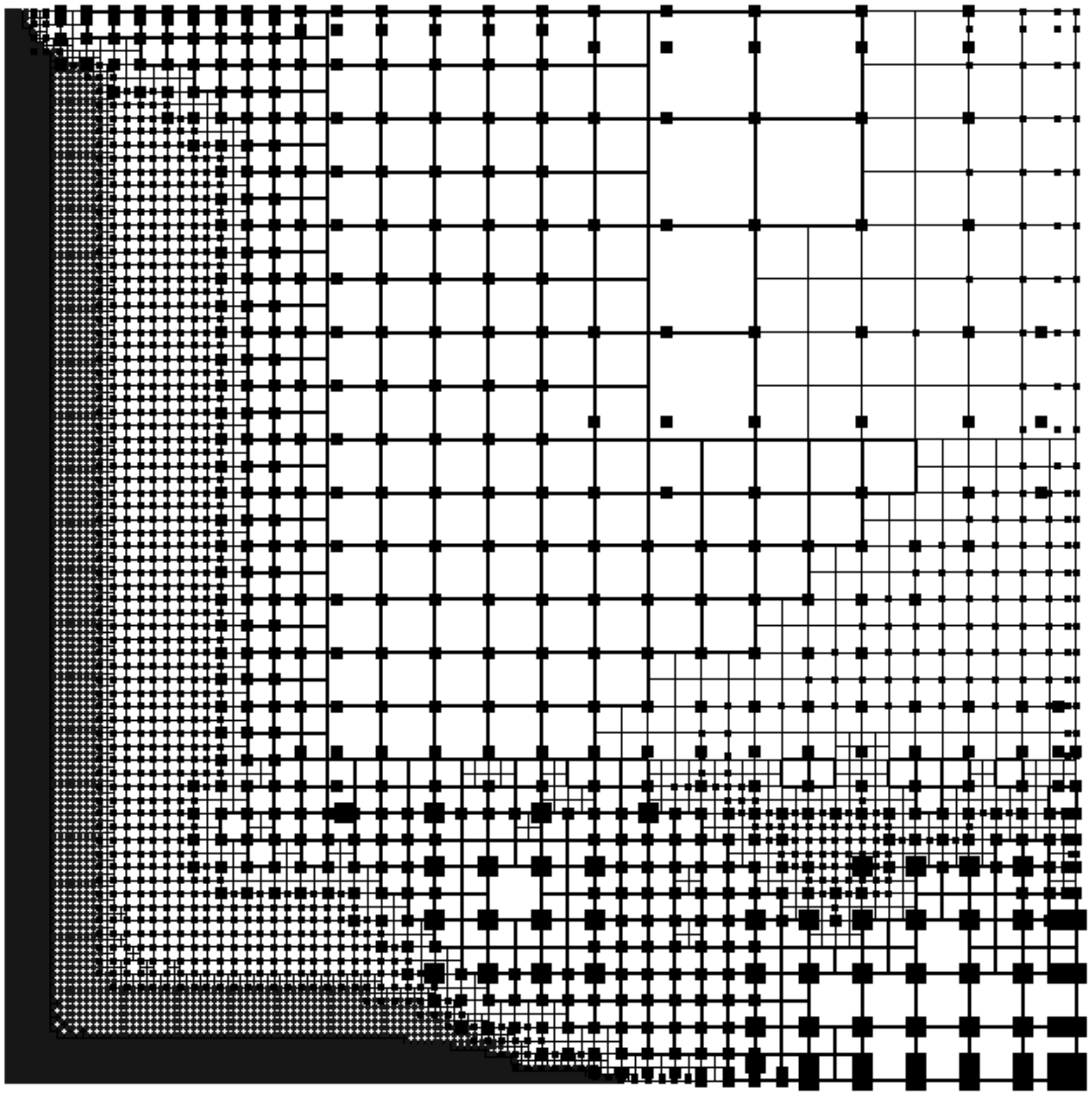}
\end{center}
\\
\vspace{-15pt}
\begin{center}
(a) Greville abscissae for fifth biquadratic mesh\end{center}
&
\vspace{-15pt}
\begin{center}
(b) Greville abscissae for fifth bicubic mesh
\end{center}\\
\begin{center}
\includegraphics[scale=0.23]{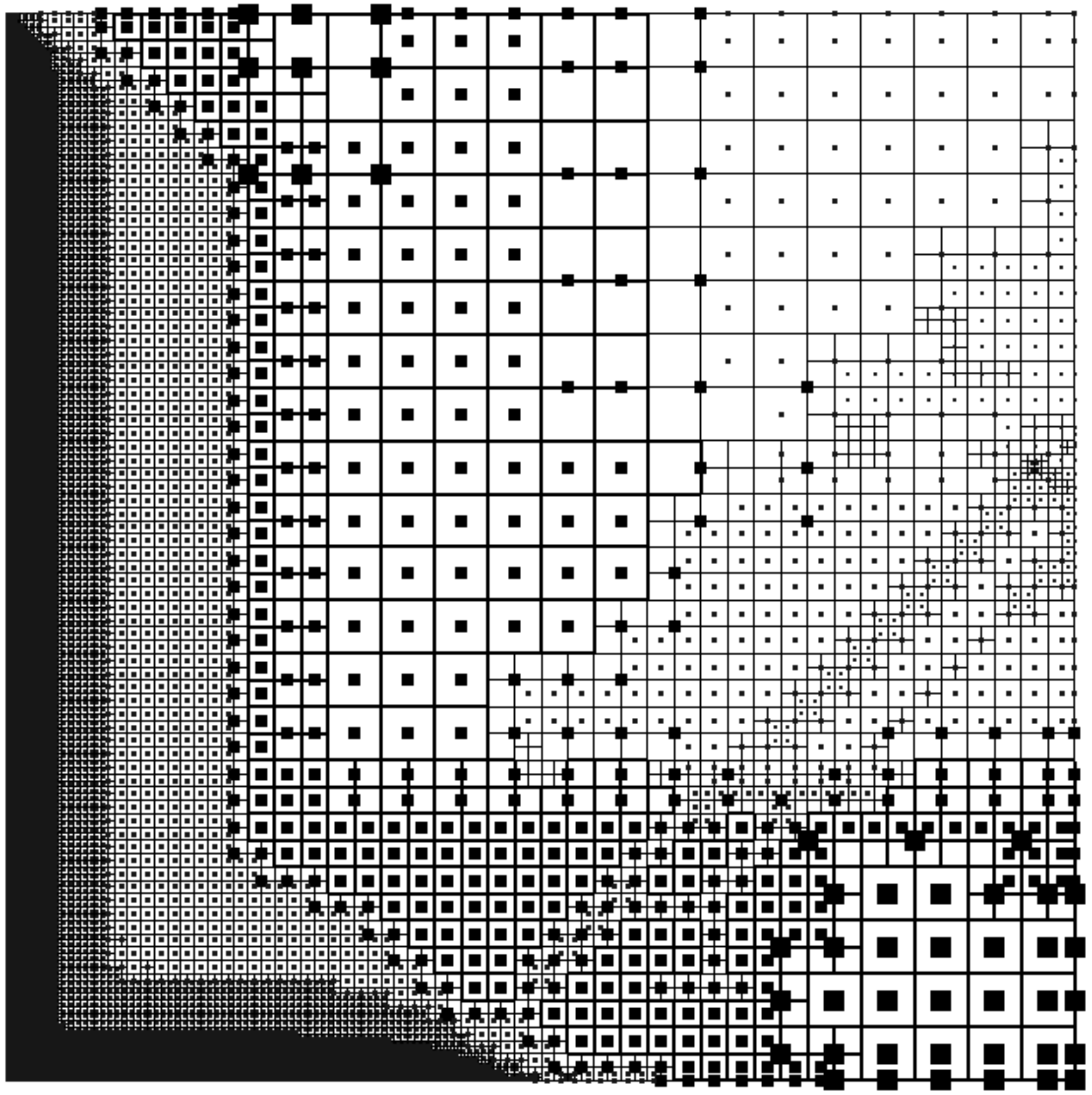}
\end{center} 
&
\begin{center}
\includegraphics[scale=0.23]{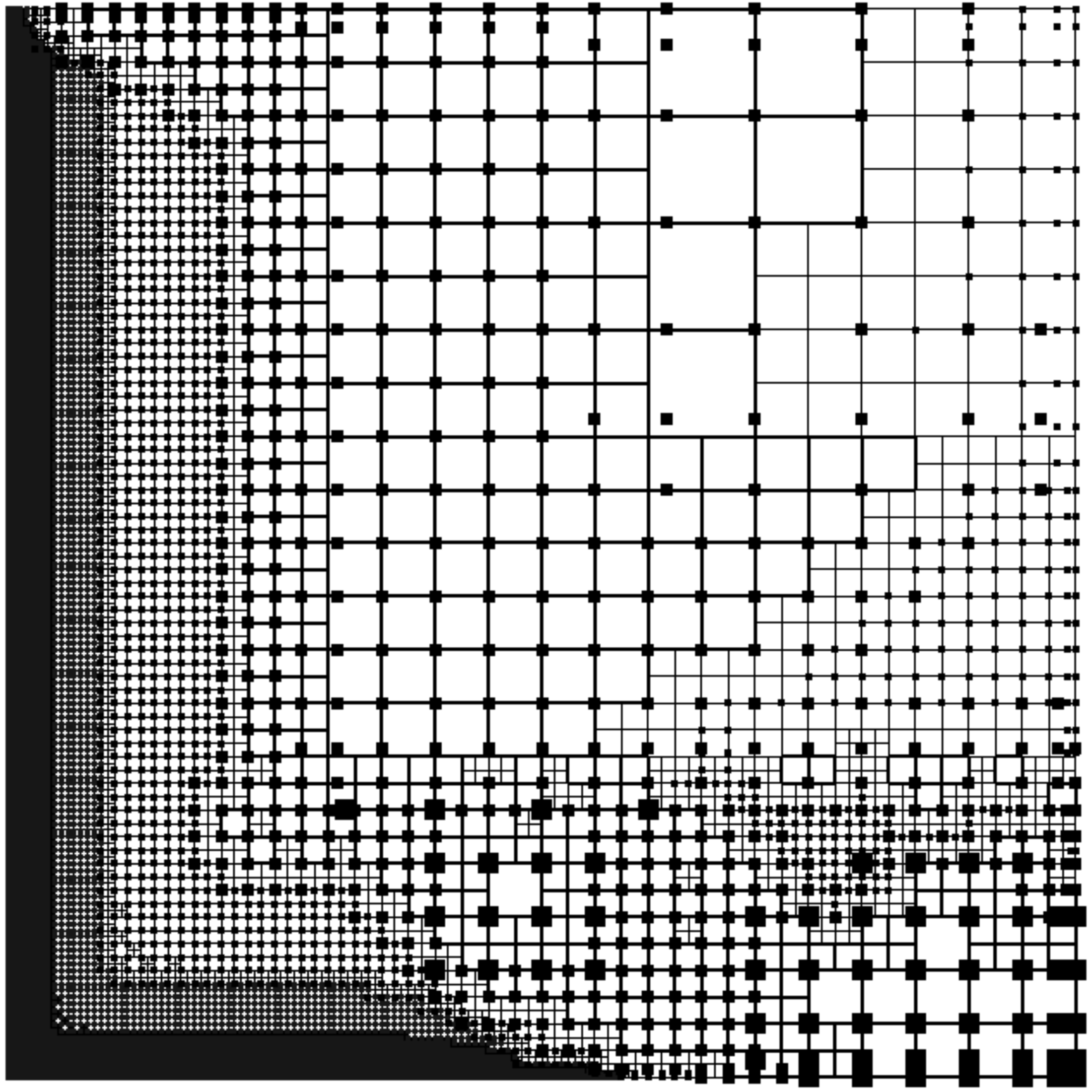}
\end{center}
\\
\vspace{-15pt}
\begin{center}
(c) Greville abscissae for sixth biquadratic mesh
\end{center}
&
\vspace{-15pt}
\begin{center}
(d) Greville abscissae for sixth bicubic mesh
\end{center}
\end{tabularx} }
\caption{The Greville abscissae (black dots) corresponding to the basis functions for
  the fifth and sixth iterations of HASTS local refinement for the 
  biquadratic (left column) and bicubic (right column)
  advection skew to the mesh problem.  The size of the dot corresponds to the 
  basis function with larger dots corresponding to lower levels.}
\label{fig:skew_static_mesh_funconly_bc}
\end{center}
\end{figure}

\begin{figure}
\begin{center}
\begin{tabularx}{0.9\textwidth}{X}
{\begin{tabularx}{0.9\textwidth}{XX}
\begin{center}
\includegraphics[scale=0.25]{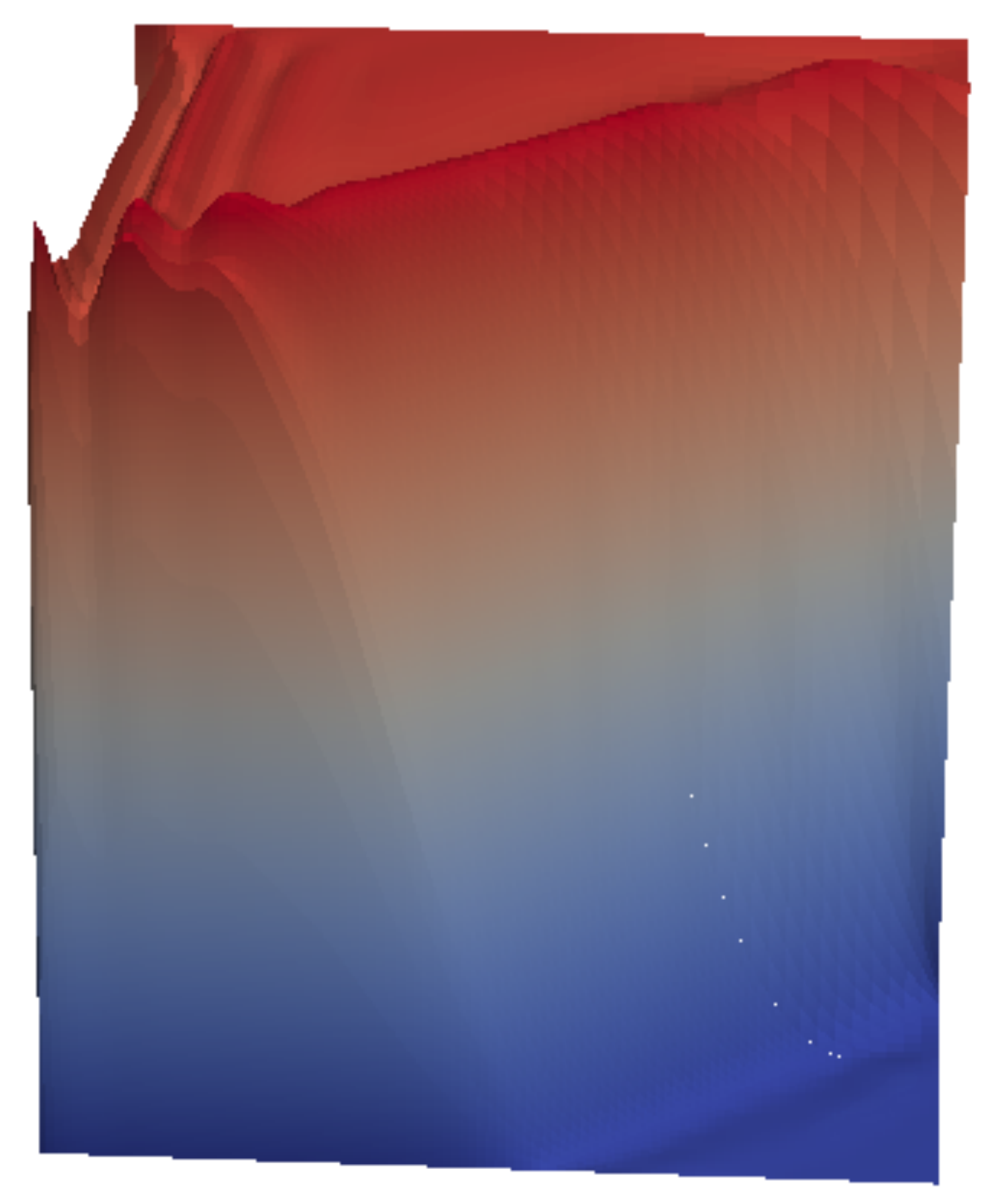}
\end{center} 
&
\begin{center}
\includegraphics[scale=0.25]{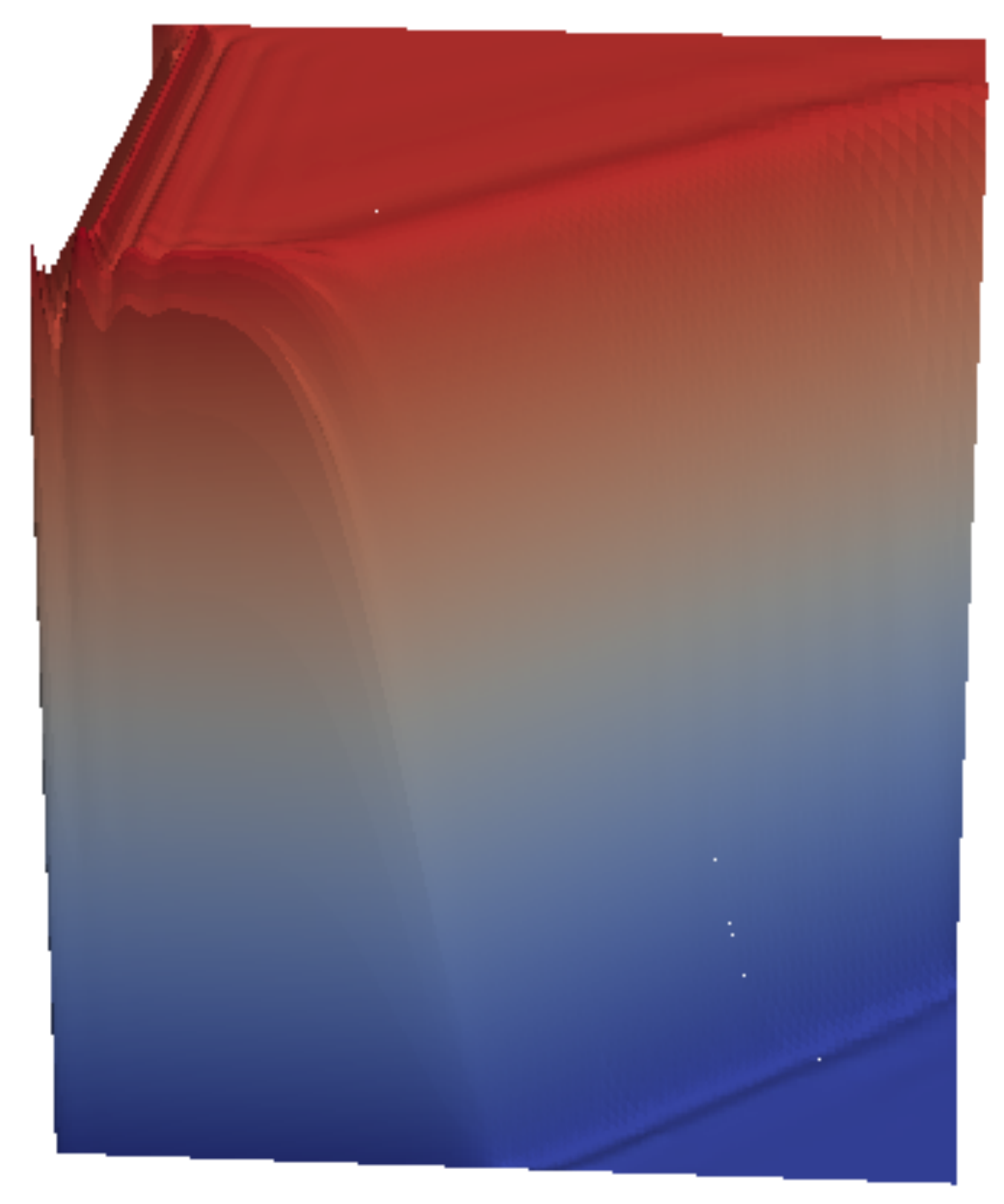}
\end{center}
\\
\vspace{-15pt}
\begin{center}
(a) Solution 1
\end{center}
&
\vspace{-15pt}
\begin{center}
(b) Solution 2
\end{center}\\
\begin{center}
\includegraphics[scale=0.25]{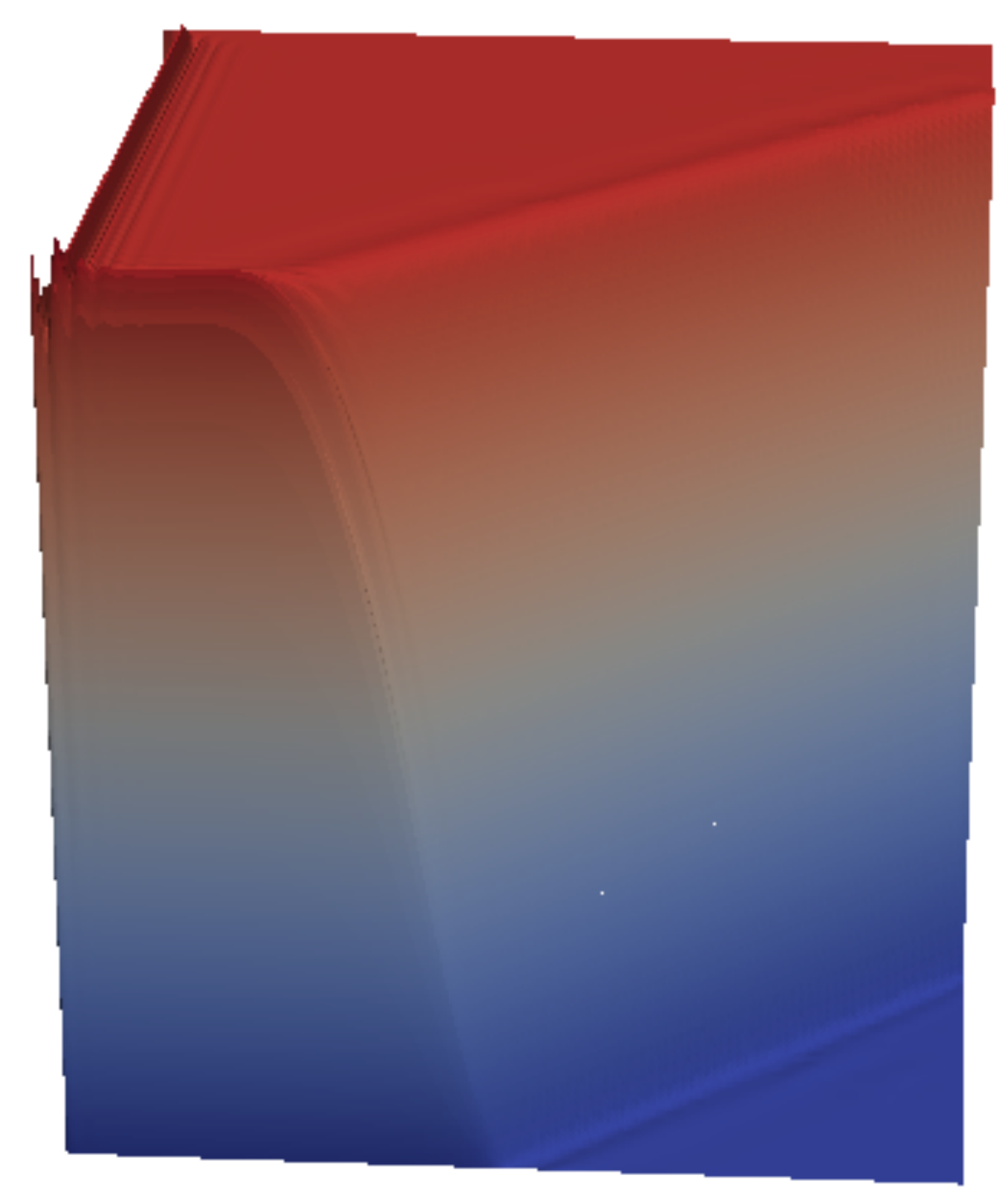}
\end{center} 
&
\begin{center}
\includegraphics[scale=0.25]{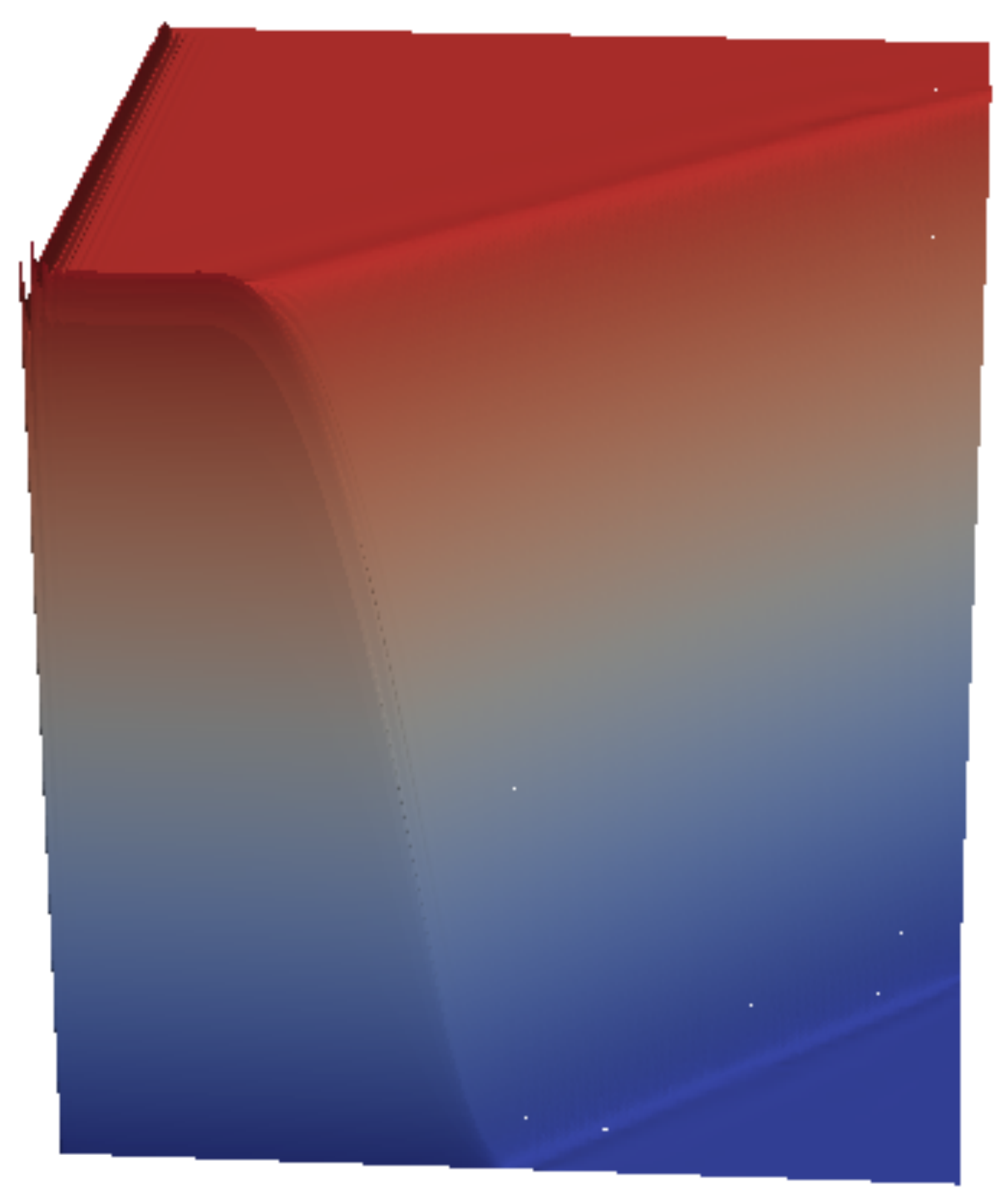}
\end{center}
\\
\vspace{-15pt}
\begin{center}
(c) Solution 3
\end{center}
&
\vspace{-15pt}
\begin{center}
(d) Solution 4
\end{center}
\end{tabularx}}
\begin{center}
\includegraphics[scale=0.25]{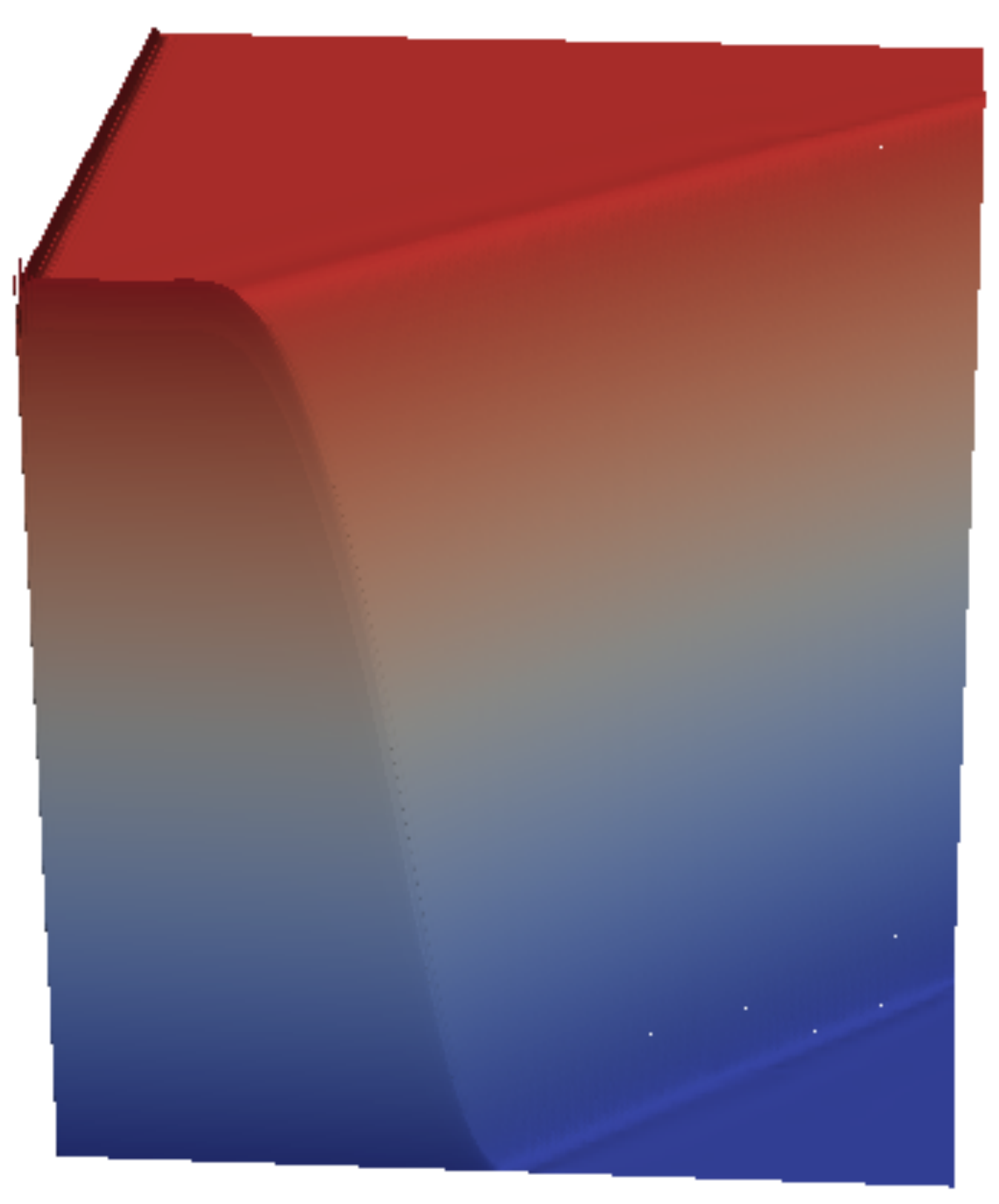}
\end{center} 
\\
\vspace{-15pt}
\begin{center}
(f) Solution 5
\end{center}
\end{tabularx}
\caption{Quadratic solutions to the advection skew to
  the mesh problem. These solutions
  correspond to the B\'{e}zier meshes in the left column
  of~\Crefrange{fig:skew_static_mesh_bq}{fig:skew_static_mesh_bc}.}
\label{fig:skew_static_sol_bq}
\end{center}
\end{figure}

\begin{figure}
\begin{center}
\begin{tabularx}{0.9\textwidth}{X}
{\begin{tabularx}{0.9\textwidth}{XX}
\begin{center}
\includegraphics[scale=0.25]{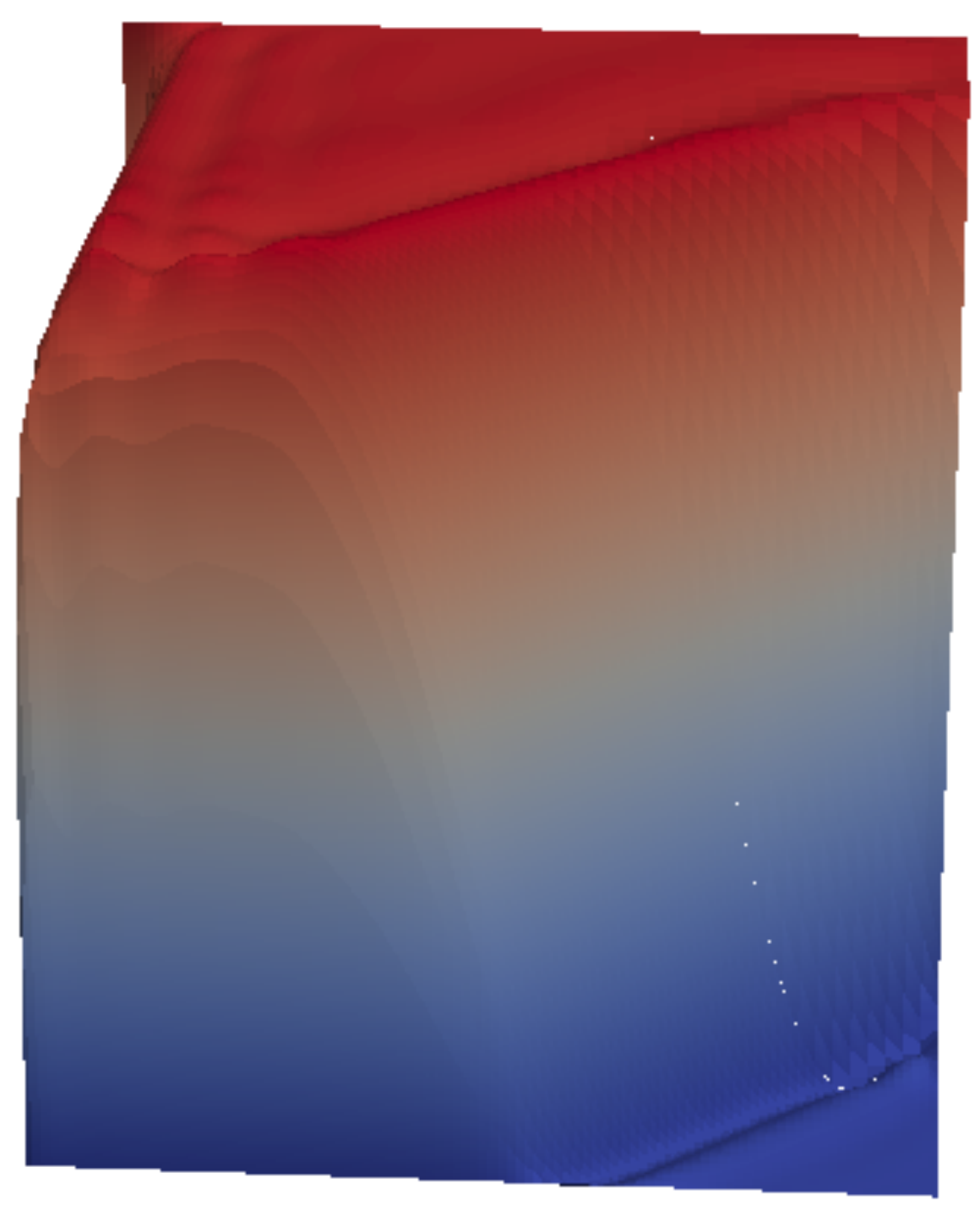}
\end{center} 
&
\begin{center}
\includegraphics[scale=0.25]{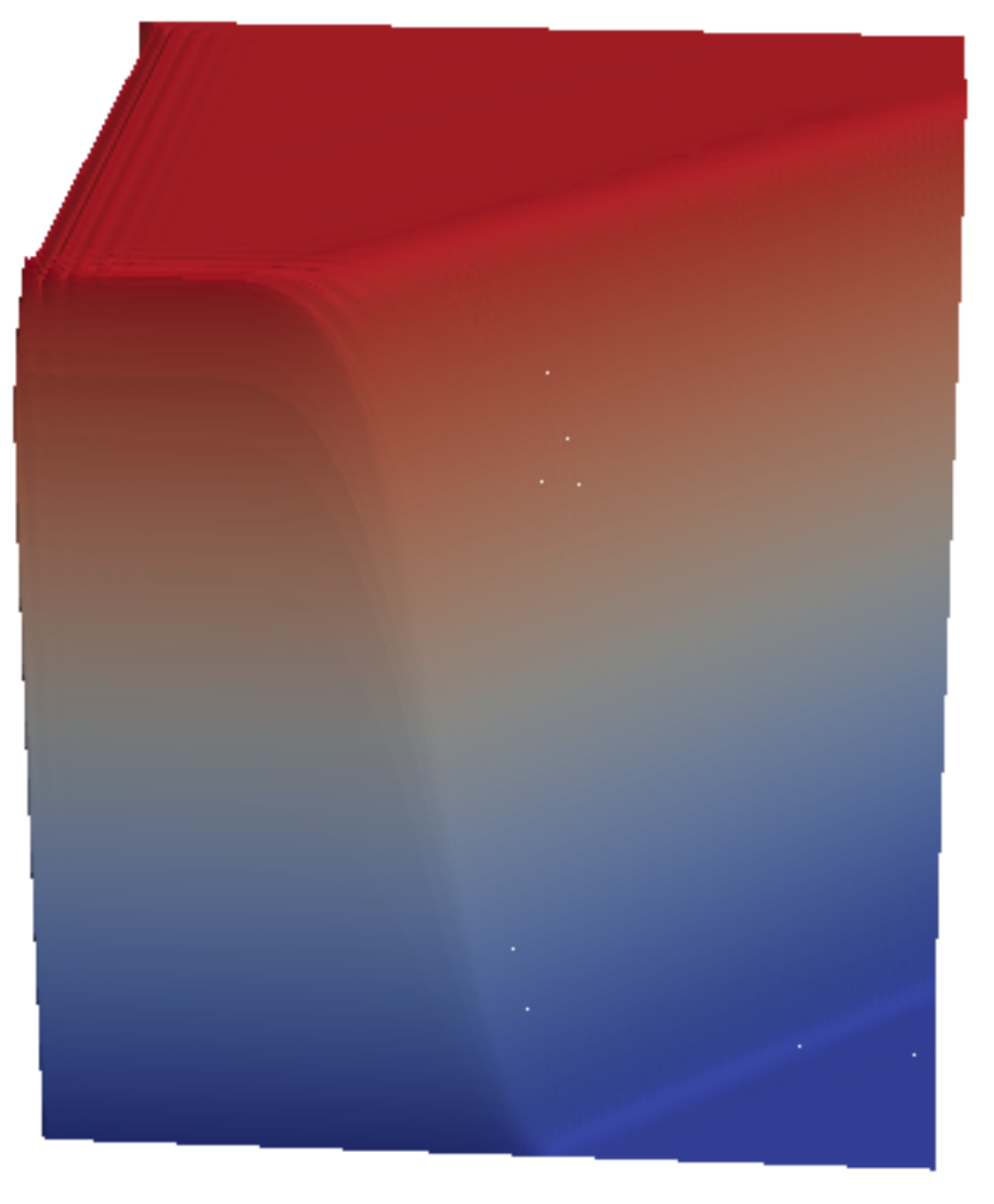}
\end{center}
\\
\vspace{-15pt}
\begin{center}
(a) Solution 1
\end{center}
&
\vspace{-15pt}
\begin{center}
(b) Solution 2
\end{center}\\
\begin{center}
\includegraphics[scale=0.25]{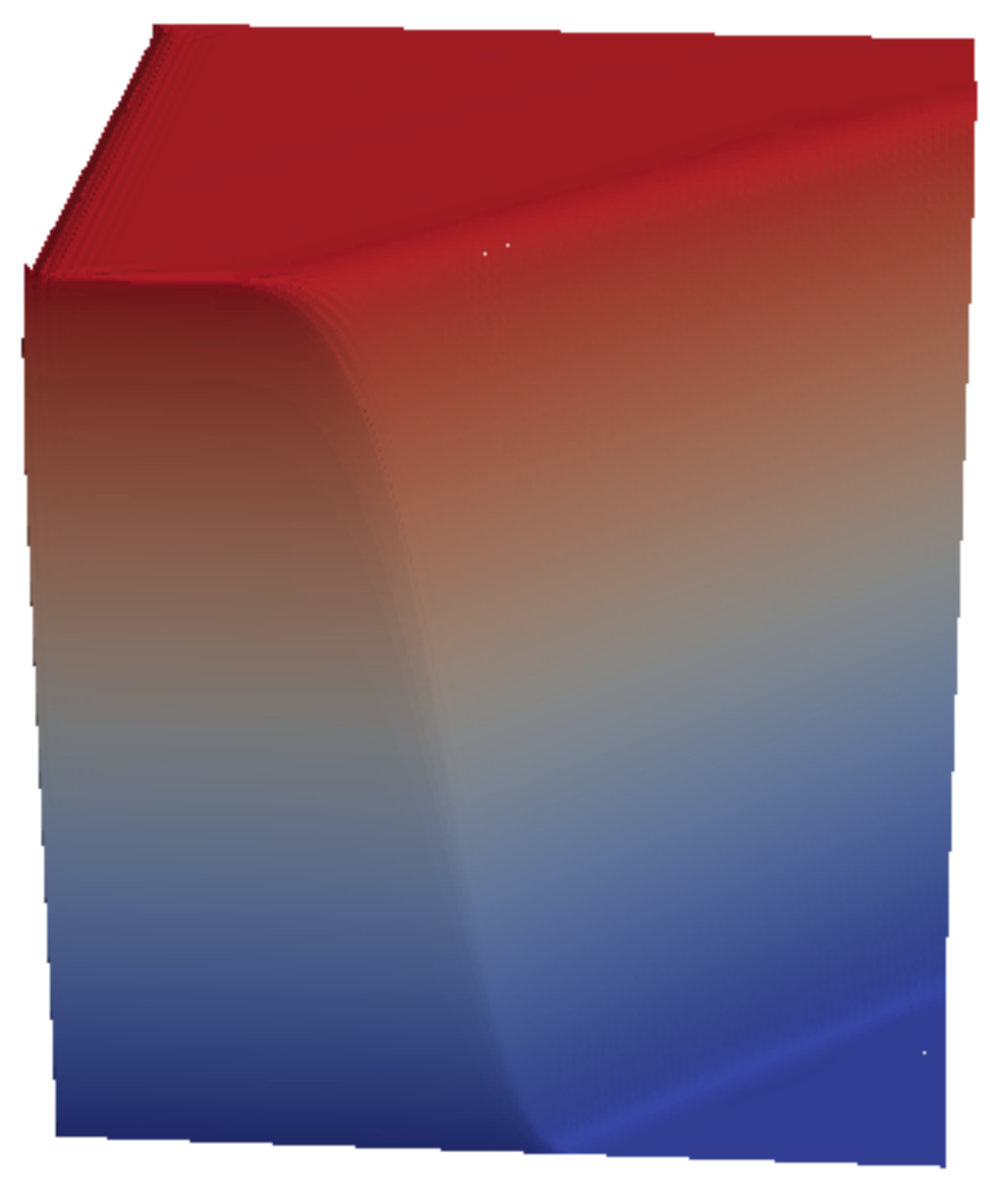}
\end{center} 
&
\begin{center}
\includegraphics[scale=0.25]{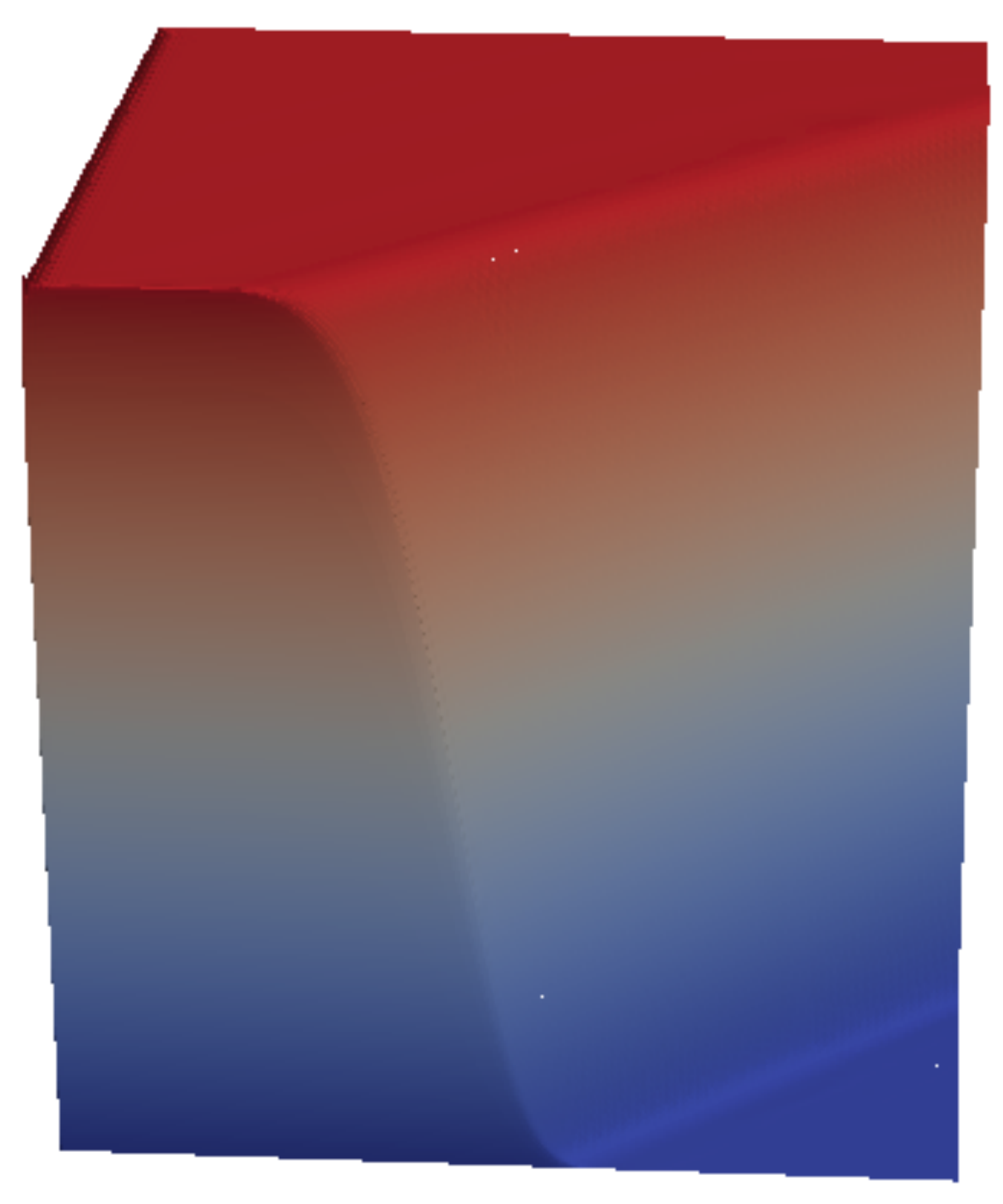}
\end{center}
\\
\vspace{-15pt}
\begin{center}
(c) Solution 3
\end{center}
&
\vspace{-15pt}
\begin{center}
(d) Solution 4
\end{center}
\end{tabularx}}
\begin{center}
\includegraphics[scale=0.25]{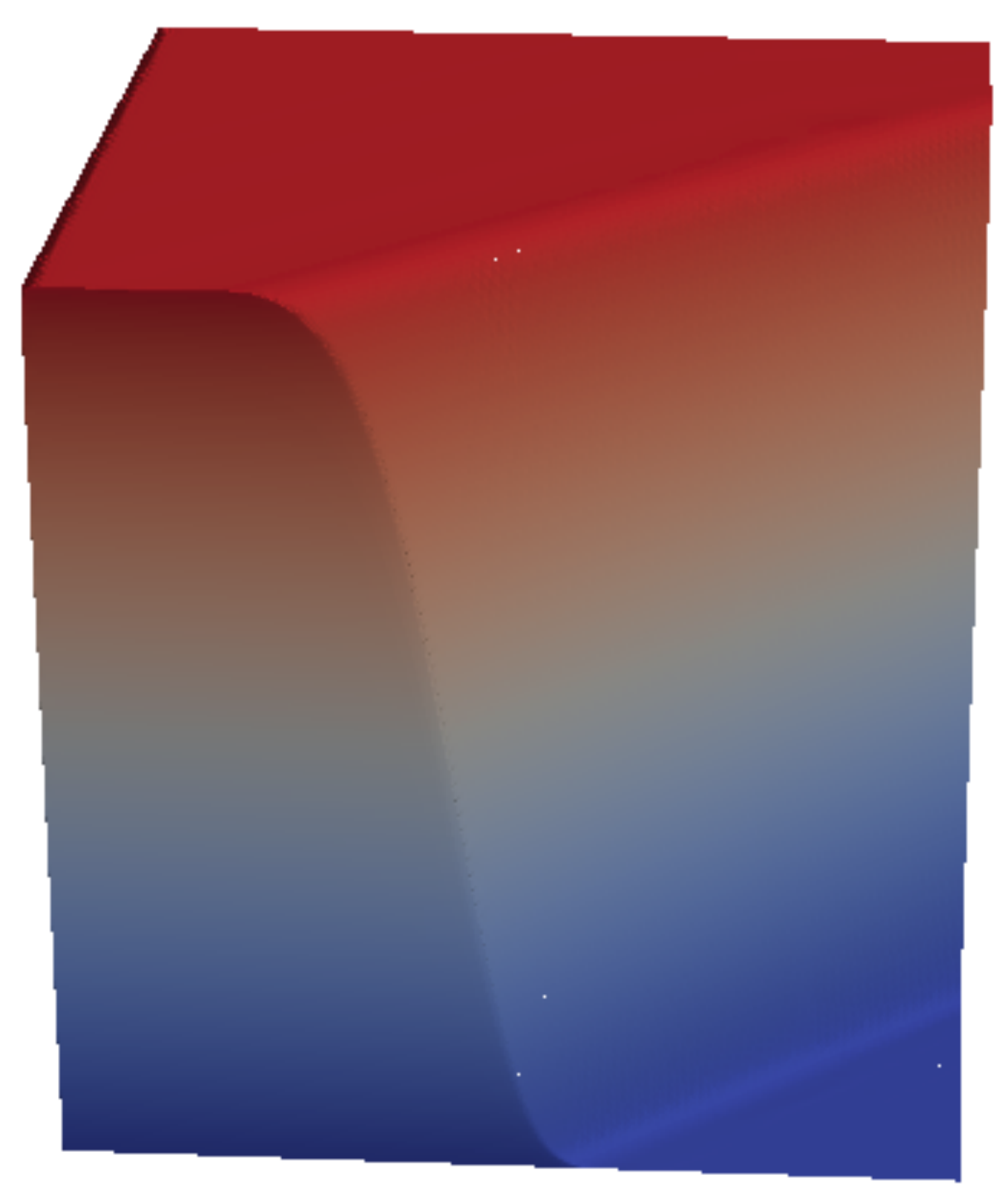}
\end{center} 
\\
\vspace{-15pt}
\begin{center}
(f) Solution 5
\end{center}
\end{tabularx}
\caption{Cubic solutions to the advection skew to
  the mesh problem. These solutions
  correspond to the B\'{e}zier meshes in the right column 
  of~\Crefrange{fig:skew_static_mesh_bq}{fig:skew_static_mesh_bc}.}
\label{fig:skew_static_sol_bc}
\end{center}
\end{figure}

\section{Conclusion}
\label{sec:conclusion}
We have presented hierarchical analysis-suitable T-splines which 
is a superset of both analysis-suitable T-splines and hierarchical B-splines. 
We have also developed the necessary theoretical formulation of 
HASTS including a proof of the local linear independence
of analysis-suitable T-splines.  We presented a simple algorithm for the creation of
nested T-spline spaces and also extended B\'{e}zier extraction to HASTS.
We then demonstrated the potential of HASTS by comparing HASTS to a
local refinement algorithm for T-splines which demonstrated the
improved efficiency and locality of using a hierarchical approach.  
We also demonstrated the use of HASTS in the context of isogeometric
analysis by solving the benchmark static skew advection problem.

In future work we will provide a detailed description of the underlying
algorithms to perform hierarchical refinement in the context of HASTS.  We
will also consider hierarchical $p$ and $k$ refinements of T-splines.  
Finally, we intend to extend the definition of spline forests in~\cite{ScThEv13} to
the T-spline regime. This will allow us to accommodate
smooth interfaces and also interface directly with commercial T-spline products.

\appendix

\bibliographystyle{elsarticle-num}
\bibliography{bibliography}

\end{document}